\documentclass[11pt]{article}
\usepackage[all]{xy}
\usepackage{amsmath, graphicx, tensor, stmaryrd, wasysym, amssymb}
\usepackage{amsfonts}
\usepackage{amssymb}
\usepackage{amscd}
\usepackage{amsthm}
\usepackage{latexsym}
\usepackage{amsbsy}
\usepackage{xcolor,enumerate}
\usepackage[normalem]{ulem}

\def\0D{\Delta^{(0)}}
\def\1D{\Delta^{(1)}}

\newcommand{\pr}{\circledast}

\newcommand{\Cc}{\mathcal{C}}

\newcommand{\Mc}{\mathcal{M}}
\newcommand{\Zc}{\mathcal{Z}}

\newcommand{\Ac}{\mathcal{A}}

\newcommand{\Dc}{\mathcal{D}}

\newtheorem{theorem}{Theorem}[section]
\newtheorem{remark}[theorem]{Remark}
\newtheorem{proposition}[theorem]{Proposition}
\newtheorem{lemma}[theorem]{Lemma}
\newtheorem{corollary}[theorem]{Corollary}

\newtheorem{definition}[theorem]{Definition}

\def\build#1_#2^#3{\mathrel{\mathop{\kern 0pt#1}\limits_{#2}^{#3}}}

\newcommand{\vect}{\text{Vec}}

\newcommand{\bimod}{\mathbb{B}\text{md}}

\newcommand{\la}{\triangleright}
\newcommand{\ra}{\triangleleft}

\newcommand{\Bc}{\mathcal{B}}

\numberwithin{equation}{section}

\def\ot{\otimes}
\def\part{\partial}

\def\ot{\otimes}

\renewcommand{\mod}{\text{-mod}}

\def\build#1_#2^#3{\mathrel{
\mathop{\kern 0pt#1}\limits_{#2}^{#3}}}

\numberwithin{equation}{section}
\newcommand{\comment}[1]{\relax}

\newcommand{\zp}{\mathbb{Z}/p}
\begin{document}

\title{Hopf-cyclic coefficients in the braided setting}
\author{Ilya Shapiro}

\date{}
\maketitle

\begin{abstract}
Considering the monoidal category $\mathcal{C}$ obtained as modules over a Hopf algebra $H$ in a rigid braided category $\mathcal{B}$,  we prove decomposition results for the Hochschild and cyclic homology categories  $HH(\mathcal{C})$ and $HC(\mathcal{C})$ of $\mathcal{C}$.  This is accomplished by defining a notion of a (stable) anti-Yetter-Drinfeld module with coefficients in a (stable) braided module over $\mathcal{B}$.  When the stable braided module is $HH(\mathcal{B})$, we recover $HH(\mathcal{C})$ and $HC(\mathcal{C})$.  The decomposition of $HC(\mathcal{C})$ now follows from that of $HH(\mathcal{B})$.

\end{abstract}

\medskip
{\it Mathematics Subject Classification:}  16E35,  16E40, 16T05, 18D10, 18G90, 18M15, 18N25, 19D55.

\tableofcontents
\section{Introduction}
Cyclic (co)homology for associative algebras was introduced independently by Boris Tsygan and Alain Connes in the 1980s.  The original ideas have since been significantly extended and branched out into many fields.  Our investigations in this paper focus on the equivariant flavour that began with Connes-Moscovici \cite{conmos} and was generalized into Hopf-cyclic cohomology by Hajac-Khalkhali-Rangipour-Sommerh\"{a}user \cite{HKRS1, HKRS2} and Jara-Stefan \cite{JS}.   Roughly speaking, the original theory defines cohomology groups for an associative algebra that play the role of the de Rham cohomology in the noncommutative setting.  The equivariant version considers an algebra with a compatible action of a Hopf algebra.  It turns out that analogous to $D$-modules in the de Rham cohomology, one has coefficients  in the Hopf setting; it is an interesting fact that unlike the de Rham setting, Hopf-cyclic cohomology requires coefficients, i.e., there is no canonical trivial coefficient.  These coefficients are known as stable anti-Yetter-Drinfeld modules, due to their similarity to the usual Yetter-Drinfeld modules, with an important distinction being that stability is an extra condition exclusive to the former.

This paper concerns itself with the case of a Hopf algebra $H$ in a rigid braided category $\Bc$, braiding being required for the notion of a Hopf algebra to exist.  The question of what should be the Hopf-cyclic coefficients, i.e., (stable) anti-Yetter-Drinfeld modules in this setting has first been examined in  \cite{masoud}.  The work therein was recently refined and clarified in \cite{balanced}; the latter should also be compared to \cite{majidribbon}. The approach consists of imitating the classical definition, i.e., that of modules and comodules over $H$ with a compatibility condition between the two structures, plus stability. These investigations can be considered as the anti-center counterpart to those in \cite{relcenter} that concern the center, see \ref{cor:aydbs}.  

Our approach is from a different perspective, motivated by the example in Section \ref{sec:blah}.  Namely, a usual Hopf algebra $T$ is simplified when replaced by a Hopf algebra $H$ in a braided $\Bc$.  If one wants to understand the usual (stable) anti-Yetter-Drinfeld modules for $T$, could one do so in terms of $H$?  To answer this question, let $\Cc=H_\Bc\mod$ be the category of $H$-modules; it is monoidal, see Section \ref{sec:b1}.    The classical definition of  (stable) anti-Yetter-Drinfeld modules has been generalized, and in particular, it is now possible to talk about them for a general monoidal category $\Cc$. Denote anti-Yetter-Drinfeld modules and their stable variants by $HH(\Cc)$  and $HC(\Cc)$  respectively, see Section \ref{sec:b3}.  We now need to describe $HC(\Cc)$ as modules and comodules over $H$.  Surprisingly, the answer to our question is different from the existing literature and in particular  it strictly subsumes \cite{balanced}.  

Namely, in Definition \ref{def:saYD}, we describe (stable) anti-Yetter-Drinfeld modules with coefficients in a (stable) braided module $\Mc$ over $\Bc$, see Definition \ref{def:stablebraided}.  Theorems \ref{thm:aYDHH} and \ref{thm:localization} then complete the description of $HC(\Cc)$ in terms of modules and comodules over $H$ in the \emph{stable braided module} $HH(\Bc)$.  Note that    $HH(\Bc)$ itself consists of anti-Yetter-Drinfeld modules for $\Bc$.  As a stable   braided module $HH(\Bc)$, in certain cases such as in Section \ref{sec:decompo}, admits a decomposition which in turn decomposes $HC(\Cc)$, see Corollary \ref{cor:hcdecomp}.   If $\Bc$ is balanced, as is the case considered in \cite{balanced}, then  $HH(\Bc)$ admits a summand, as a stable braided module, isomorphic to $\Bc$ itself.  We can then summarize \cite{balanced} and \cite{masoud} by stating that the definitions given there recover some of the piece of $HC(\Cc)$ that corresponds only to $\Bc$ in the full $HH(\Bc)$, see Lemma \ref{lem:modpair}.  Note that such a piece need not exist, in fact it exists if and only if $\Bc$ is balanced, see Lemma \ref{class:lem}.   The choice of the twist in the balancing  affects the  category of stable anti-Yetter-Drinfeld modules that one gets.

The notion of a braided module is equivalent to a module over $HH(\Bc)$ where the latter is given the ``cylinder stacking" product from factorization homology \cite{quantum}.  This point of view is surveyed in \cite{survey} and will be ignored in this paper as we are mainly focused on the module/comodule description of $HC(\Cc)$ and the explicit calculations of Section \ref{sec:blah}.

\subsection{Conventions} We fix an algebraically closed ground field $k$, of characteristic $0$; $\vect$ denotes the category of finite dimensional $k$-vector spaces.  Our monoidal categories are linear over $k$.  All algebras $A$ in monoidal categories $\Cc$ are assumed to be unital associative; we say that $A\in Alg(\Cc)$.  We let $A_\Cc\mod$ stand for left $A$-modules in $\Cc$ which is a right $\Cc$-module category.  Furthermore, if $\Mc$ is a $\Cc$-module category then $A_\Mc\mod$ denotes the category of $A$-modules in $\Mc$.  Finally, if $\Ac$ is a monad on a category $\Cc_0$ then $\Ac_{\Cc_0}\mod$ denotes the category of $\Ac$-modules in $\Cc_0$.

\subsection{Organization of the paper}
After the preliminaries of Section \ref{sec:prelim} we  introduce stable braided modules over a braided category in Section \ref{sec:stablebraidedmodules}.  They are an essential ingredient for defining the diagram over the Connes' cyclic category $\Lambda$ that yields the notion of (stable) anti-Yetter-Drinfeld modules in the braided setting.  Braided modules are in fact our third attempt at the right concept, with the first  being centered around the notion of a twist in a braided category \eqref{eq:twist1}, the second being the $L$ of Definition \ref{def:ocross}.  The main examples of stable braided modules are $HH(\Bc)$ and $\Bc_\varsigma$ ($\varsigma$ is an anti-twist on $\Bc$), with the decomposition of Section \ref{sec:decompo}: $$HH(\Bc)=\bigoplus_\varsigma\Bc_\varsigma$$ relating the two \emph{in a special case}.

Section \ref{sec:stablemodsandcyclic} justifies the preceding definitions by demonstrating that $$\Cc_n=H^{\ot^\tau(n+1)}_\Mc\mod$$ is indeed a diagram of categories over $\Lambda$.  This yields the monadic description of the limit which in turn forces (abridged version of Definition \ref{def:saYD}):
\begin{definition}\label{def1}
Suppose that $\Mc$ is a braided $\Bc$-module.  Let  $aYD^H_\Mc$ denote the category of $M\in\Mc$ such that $$M\in H_\Mc\mod\quad\text{and}\quad M\in {}^*H_\Mc\mod.$$   And the two actions are compatible as follows: \begin{equation*}
\includegraphics[height=1in]{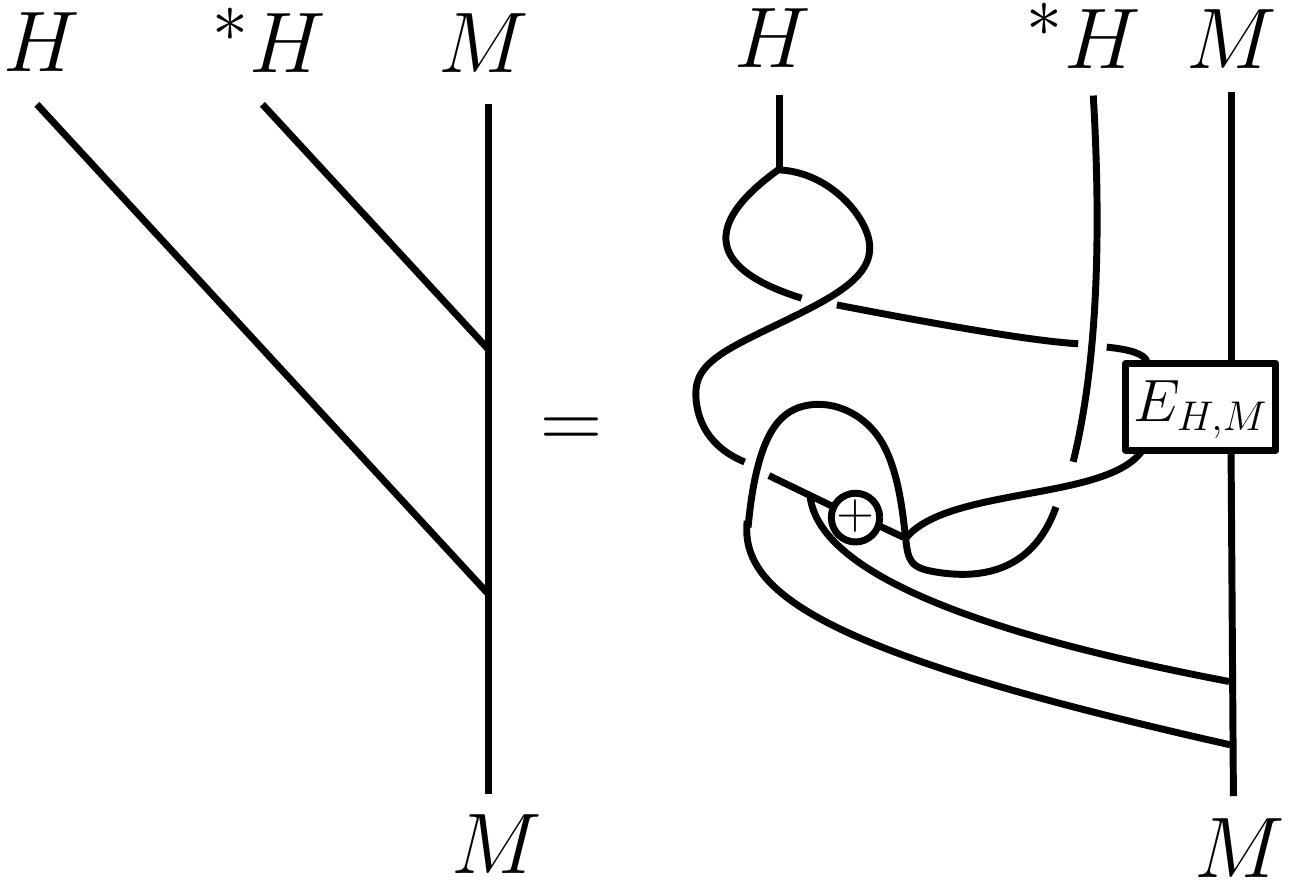}
\end{equation*}   
\end{definition} Here $E_{H,M}$ is the braided module structure of $\Mc$.  Section \ref{sec:frommtocoh} explains how to go from an $M$ above to a cohomology theory for algebras in $H_\Bc\mod$.  This section is useful to us here because it helps prove the main result of the paper (a version of Theorem \ref{thm:localization}): 
\begin{theorem}\label{thm1}
Let $H\in\Bc$ be a Hopf algebra, let $\Cc=H_\Bc\mod$, then $$aYD^H_{HH(\Bc)}\simeq HH(\Cc)$$ and the isomorphism identifies $\varsigma^H$ on LHS with  $\varsigma$ on RHS.
\end{theorem} This  describes $HC(\Cc)$ as modules/comodules \emph{in $HH(\Bc)$} over $H$.  See also Corollary \ref{cor:aydbs} that analogously describes $aYD^H_{\Bc_\varsigma}$. Note that, after the fact, one could use Definition \ref{def1} specifically for $\Mc=HH(\Bc)$, and then prove Theorem \ref{thm1} directly.  This avoids any discussion of diagrams of categories, limits, and monads.  The trade-off is that the proofs would increase significantly in difficulty and the motivation behind the definition would be shrouded in mystery.  Namely, it would  sever the connection to the conceptual framework that produced these notions naturally and connects them to each other and the wider mathematical setting.

Section \ref{sec:blah} applies the machinery developed in this text to the Taft algebra $T_p(\xi)$.  The monoidal category  is $$\Cc=T_p(\xi)\mod=H_\Bc\mod$$ for an appropriate $H$ in an appropriate $\Bc$.  The category of the classical stable anti-Yetter-Drinfeld modules for $T_p(\xi)$ thus exhibits a decomposition into packets; these are explicitly identified via calculations.

Finally,  the appendix contains some material that is likely to be general knowledge and is mostly included to fix notation.

\medskip

\noindent\textbf{Acknowledgements}: This research was supported in part by an NSERC Discovery Grant.

\section{Some preliminaries}\label{sec:prelim}
Here we collect some background material that facilitates the reading of this text.  We aim to give references, but avoid precise definitions, rather we hope to impart a working understanding of the concepts sufficient for following the arguments in the paper.

\subsection{Rigid braided categories and Hopf algebras}\label{sec:b1}

Roughly speaking, a monoidal category $\Cc$ \cite{braidedjoyal} is a category equipped with a bifunctor  $\ot:\Cc\times\Cc\to\Cc$ written as $(X,Y)\mapsto X\ot Y$ and an associativity constraint $$\alpha_{X,Y,Z}:X\ot (Y\ot Z)\simeq (X\ot Y)\ot Z$$ that satisfies a coherence (pentagon) axiom.  There is also a monoidal unit $1\in\Cc$ with natural requirements.  The main point of the pentagon axiom is to ensure the following.  Take a finite number of objects in $\Cc$ and use the $\ot$ iteratively to construct a single object.  This object depends on the particular choice of bracketing, but any two choices are isomorphic using successive  applications  of the associator $\alpha$.   Now the isomorphism could in principle depend on the particular sequence of the applications of the associator, but it doesn't.  A left or right $\Cc$-module category $\Mc$ is defined similarly, i.e., replace one of the $\Cc$'s by an $\Mc$ in the definitions.  This and other concepts can be found in \cite{egno}.

A monoidal category is rigid if any object $X$ has a right ${}^*X$ and a left dual $X^*$.  Note that in general $X^{**}\ncong X$, we use $X^\#$ to denote the former.  We have evaluation and coevaluation maps: $X^*\ot X\to 1$, $X\ot {}^*X\to 1$, $1\to X\ot X^*$, and $1\to {}^*X\ot X$ such  that the functor pairs $(-\ot X,-\ot X^*)$ and $(X\ot -,{}^*X\ot -)$ are adjoint pairs; this is best understood in terms of straightening strings (see below for graphical calculus).  A category is pivotal if it is equipped with a pivot: a monoidal isomorphism $\rho\in Nat^\ot(Id, (-)^\#)$, i.e.,  $\rho_X: X\to X^\#$ with $$\rho_{XY}=\rho_X\ot\rho_Y.$$

Recall that a monoidal category $\Bc$ is braided \cite{braidedjoyal} if there is a braiding isomorphism $$\tau_{X,Y}:X\ot Y\simeq Y\ot X.$$  The braiding is to satisfy certain coherence conditions that ensure that braid diagrams can be used to perform calculations; thus, this structure gives rise to braid group actions on powers $X^{\ot n}$.  The braided category $\Bc$ is symmetric if $$\tau^2_{X,Y}=\tau_{Y,X}\tau_{X,Y}=Id_{X\ot Y};$$ the braid action then factors through the symmetric group.  On the opposite side of the spectrum, we say that $\Bc$ is non-degenerate if $\tau^2_{X,Y}=Id_{X\ot Y}$  for all $Y$, implies that $X=1$.

A braided category is balanced if it is equipped with a twist $\theta\in Nat(Id, Id)$, i.e., an isomorphism $\theta_X:X\to X$ such that \begin{equation}\label{eq:twist1}\theta_{XY}=\tau^2_{X,Y}\theta_X\ot\theta_Y.\end{equation} Thus, a balanced category, though it is not necessarily symmetric, is so up to a ``boundary".  A balanced category is a ribbon category if ${}^*\theta_{X^*}=\theta_X$ for all $X$.

A bialgebra $A$ in $\Bc$ \cite{majidbraidhopf} is an algebra and a coalgebra in a compatible manner; to express the compatibility we need the braiding $\tau$.  More precisely, $(A,m,u,\Delta, \epsilon)$ is a bialgebra if $(m,u)$ is an algebra structure, $(\Delta,\epsilon)$ a coalgebra structure and $$\Delta m=(m\ot m)(Id_A\ot \tau_{A,A}\ot Id_A)(\Delta\ot \Delta),$$ $\Delta u=u\ot u$, $\epsilon m=\epsilon\ot\epsilon$, and $\epsilon u=Id_1$.  Namely,  recall (or see Definition \ref{algprod}) that we can form the algebra $A\ot^\tau A$ and we require that $\Delta: A\to A\ot^\tau A$ and $\epsilon: A\to 1$ are algebra maps.

In particular, given $V,W\in A_\Bc\mod$, we can form $V\ot^\tau W\in (A\ot^\tau A)_\Bc\mod$, and then use $\Delta^*:(A\ot^\tau A)_\Bc\mod\to A_\Bc\mod$ to define a monoidal structure on $A_\Bc\mod$: \begin{equation}\label{productinh}V\pr W:=\Delta^*(V\ot^\tau W).\end{equation}

A Hopf algebra $H$ in $\Bc$ is a bialgebra with the additional property that it has an (invertible, for us) antipode $S: H\to H$ that satisfies: $$m(S\ot Id_H)\Delta=u\epsilon=m(Id_H\ot S)\Delta.$$  It follows that $S$ is both anti-multiplicative and anti-comultiplicative, in particular: $$Sm=m(S\ot S)\tau_{H,H},\quad\Delta S=\tau_{H,H}(S\ot S)\Delta.$$  Such an $S$ is unique if it exists and $S$ ensures the existence of duals for $H$-modules.

In this paper we will consider a Hopf algebra $H$ in a rigid braided category $\Bc$.  It is immediate by the above that $$\Cc=H_\Bc\text{-mod},$$ the category of  modules over $H$ in $\Bc$, is a rigid monoidal category.  That is our main object of interest.

\subsection{Categories $HH(\Cc)$, $HC(\Cc)$ and Hopf cyclic coefficients}\label{sec:b3}
Let $\Cc$ be a monoidal category.  Denote the product by \begin{equation}\label{tensor1}\Delta^*:\Cc\boxtimes\Cc\to\Cc\end{equation} and the unit by \begin{equation}\label{unit}\epsilon^*:\vect\to\Cc.\end{equation}  Categorifying the construction of a simplicial or cyclic object \cite{loday} associated to an algebra, we have a diagram, see Section \ref{lim:sec}, over Connes' cyclic category $\Lambda$, with $$\Cc_n=\Cc^{\boxtimes(n+1)}.$$  It is constructed from three key structures:  \eqref{tensor1}, \eqref{unit}, and $\sigma:\Cc^{\boxtimes 2}\to\Cc^{\boxtimes 2}$ which flips the two copies.  We then have $$HH(\Cc)=\varinjlim_{\Delta^{op}}\Cc_\bullet\quad\text{and}\quad HC(\Cc)=\varinjlim_{\Lambda^{op}}\Cc_\bullet$$ where $\Delta$ is the simplex category.

The  Hochschild homology category $HH(\Cc)$ has a much simpler description that is essentially a copy of  the classical center construction $\Zc(\Cc)$.  Namely, $HH(\Cc)$ consists of $M\in\Cc$ equipped with the structure of isomorphisms: $$\tau^\bullet_{M,X}:M\ot X\to {}^\#X\ot M$$ satisfying $\tau^\bullet_{M,XY}=\tau^\bullet_{M,Y}\tau^\bullet_{M,X}$.  Note that, unlike $\Zc(\Cc)$, there is also a $\varsigma\in Aut(Id_{HH(\Cc)})$: $$\varsigma_M=(Id_M\ot ev_{M^\#,M^*})\tau^{\bullet -1}_{M,M}(Id_M\ot coev_{M,M^*}),$$ also see \eqref{centerhhbraid}.  The cyclic homology category $HC(\Cc)$ is the full subcategory of $HH(\Cc)$ consisting of objects $M$ with $\varsigma_M=Id_M$;  we say that $HC(\Cc)=HH(\Cc)^\varsigma$.  

In \cite{chern}  we provide a description of $HH(\Cc)$, that is applicable to our case here, in terms of a monad on $\Cc$.  Namely, $$\Ac(M)=\Delta^*\sigma\Delta_*M$$ where $\Delta_*$ is the right adjoint to the product \eqref{tensor1}.  It is shown that the limit can be identified with $\Ac_\Cc\mod$.  The category $HC(\Cc)$ is obtained using the action of $\mathfrak{z}\in\Zc(\Ac)$ that produces $\varsigma$.  The monadic approach is central to this paper.

Let $H$ be a usual Hopf algebra, i.e., in vector spaces.  Let $\Cc=H\mod$, so that $$\Cc_n=H^{\ot (n+1)}\mod$$ and the three structures are literally $\Delta^*$ (where $\Delta:H\to H^{\ot 2}$ is the coproduct), $\epsilon^*$ (where $\epsilon: H\to k$ is the counit), and $\sigma=\sigma^*$ (where $\sigma(x\ot y)=y\ot x$ for $x,y\in H$).  Having defined the diagram over $\Lambda$ we obtain that $HH(\Cc)$ is exactly the usual anti-Yetter-Drinfeld modules.  The path from the limit to the module/comodule description lies through the monad $\Ac$ which in this case is isomorphic to $Hom_k(H,-)$.

The main difficulty, as far as the limit description is concerned, in passing from $H$ in $\vect$ to $H$ in a rigid braided $\Bc$ is the absence of $\sigma$. Keep in mind that $HH(H_\Bc\mod)$ is still perfectly well defined according to the above discussion.  The natural thing to try is $\tau_{H,H}:H^{\ot 2}\to H^{\ot 2}$ but that is not an algebra map in $\Bc$ (unless $\Bc$ is symmetric) so that in general $\Cc_n=H^{\ot (n+1)}_\Bc\mod$ is not even a diagram over $\Delta$.   This problem is rectified by replacing $\Bc$ with a stable braided module $\Mc$, see Definition \ref{def:stablebraided}.

\subsection{Graphical calculus: string diagrams}\label{sec:b4}
Most of the computations in this paper are done using string diagrams.  These depict the compositions of various structures available in rigid braided categories and drastically simplify their manipulations. Our string diagrams are read top to bottom.  The following is a compact summary of the notation:\begin{equation*}
\includegraphics[height=2in]{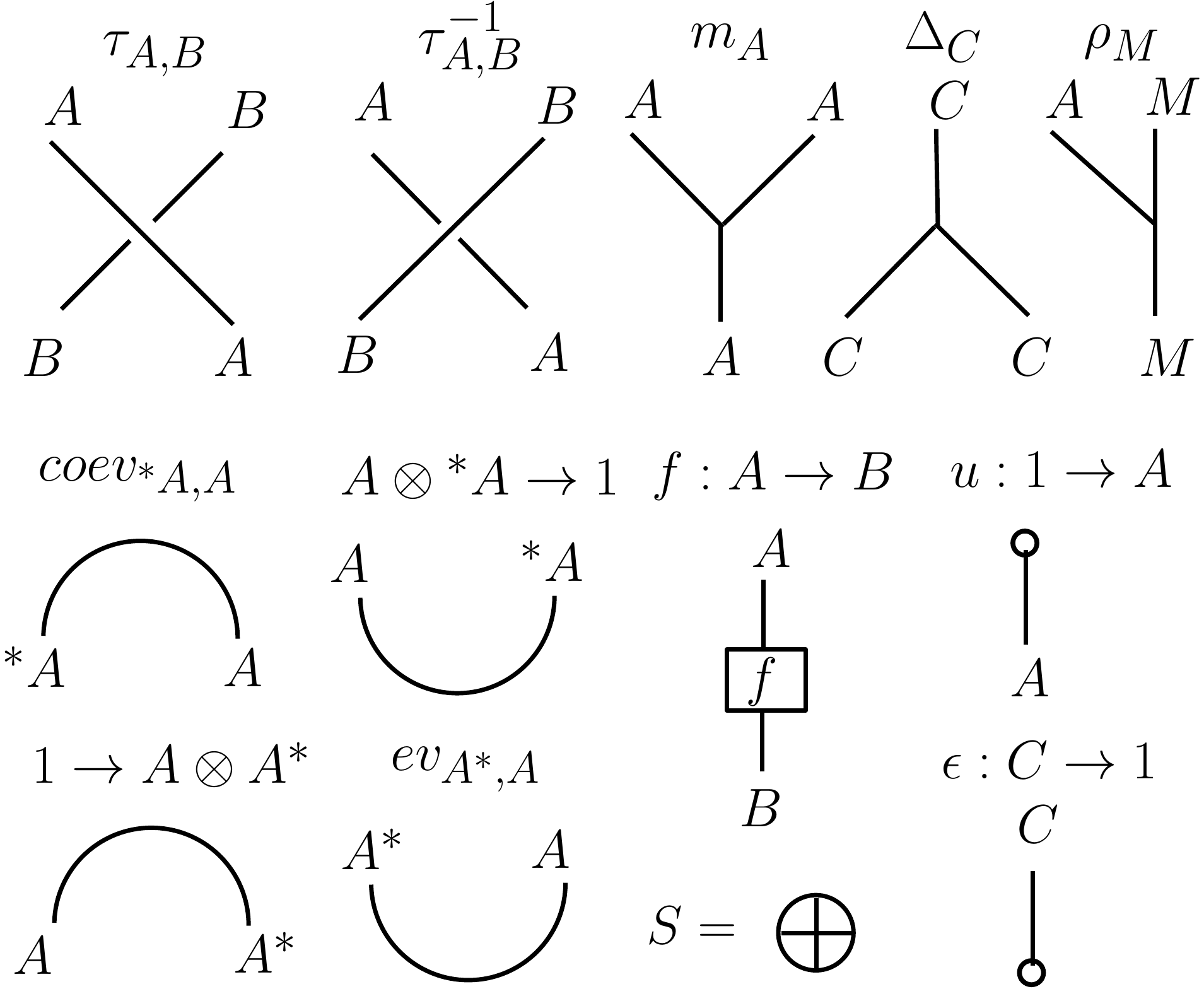}
\end{equation*} 

In the above, $\tau_{A,B}$ is the braided structure with $\tau^{-1}_{A,B}=(\tau_{B,A})^{-1}$.  Thus, a string crossing over another is ``doing all the work".  This is important to keep in mind when dealing with multiple crossing structures where the type of crossing will be marked accordingly; see Definition \ref{def:ocross} and Proposition \ref{half:prop}.  

If $A$ is an algebra then $m_A$ denotes its multiplication, while if $C$ is a coalgebra then $\Delta_C$ is its comultiplication.  If $M$ is an $A$-module then $\rho_M$ is the module structure.  The arcs above denote the evaluation and the co-evaluation  maps of the rigid structure.  The boxed $f$ marks a morphism in the category, with the circled plus denoting the antipode $S$ of a Hopf algebra.  Finally, we have the unit and counit maps.

\section{(Stable) braided modules}\label{sec:stablebraidedmodules}
Let $\Bc$ be a rigid braided category.  Let $A, B$ be  algebras in $\Bc$.  Recall (or see Definition \ref{algprod}) that we can construct, using the braiding,  new algebras: $A\ot^{\tau^{-1}} B$ and $B\ot^{\tau} A$.  Note that $$\tau_{A,B}:A\ot^{\tau^{-1}} B\to B\ot^\tau A$$ is an isomorphism of algebras.  This is not going to work for the cyclic shift because the two algebras use different $\tau$'s.  We need something more sophisticated.

The following notion is essentially identical to the one in \cite{braided}, with differences that make our definition  more suitable to our case.  It allows us to define a diagram over the simplex category $\Delta$, and thus, the Hochschild homology category. 

\begin{definition}
A $\Bc$-module category $\Mc$ is braided if it is equipped with the data of a natural isomorphism  $E_{X,M}$ of the action functor ($\Bc\times\Mc\to\Mc$) satisfying:\begin{enumerate}
\item[C1:] $E_{Y,XM}=\tau^{-1}_{X,Y}E_{Y,M}\tau^{-1}_{Y,X}$
\item[C2:]  $E_{YX,M}=E_{Y,XM}E_{X,M}$
\end{enumerate} 
\end{definition}

We need the following enhancement of the definition above, in order to obtain a diagram over the Connes' cyclic category $\Lambda$, and thus, cyclic homology:
\begin{definition}\label{def:stablebraided}
A braided $\Bc$-module $\Mc$ is stable if it is equipped with the data of a natural isomorphism $\varsigma: Id_\Mc\to Id_\Mc$ satisfying $$E_{X,M}=\varsigma_{XM}\varsigma^{-1}_{M}.$$
\end{definition}

Note that given a braided structure, a stable structure compatible with it is not unique. On the other hand, $\varsigma$ determines $E$ and any $E$ that comes from a $\varsigma$ satisfies $C2$ automatically.

\begin{remark}
Unless $\Bc$ is symmetric, $\Bc$ need not be a braided module over itself.  When it is, there will be usually more than one natural braided structure and even more stable structures.  On the other hand, if $\Bc$ is symmetric then any $\Bc$-module $\Mc$ can be endowed with a stable braided structure $\varsigma_M=Id_M$ for $M\in\Mc$.
\end{remark}

\begin{definition}
A $\Bc$-equivariant functor $F:\Mc\to\Mc'$ between braided modules is braided if the following diagram commutes:$$\xymatrix{F(X\cdot M)\ar[r]\ar[d]^{F(E_{X,M})} & X\cdot F(M)\ar[d]^{E_{X,F(M)}}\\
F(X\cdot M)\ar[r] & X\cdot F(M)
}$$ where the horizontal arrows are part of the equivariant structure.  If $\Mc,\Mc'$ are stable, then $F$ is stable if $$F(\varsigma_M)=\varsigma_{F(M)}.$$ Note that stable implies braided. 
\end{definition}

\subsection{First example: $HH(\Bc)$}\label{firstex}

Let $\Cc$ be a rigid category, then the  center of $\Cc$ is a rigid braided category $\Zc(\Cc)$ and it has a natural stable braided module $HH(\Cc)$ (the anti-center, i.e., the Hochschild homology of $\Cc$).  Namely, for $Z\in\Zc(\Cc)$ and $M\in HH(\Cc)$ we have $Z\cdot M\in HH(\Cc)$ as follows ($X\in \Cc$):  
\begin{equation}\label{centerhh}
	\includegraphics[height=.7in]{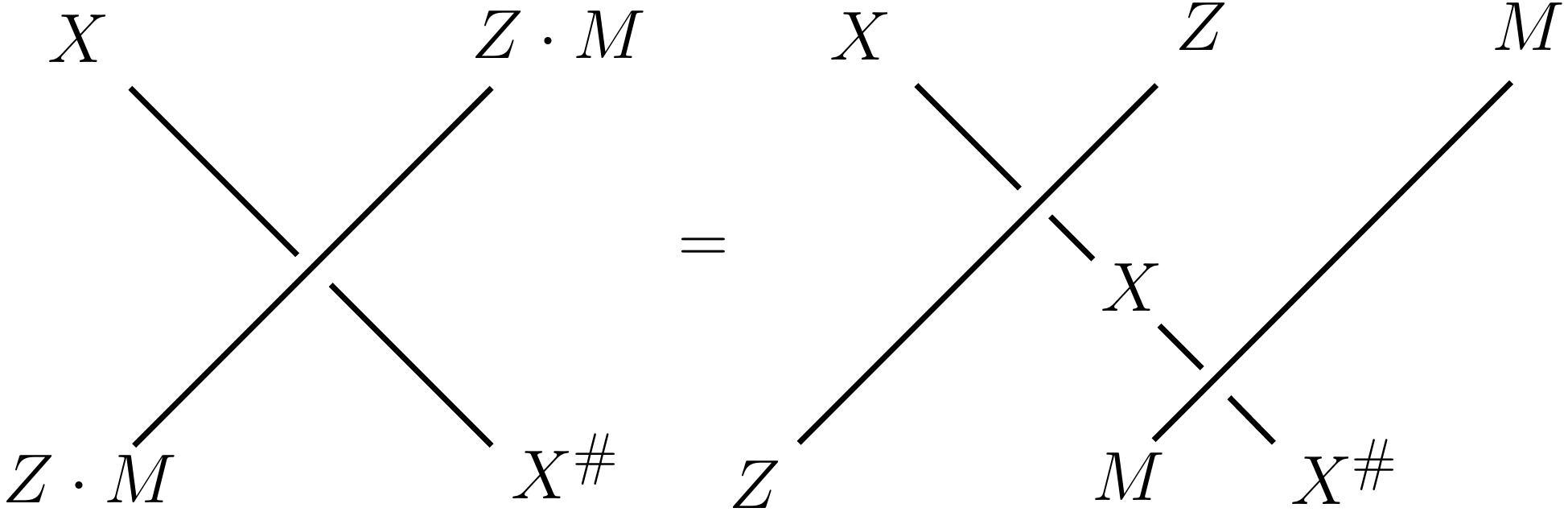} 
\end{equation}  The stable and induced braided structures are 
\begin{equation}\label{centerhhbraid}
	\includegraphics[height=.8in]{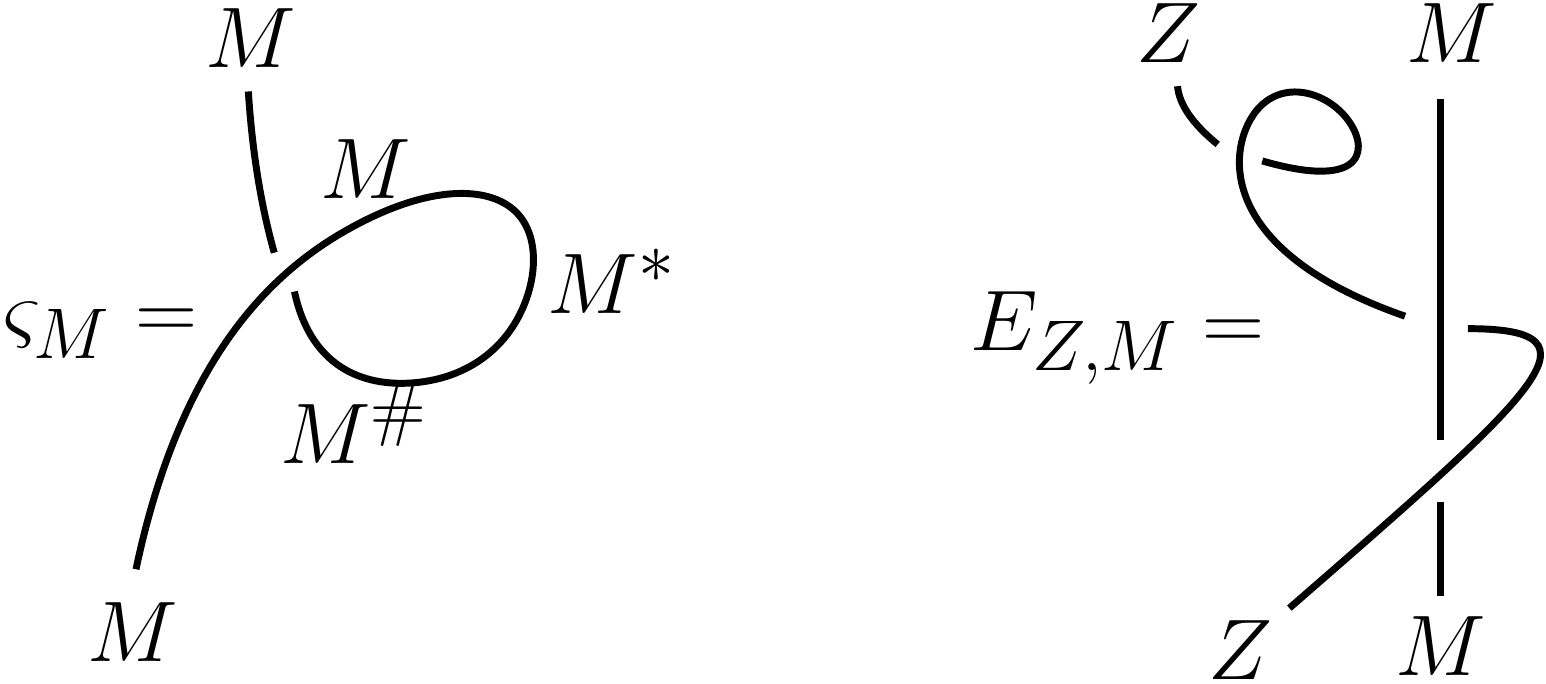} 
\end{equation} the verification of the conditions is immediate.  Namely, we check that $E_{Z,M}\varsigma_M=\varsigma_{ZM}$, the map $\varsigma_M$ is invertible with its inverse being its reflection in the vertical axis, and while $C2$ is automatic, $C1$ is immediate by inspection.  To summarize:
\begin{lemma}
Let $\Cc$ be a rigid category, then $HH(\Cc)$ is a stable braided $\Zc(\Cc)$-module via \eqref{centerhh} and \eqref{centerhhbraid}.
\end{lemma}

\begin{remark}
If $F:\Bc\to\Bc'$ is a braided functor between braided categories and $\Mc'$ is a (stable) braided $\Bc'$-module, then $F^*\Mc'$ is a (stable) braided $\Bc$-module, via $F$, i.e., $X\cdot M=F(X)\cdot M$.
\end{remark}

\begin{corollary}
Let $\Bc$ be a rigid braided category then $HH(\Bc)$ is a stable braided $\Bc$-module via the embedding $\Bc\to\Zc(\Bc)$ given by $\tau$, the braiding of $\Bc$.
\end{corollary}

\begin{remark}\label{tau:rem}
Let $H\in\Bc$ be a Hopf algebra.  Endow an $X\in\Bc$ with the trivial $H$-module structure via $\epsilon$, the counit of $H$.  Let  $V\in\Cc=H_{\Bc}\mod$ then  while $\tau^{-1}_{X,V}$ is $H$-linear, $\tau_{X,V}$ is not.
\end{remark}

 Note that by the above remark we have a braided embedding $$\overline{\Bc}\to\Zc(\Cc),$$   where $\overline{\Bc}$ denotes the braided structure with $\tau^{-1}$ replacing $\tau$.   Thus, the central structure of $X$, viewed as an object in $\Cc$ via $\epsilon$, is given by $\tau^{-1}_{X,-}$.
\begin{corollary}
Let $H\in\Bc$ be a Hopf algebra, let $\Cc=H_{\Bc}\mod$.  Then $HH(\Cc)$ is a stable braided $\overline{\Bc}$-module.
\end{corollary}

\subsection{More (stable) braided modules: $L_\Bc\mod$}
So far the only stable braided module that $\Bc$ is guaranteed to have is $HH(\Bc)$.  In this section we define a notion of an algebra $L$ in $HH(\Bc)$ so that $L_\Bc\mod$ is a (stable) braided $\Bc$-module.  The justification for calling $L$ an algebra in $HH(\Bc)$ is that $HH(\Bc)$ does indeed possess a monoidal structure  with respect to which $L$ is indeed an algebra.  We do not pursue this here.

\begin{remark}
In this section we need to start distinguishing crossings, something that was not necessary in Section \ref{firstex} as the central or anti-central structure was the only structure an object processed there.  Here, an object in $HH(\Bc)$ can be crossed over another object in $\Bc$ in two ways, as an object in $\Bc$ or as an object in $HH(\Bc)$.  We distinguish them as indicated. 
\end{remark}

\begin{definition}\label{def:ocross}
Let $L$ be an algebra in $\Bc$, we say that $L$ is an algebra in $HH(\Bc)$ if $L\in HH(\Bc)$,  in particular $L$ is equipped with the structure $\tau^\circ_{X,L}$ (the hollow dot indicates that it is not the $L\in\Bc\subset\Zc(\Bc)$ structure):
\begin{equation}\label{cross}
	\includegraphics[height=.7in]{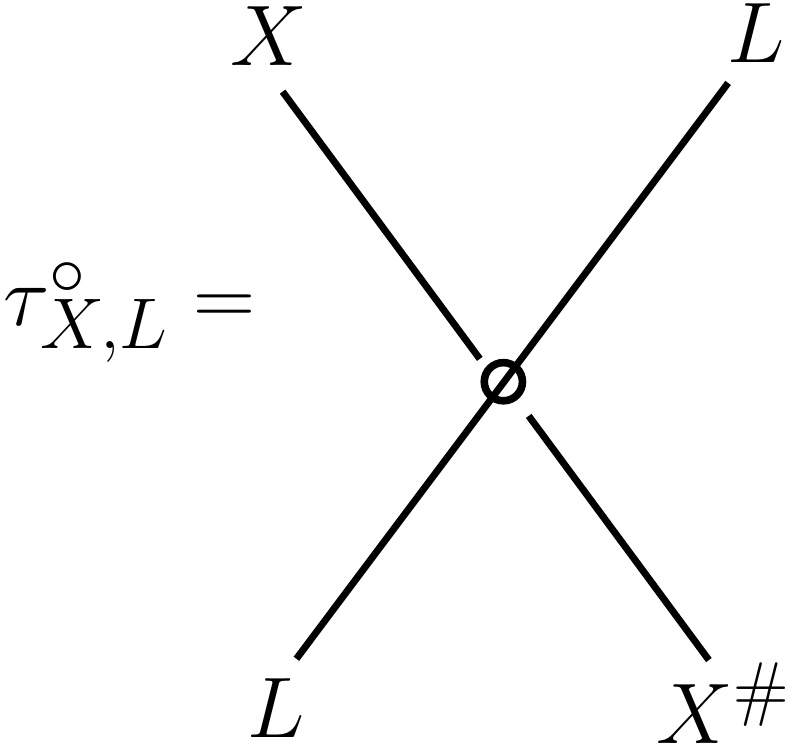} 
\end{equation}
for $X\in\Bc$, that satisfies:
\begin{equation}\label{triplecond}
	\includegraphics[height=.8in]{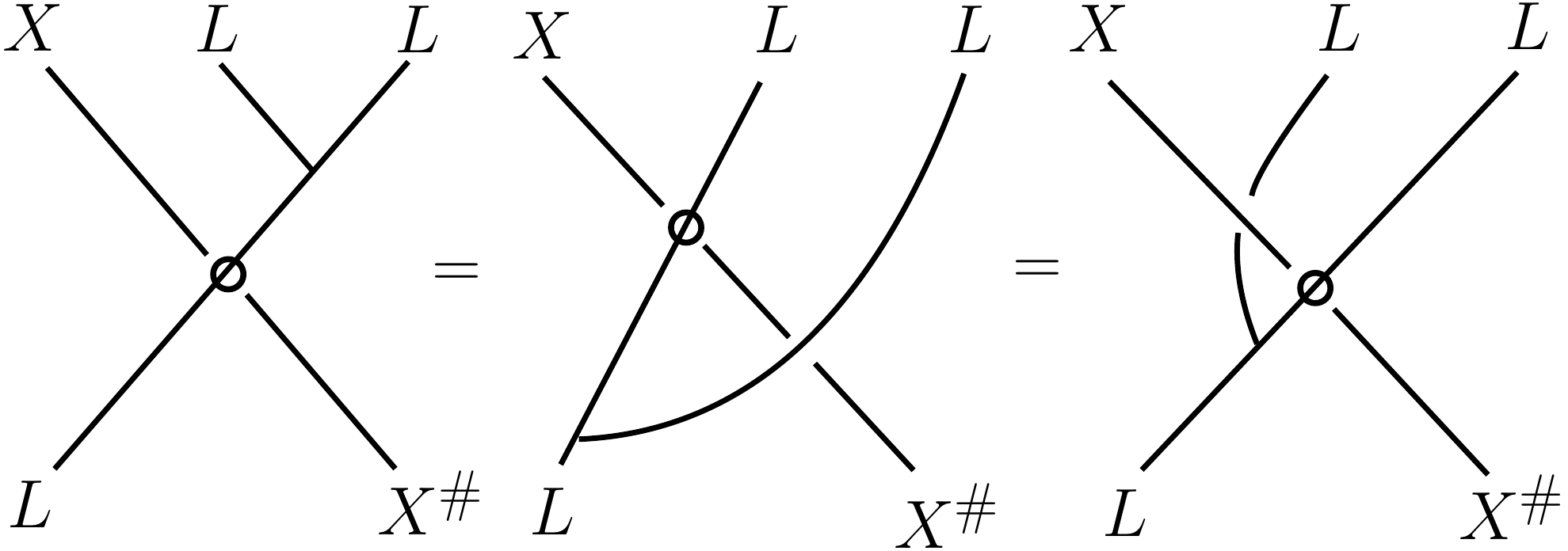} 
\end{equation}
We say that $L$ is stable if it is so as an object of $HH(\Bc)$, i.e., see \eqref{centerhhbraid} we have $\varsigma_L=Id$.

\end{definition}

\begin{definition}\label{Udef}
Let $U: L_\Bc\mod\to HH(\Bc)$ be given by $U(M)=M$ with the the structure:
\begin{equation}\label{hhbdef}
	\includegraphics[height=.8in]{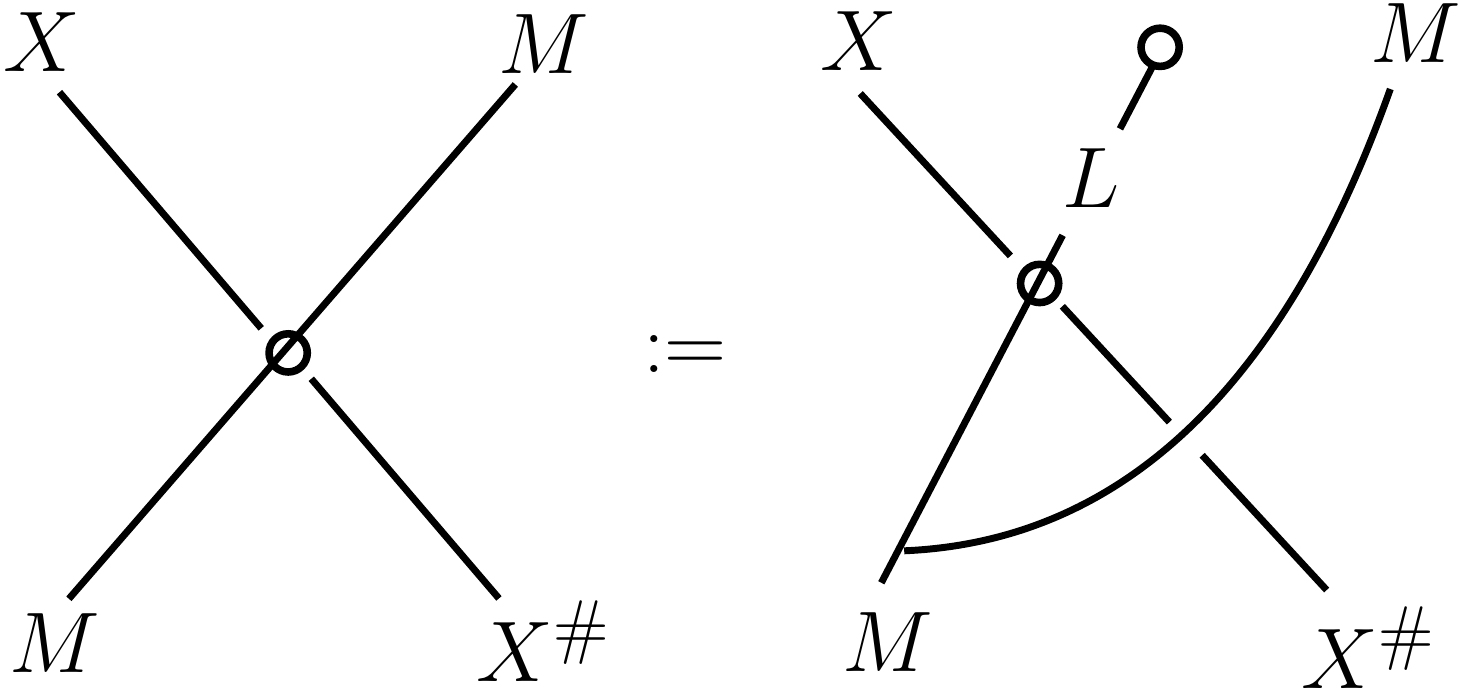} 
\end{equation}
\end{definition}

When considered as a module over itself $L$ recovers its original $HH(\Bc)$ structure.  The following lemma is immediate.

\begin{lemma}\label{LMlem}
Let $L\in HH(\Bc)$ be an algebra. Let $M\in L_\Bc\mod$, then with the structure from Definition \ref{Udef} we have:
\begin{equation}\label{triplecondM}
	\includegraphics[height=.8in]{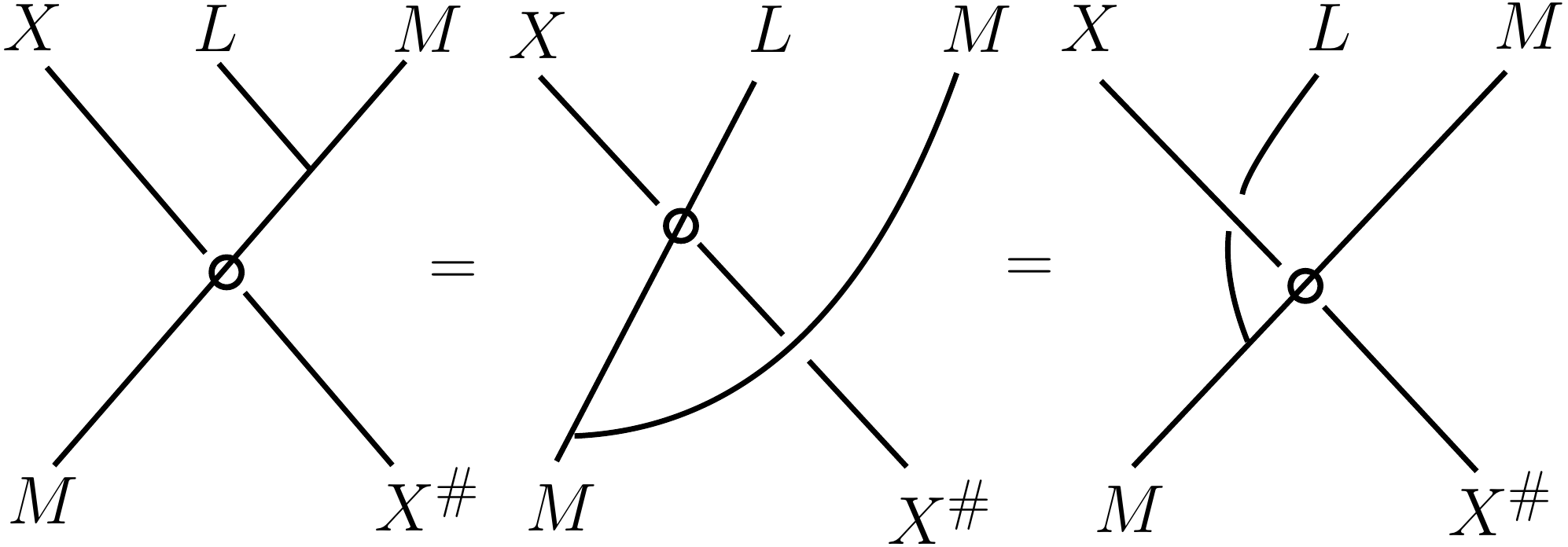} 
\end{equation}
\end{lemma}
The following corollary is immediate using both of the equalities in Lemma \ref{LMlem}.
\begin{corollary}\label{LMcor}
With $L,M$ as above, we have: $$\varsigma_M\rho_{L,M}=\rho_{L,M}\varsigma_L\ot\varsigma_M.$$

\end{corollary}

\begin{definition}\label{Bactdef}
Define a $\Bc$-module structure on $L_\Bc\mod$ as follows: for $X\in\Bc$ and $M\in L_\Bc\mod$ we set $X\cdot M=X\ot M$ with $$\rho_{L,XM}=\rho_{L,M}\tau^{-1}_{L,X}.$$
\end{definition}
The following lemma is immediate.
\begin{lemma}\label{Beqlem}
The functor $U: L_\Bc\mod\to HH(\Bc)$ is $\Bc$-equivariant.
\end{lemma}

\begin{theorem}\label{forget:thm}
Let $L\in HH(\Bc)$ be an algebra then:
\begin{itemize}
\item $L_\Bc\mod$ is naturally a braided $\Bc$-module.
\item If $L$ is stable then so is $L_\Bc\mod$.
\item $U$ is a (stable) braided functor.
\end{itemize}
\end{theorem}
\begin{proof}
For $X\in\Bc$, $M\in L_\Bc\mod$ consider \begin{equation}\label{Lexm}
	\includegraphics[height=.8in]{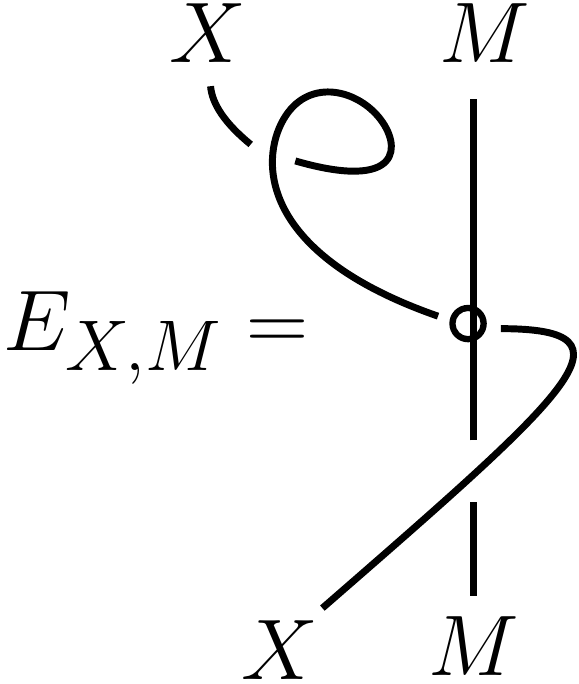} 
\end{equation}then by Lemma \ref{LMlem} and Definition \ref{Bactdef} We have that $E_{X,M}$ is a map of $L$-modules.  Observe that $C1$ follows from Lemma \ref{Beqlem}, and this lemma also reduces $C2$ to the case of $HH(\Bc)$.  Namely, to prove $E_{Y,XM}E_{X,M}=E_{YX,M}$ we need only observe that \begin{equation}\label{Ehhb}
	\includegraphics[height=.9in]{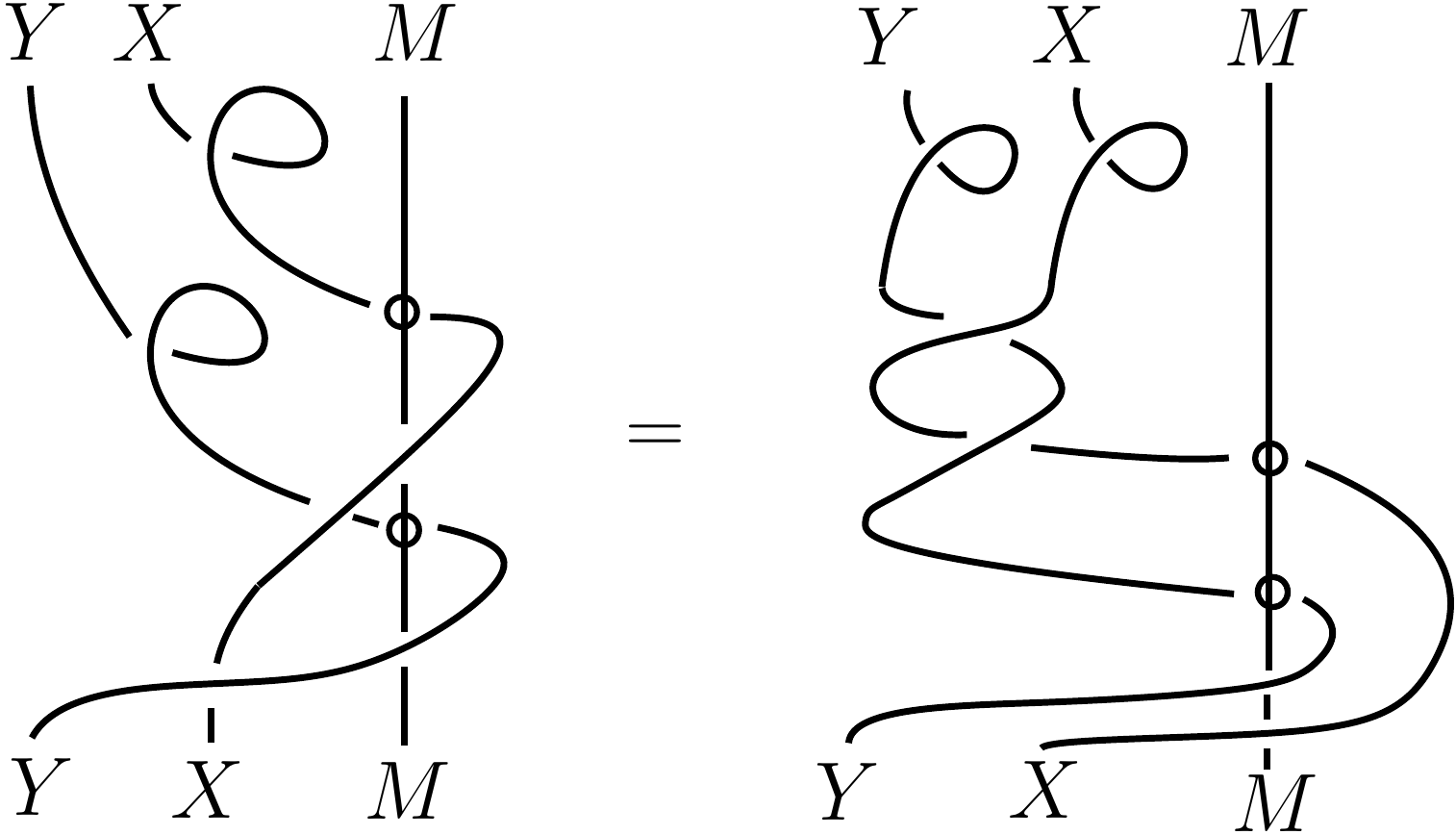} 
\end{equation} which proves that $HH(\Bc)$ is a braided $\Bc$-module; a fact that previously followed from the existence of a stable structure.

If $L$ is stable then by Corollary \ref{LMcor} the map $\varsigma_M$ is $L$-linear, and we again have $E_{X,M}\varsigma_M=\varsigma_{XM}$, proving the second statement.

The last statement  is immediate; the definitions were made to make it so.
\end{proof}

\subsection{Categorified modular pairs in involution: $\Bc_\varsigma$}
We now specialize to the case when $L=1$, but with a necessarily non-trivial $HH(\Bc)$ structure.  The latter is equivalent to a pivotal structure, i.e., a monoidal natural isomorphism $$\rho_X:X\to X^\#$$ for all $X\in\Bc$.  More precisely,  we have $$\tau^\circ_{X,1}=\rho_X,$$ where $\tau^\circ$ denotes the $HH(\Bc)$ structure of $1$.  It is not hard to see that $L=1$ is a stable algebra in $HH(\Bc)$.

Observe that in this case $L_\Bc\mod=\Bc$ as left $\Bc$-modules.  Thus, if $M\in\Bc$ then it acquires $HH(\Bc)$ structure:

\begin{equation}\label{hhstr1}
	\includegraphics[height=.7in]{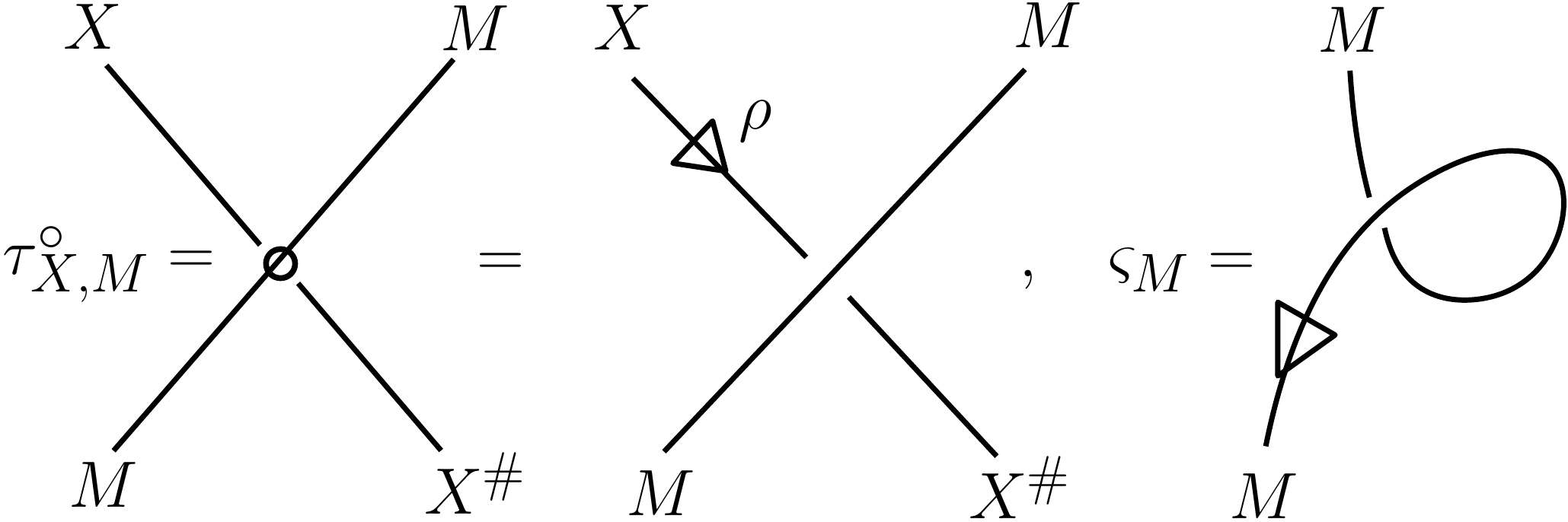} 
\end{equation}

Note that $\theta=\varsigma^{-1}$ is  a twist on $\Bc$, i.e., $$\theta_{XY}=\tau_{Y,X}\tau_{X,Y}\theta_X\theta_Y$$ so that $\Bc$ is balanced.  The structures $\theta,\varsigma$, and $\rho$ are equivalent.   Explicitly, the braided structure on $\Bc$ is  $$E_{X,M}=\tau^{-2}_{X,M}\varsigma_X.$$

\begin{lemma}\label{class:lem}
Consider $\Bc$ as a module over itself under left multiplication.  Then braided structures on this module are in bijective correspondence with anti-twists, i.e.,  $\varsigma\in Aut(Id_\Bc)$ satsifying: $$\varsigma_{XY}=\tau^{-2}\varsigma_X\varsigma_Y.$$ 
\end{lemma}
\begin{proof}
Given $\varsigma$, let $E^\varsigma_{X,M}=\tau^{-2}_{X,M}\varsigma_X$, then it is a stable braided structure by the preceding discussion.  Conversely, given $E_{X,M}$, let $\varsigma^E_X=E_{X,1}$. Then $$\varsigma^E_{XY}=E_{XY,1}=E_{X,Y1}E_{Y,1}=E_{X,Y}\varsigma^E_Y=\tau^{-1}_{Y,X}E_{X,1}\tau^{-1}_{X,Y}\varsigma_Y^E=\tau^{-2}_{X,Y}\varsigma_X^E\varsigma_Y^E,$$ so that $\varsigma^E$ is an anti-twist.  Furthermore, $$\varsigma^{E^\varsigma}_X=E^\varsigma_{X,1}=\tau^{-2}_{X,1}\varsigma_X=\varsigma_X$$ and $$E^{\varsigma^E}_{X,M}=\tau^{-2}_{X,M}\varsigma_X^E=\tau^{-2}_{X,M}E_{X,1}=\tau^{-1}_{M,X}E_{X,1}\tau^{-1}_{X,M}=E_{X,M}.$$
\end{proof}

\begin{remark}
Lemma \ref{class:lem} shows that any braided module structure on $\Bc$ is automatically stable in a canonical way: $\varsigma_X=E_{X,1}$.  Note, however, that here stable structures form a torsor over $Aut_\Bc(1)$ and so are not unique.
\end{remark}

\begin{definition}
If $\varsigma$ is an anti-twist on $\Bc$, let $\Bc_\varsigma$ denote the corresponding stable braided module.
\end{definition}

\begin{remark}\label{rem:decomp}
With the assistance of  Theorem \ref{forget:thm} we have that $$\bigoplus_{\varsigma\in\text{anti-twists}}\Bc_\varsigma\to HH(\Bc)$$ is a fully faithful stable braided embedding.  
\end{remark}

\begin{definition}
 Let $$G=(\overline{\Bc^\times},\ot)$$ be the abelian group of isomorphism classes of invertible objects in $\Bc$.  Let $$\hat{G}=(Aut^\ot(Id_\Bc),\circ)$$ be the abelian group of monoidal natural automorphisms of $Id_\Bc$.  If $y\in\Bc^\times$, i.e., $y$ is an invertible object of $\Bc$ then let $\phi^y\in Aut^\ot(Id_\Bc)$ be given by \begin{equation}\label{phidef}
	\includegraphics[height=.8in]{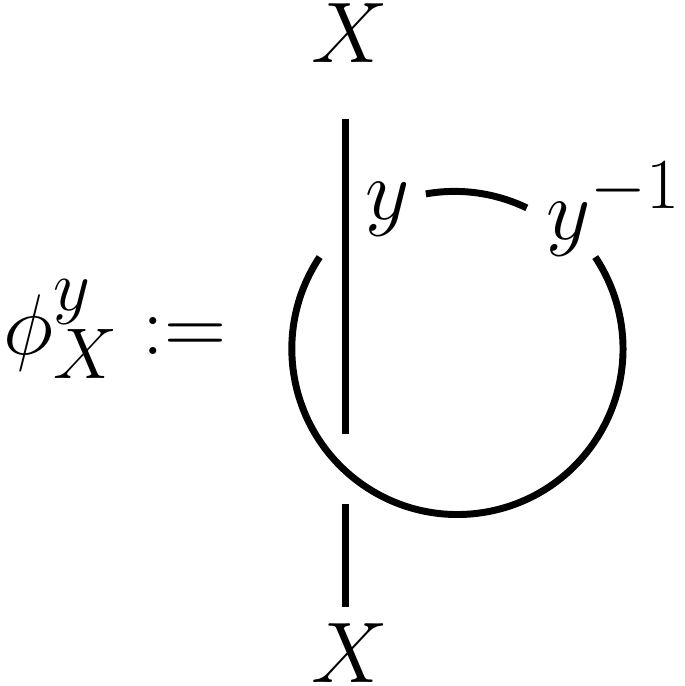} 
\end{equation} Thus,  $\phi:\overline{\Bc^\times}\to Aut^\ot(Id_\Bc)$ is a group homomorphism. 
\end{definition}

\begin{lemma}\label{classbr:lem}
If $F:\Bc_\varsigma\to\Bc_{\varsigma'}$ is a braided equivalence, then $F(M)=M\ot y$ for some $y\in\Bc^\times$ and $$\varsigma'=\phi^y\varsigma.$$ Furthermore, $F$ is stable if $$\varsigma'_y=Id_y.$$
\end{lemma}

\begin{proof}
Any equivariant equivalence $F$ is of the form $F(M)=My$ for some $y\in\Bc^\times$.  $F$ braided needs $E_{X,M}=E'_{X,My}$, so $\tau^{-2}_{X,M}\varsigma_X=\tau^{-2}_{X,My}\varsigma'_X$, so $\tau^{-1}_{M,X}\varsigma_X\tau^{-1}_{X,M}=\tau^{-1}_{M,X}\tau^{-2}_{X,Y}\varsigma'_X\tau^{-1}_{X,M}$, so $\varsigma_X=\tau^{-2}_{X,Y}\varsigma'_X$, so $\varsigma'=\phi^y\varsigma$.  In addition, for stability of $F$, we must have $\varsigma'_{My}=\varsigma_M$.  Thus, \begin{equation}\label{eq:newsigma}\varsigma_M=\varsigma'_{My}=\phi^y_{My}\varsigma_{My}=\phi^y_M\phi^y_{y}\tau^{-2}_{M,y}\varsigma_M\varsigma_y=\varsigma_M(\phi^y_M\tau^{-2}_{M,y})(\phi^y_y\varsigma_y)=\varsigma_M\varsigma'_y,\end{equation} so $\varsigma'_y=Id_y.$
\end{proof}

Note that Lemma \ref{classbr:lem} classifies braided structures on $\Bc$ together with their equivalences. It also accounts for their stable equivalences if the canonical stable structures are considered.    If we set $$C=Aut_\Bc(1),$$ then $C$ is canonically isomorphic, as an abelian group, to $Aut_\Bc(y)$ for every $y\in\Bc^\times$.  
\begin{remark}\label{rem:ctorsor}
The set of stable structures on $\Bc_\varsigma$ is canonically isomorphic to $C$.  Furthermore, if $-\ot y:\Bc_\varsigma\to\Bc_{\varsigma'}$ is a braided equivalence then by \eqref{eq:newsigma} the new stable structure on $\Bc_\varsigma$ that it inherited from $\Bc_{\varsigma'}$ is $$\varsigma^{new}=\varsigma\varsigma'_y,\quad \varsigma'_y\in C.$$
\end{remark}

The set of anti-twists of $\Bc$ is a $\hat{G}$-torsor.  Let $\mathcal{G}$ denote the action groupoid of $G$ on anti-twists via $\phi$.  Define $\eta:\mathcal{G}\to C$   by $\eta(y,\varsigma)=\phi^y_y\varsigma_y$ then $\eta$ is  a homomorphism.  The following is immediate:

\begin{corollary}\label{class:cor}
The groupoid of braided structures on $\Bc$ is isomorphic to $\mathcal{G}$.  The groupoid of canonically stable braided structures is isomorphic to $ker(\eta)$.
\end{corollary}

\subsection{$HH(\Bc)$ decomposes into $\Bc_\varsigma$'s}\label{sec:decompo}
Here we examine in detail a case when the fully faithful embedding of Remark \ref{rem:decomp} is an equivalence; we will use this in Section \ref{sec:blah} for explicit computations. Let $G$ be a finite  abelian group. One may consider $\vect_G$, the monoidal category of $G$-graded vector spaces.  More precisely, with $V_x\in\vect$ $$V\in\vect_G\,\,\text{is given by}\,\,V=\oplus_{x\in G} V_x\,\,\text{and}\,\,(V \ot W)_x=\oplus_{y\in G}V_y\ot W_{y^{-1}x}.$$ 

Let $\chi:G\times G\to k^\times$ be a bi-character of $G$.  Then $\Bc=(Vec_G,\chi)$ is a ribbon category with the braiding: \begin{align*}V_x\ot W_y&\to W_y\ot V_x\\ v\ot w&\mapsto \chi(x,y)w\ot v\end{align*} and the ribbon element $\theta(v)=\chi(x,x)v$ for $v\in V_x$.  Define $$\omega(x,y)=\chi(x,y)\chi(y,x)$$ and let $\varsigma=\theta^{-1}$ be the canonical anti-twist, so that $$\varsigma(x)=\chi(x,x^{-1}).$$  Note that any anti-twist can be obtained from $\varsigma$ via $\varsigma\lambda$ with $\lambda\in\hat{G}.$  

\begin{lemma}
Set $\Bc_\lambda=\Bc_{\varsigma\lambda}$, the stable braided $\Bc$-module $\Bc$ with stable braided structure given by the anti-twist $\varsigma\lambda$.  Then  $$HH(\Bc)=  \bigoplus_{\lambda\in \hat{G}}\Bc_\lambda$$ as stable braided $\Bc$-modules.
\end{lemma}

\begin{proof}
In light of Remark \ref{rem:decomp} it suffices to point out that as abelian categories we have $HH(\vect_G)=HH(\Bc)=\Zc(\Bc)=\vect_{G\times\hat{G}}$.
\end{proof}

An immediate corollary of Lemma \ref{classbr:lem} is the following:

\begin{lemma}\label{lem:iso}
As braided modules, $\Bc_\lambda\simeq\Bc_\mu$, if and only if  there exits a $y\in{G}$ such that  $$\lambda(x)=\mu(x)\omega(x,y)\forall x\in G.$$  They are isomorphic as stable braided modules if in addition $$\mu(y)=\varsigma(y).$$
\end{lemma}

\begin{corollary}\label{cor:decompmore}
Consider the bicharacter $\omega$ as a homomorphism $\omega:G\to\hat{G}$. \begin{itemize}
\item If $\omega$ is trivial, thus, $\Bc$ is symmetric then $\Bc_\lambda\simeq\Bc_\mu$ as braided modules if and only if $\lambda=\mu$. 
\item If $\omega$ is an isomorphism so that $\Bc$ is non-degenerate, then $\Bc_\lambda\simeq\Bc_\mu$ as braided modules for all $\lambda,\mu\in\hat{G}$.  However, as stable braided modules $$\Bc_\lambda\simeq\Bc_\mu\iff \varsigma(\omega^{-1}(\lambda))=\varsigma(\omega^{-1}(\mu)).$$ 
\item Recall that $\Bc$ is equipped with a canonical anti-twist $\varsigma(x)=\chi(x,x^{-1})$. If $\omega$ is an isomorphism then $$(HH(\Bc),\varsigma)=\bigoplus_{y\in G}(\Bc_\varsigma, \varsigma^{-1}(y)\varsigma)$$ as stable braided modules, where $(\Bc_\varsigma, \varsigma^{-1}(y)\varsigma)$ denotes the braided module $\Bc_\varsigma$ but with the canonical stable structure modified by the scalar $\varsigma^{-1}(y)$.
\end{itemize}
\end{corollary}

\begin{proof}
The first two items are consequences of Lemma \ref{lem:iso}.  The last follows from Remark \ref{rem:ctorsor}.
\end{proof}

\begin{remark}\label{pair:rem}
Note that the groupoid $\mathcal{G}$  of Corollary \ref{class:cor} admits an involution $(y,\varsigma)\mapsto (y^{-1},\varsigma^\dagger)$ where $\varsigma_X^\dagger= {}^*\varsigma_{X^*}$; in the above $(\varsigma\lambda)^\dagger=\varsigma\lambda^{-1}$.  Observe that if $\Bc_\lambda\simeq\Bc_{\lambda^{-1}}$ as braided modules, then since $\varsigma(x)=\varsigma(x^{-1})$ so $\Bc_\lambda\simeq\Bc_{\lambda^{-1}}$ as stable braided modules.  Thus, if $\omega$ is an isomorphism, and $\lambda\neq\lambda^{-1}$, we have a pair of stably isomorphic components in $HH(\Bc)$.

\end{remark}

\subsection{$HH(\Bc)$ does not decompose into $\Bc_\varsigma$'s}\label{sec:nodecompo}
Consider a non-abelian finite group $G$.  Let $\Bc=Rep(G)$ be the symmetric braided category of finite dimensional $G$-representations, with the  trivial braiding.  For $g\in G$ consider the $\Bc$-module category $\Mc_g=Rep(C_G(g))$, where $C_G(g)$ denotes the centralizer of $g$ in $G$.  For $M\in \Mc_g$, define $$\varsigma_M=g|_M,$$ then $\varsigma$ is a stable braided structure on $\Mc_g$.  Let $I$ denote the set of conjugacy classes of $G$, let $g_i$ be a representative of the $i$th class, then it is immediate that $$HH(\Bc)=\bigoplus_{i\in I}\Mc_{g_i}$$ as stable braided modules.   Thus, the embedding of Remark \ref{rem:decomp} recovers only the part of $HH(\Bc)$ that corresponds to singleton conjugacy classes, i.e., elements $g\in \Zc(G)$.

\section{Stable modules and cyclic homology}\label{sec:stablemodsandcyclic}
In this section we demonstrate the relevance for us of the above notions of stability and braiding of module categories over $\Bc$.

\begin{lemma}\label{L1}
Let $\Mc$ be a braided $\Bc$-module, $B\in\Bc$ an algebra, and $M\in B_\Mc\mod$.  Let $\rho_{B,M}:B\cdot M\to M$ denote the $B$-module structure on $M$.  Then $$\hat{\rho}_{B,M}=\rho_{B,M}E_{B,M}$$ is also a $B$-module structure on $M$, denoted by ${}^B \!M$.  Furthermore, if $\Mc$ is stable then $$\varsigma_M:M\to{}^B\!M$$ is an isomorphism in $ B_\Mc\mod$.
\end{lemma}
\begin{proof}
That $\hat{\rho}_{B,M}$ is an action is immediate from $C2$.  On the other hand $$\hat{\rho}_{B,M}\varsigma_M=\rho_{B,M}E_{B,M}\varsigma_M=\rho_{B,M}\varsigma_{BM}=\varsigma_M\rho_{B,M}.$$
\end{proof}

Note that we also have $M^B$ obtained by modifying the action thus: $\rho_{B,M}E^{-1}_{B,M}$, so that ${}^B\!M^B=M$ and $M\mapsto {}^B\!M$ is an automorphism of $B_\Mc\mod$.

\begin{lemma}\label{L2}
Let $M\in A\ot^\tau B_\Mc\mod$ then ${}^B\!M$ with the same $A$-action, and the $B$-action modified to $\hat{\rho}$ is in $B\ot^\tau A_\Mc\mod$. 
\end{lemma}
\begin{proof}
In light of \eqref{f1} this follows from $\rho_{A,M}\hat{\rho}_{B,M}\tau^{-1}_{B,A}=\rho_{A,M}\rho_{B,M}E_{B,M}\tau^{-1}_{B,A}\\=\rho_{B,M}\rho_{A,M}\tau^{-1}_{A,B}E_{B,M}\tau^{-1}_{B,A}=\rho_{B,M}\rho_{A,M}E_{B,AM}=\rho_{B,M}E_{B,M}\rho_{A,M}=\hat{\rho}_{B,M}\rho_{A,M}.$
\end{proof}

The following definition is why we need braided $\Bc$-modules.  It is essentially equivalent to the concept.

\begin{definition}
Let $\Mc$ be a braided $\Bc$-module and $A,B$ be algebras in $\Bc$.  We define an isomorphism of categories: $$\sigma_{A,B}^\Mc:A\ot^\tau B_\Mc\mod\to B\ot^\tau A_\Mc\mod$$ by $\sigma_{A,B}^\Mc(M)={}^B\!M$.
\end{definition}

\begin{remark}\label{nat:rem}
Let $f:A\to A'$ and $g:B\to B'$ be algebra maps, then the functors $\sigma^\Mc_{A,B}\circ(f^*\ot g^*)$ and $(g^*\ot f^*)\circ\sigma^\Mc_{A',B'}$ from $A'\ot^{\tau}B'_\Mc\mod$ to $B\ot^{\tau}A_\Mc\mod$ are equal.
\end{remark}

\begin{lemma}\label{L3}
Let $M\in A\ot^\tau B\ot^\tau C_\Mc\mod$ then in $B\ot^\tau C\ot^\tau A_\Mc\mod$ we have: $${}^B\!({}^C\!M)={}^{B\ot^\tau C}\!M,\quad\text{i.e.,}\quad \sigma^\Mc_{C\ot^\tau A, B}\sigma^\Mc_{A\ot^\tau B,C}=\sigma^\Mc_{A,B\ot^\tau C}.$$  We say that $\sigma_{-,-}^\Mc$ is associative in the second component.
\end{lemma}

\begin{proof}
The $A$-action is unaffected, for the rest: $\hat{\rho}_{BC,M}=\rho_{B,M}\rho_{C,M}E_{BC,M}\\=\rho_{B,M}\rho_{C,M}E_{B,CM}E_{C,M}=\rho_{B,M}E_{B,M}\rho_{C,M}E_{C,M}=\hat{\rho}_{B,M}\hat{\rho}_{C,M}$.
\end{proof}

\begin{corollary}
If $\Mc$ is a stable braided $\Bc$-module then $$\varsigma:Id\to \sigma_{B,A}^\Mc \sigma_{A,B}^\Mc$$ is a natural isomorphism in $A\ot^\tau B_\Mc\mod$.
\end{corollary}
\begin{proof}
By Lemma \ref{L3} we have $\sigma_{B,A}^\Mc \sigma_{A,B}^\Mc=\sigma^\Mc_{1,A\ot^\tau B}$ and Lemma \ref{L1} completes the proof.
\end{proof}

\subsection{Definitions of $HH_\Mc(\Cc)$ and $HC_\Mc(\Cc)$ via limits}\label{seccyclic}
Assume that $\Mc$ is a braided $\Bc$-module.  For $n\geq 0$, define the categories of modules $$\Cc_n=H^{\ot^{\tau}n+1}_\Mc\mod.$$ We will abuse notation and write $H^{n+1}$ for the algebra $H^{\ot^{\tau}n+1}=H_0\ot^\tau\cdots\ot^\tau H_n$ in $\Bc$. 
\begin{definition}\label{cyclicdefinition}
Define functors:
\begin{align}\tau_n&=\sigma^\Mc_{H^n,H}:\Cc_n\to\Cc_n\label{tt}\\
d_i&=(Id_H^i\ot \Delta\ot Id_H^{n-1-i})^*:\Cc_n\to\Cc_{n-1}, \quad 0\leq i\leq n-1\\
d_n&=d_0\tau_n\label{ddef}\\
s_i&=(Id_H^{i+1}\ot\epsilon\ot Id_H^{n-i})^*:\Cc_n\to\Cc_{n+1}, \quad -1\leq i\leq n
\end{align}
\end{definition}

\begin{lemma}\label{rels} The functors satisfy the following relations (the equalities are those of functors, they are not isomorphisms):
\begin{align}d_i\tau_n&=\tau_{n-1}d_{i-1}, \quad 1\leq i\leq n-1\label{l0}\\
d_0\tau_n^2&=\tau_{n-1}d_{n-1}\label{ll3}\\
\intertext{\eqref{l0} and \eqref{ll3} together with \eqref{ddef} imply that \eqref{l0} is valid for $1\leq i\leq n$}
s_i\tau_n&=\tau_{n+1}s_{i-1}, \quad 0\leq i\leq n\label{l1}\\
s_{-1}&=\tau_{n+1}s_n\label{ll2}\\
\intertext{\eqref{l1} and \eqref{ll2} imply that}
s_0\tau_n&=\tau^2_{n+1}s_n
\end{align}
\end{lemma}
\begin{proof}
The identity \eqref{l0} follows from Remark \ref{nat:rem} by considering $$f=Id^{i-1}\ot\Delta\ot Id^{n-1-i}:H^{n-1}\to H^n\quad\text{and}\quad g=Id:H\to H.$$ Indeed, $d_i\tau_n=(g\ot f)^*\sigma^\Mc_{H^n,H}$ and $\tau_{n-1}d_{i-1}=\sigma^\Mc_{H^{n-1},H}(f\ot g)^*$.  Similarly, \eqref{l1} follows from the same remark with $$f=Id^i\ot\epsilon\ot Id^{n-i}: H^{n+1}\to H^n\quad\text{and}\quad g=Id:H\to H.$$  And \eqref{ll2} follows, again by Remark \ref{nat:rem}, from $$s_{-1}=(\epsilon\ot Id^{n+1})^*=(\epsilon\ot Id^{n+1})^*\sigma^\Mc_{H^{n+1},1}=\sigma^\Mc_{H^{n+1},H}(Id^{n+1}\ot\epsilon)^*=\tau_{n+1}s_n.$$  Finally, \eqref{ll3} is obtained by following Lemma \ref{L3} with Remark \ref{nat:rem}, i.e., $$d_0\tau_n^2=(\Delta\ot Id^{n-1})^*(\sigma^{\Mc}_{H^n,H})^2=(\Delta\ot Id^{n-1})^*\sigma^\Mc_{H^{n-1}, H^2}=\sigma^\Mc_{H^{n-1},H}(Id^{n-1}\ot \Delta)^*=\tau_{n-1}d_{n-1}.$$
\end{proof}

\begin{remark}\label{taun}
Note that in the above $\tau^{n+1}_n\neq Id$, but $\tau^{n+1}_n=\sigma^\Mc_{1,H^{n+1}}$ and so (we need that $\Mc$ is stable here), by Lemma \ref{L1} \begin{equation}\label{taucycliciso}\varsigma:Id\simeq \tau^{n+1}_n.\end{equation} Furthermore, observe that as endofunctors of $\Cc_0$ we have $d_i s_j=Id$ except for \begin{equation}\label{center}d_1s_{-1}=\tau_0=\sigma^\Mc_{1,H},\end{equation} since $d_1s_{-1}=d_0\tau_1^2s_0=d_0s_0\tau_0=\tau_0$.
\end{remark}

\begin{proposition}\label{fullcyc}
Let $\Mc$ be a stable braided $\Bc$-module.  The categories $\Cc_n$ together with the functors of Definition \ref{cyclicdefinition} and the isomorphism \eqref{taucycliciso} form a cyclic object in categories.  If stability is not assumed then we only get a simplicial object.
\end{proposition}
\begin{proof}
Given the content of Lemma \ref{rels}, it suffices to demonstrate only the simplicial relations on the functors.  More precisely, we need:
\begin{align}
d_{i}  d_{j} &= d_{j-1}d_{i},  \quad   i <j\label{r1}\\
s_{i}  s_{j} &= s_{j} s_{i-1},    \quad i > j\notag\\
d_{i} s_{j} &=
 \begin{cases}\label{r3}
s_{j-1}d_i,   &
 i<j\\
Id,   &   i=j \,\,\text{or}\,\, i=j+1\\
s_{j}d_{i-1} ,  & i>j+1
\end{cases}
\end{align} and most of these are immediate since they are obtained by applying $(-)^*$ to algebra maps that classically satisfy these relations.  The non-obvious ones are only those that involve $d_n$.  For \eqref{r1} we have (using the definition and applying the classic relations followed by Lemma \ref{rels}) $d_i d_n=d_id_0\tau_n=d_0d_{i+1}\tau_n=d_0\tau_{n-1}d_i=d_{n-1}d_i$. For \eqref{r3} we get $d_ns_{n-1}=d_0\tau_n s_{n-1}=d_0 s_{-1}=Id$ and $d_n s_j=d_0\tau_n s_j=d_0s_{j+1}\tau_{n-1}=s_jd_0\tau_{n-1}=s_j d_{n-1}$.  Remark \ref{taun} finishes the proof.

\end{proof}

\begin{definition}\label{HHHC:def}
Let $H$ be a Hopf algebra in $\Bc$, set $\Cc=H_\Bc\mod$.  Let $\Mc$ be a  braided $\Bc$-module.  Let $$HH_\Mc(\Cc)=\varinjlim_{\Delta^{op}}\Cc_\bullet$$ and if $\Mc$ is stable braided, let $$HC_\Mc(\Cc)=\varinjlim_{\Lambda^{op}}\Cc_\bullet.$$
\end{definition}

\subsection{Dualizing  the diagrams over $\Lambda$}\label{lambdadiagram}
To compute the limits in Definition \ref{HHHC:def} we will use the right adjoints of the functors of Definition \ref{cyclicdefinition}. Let us change notation and use $d_i,s_i,\tau_n$ for arrows in $\Lambda$. Rename what we called by such names in Section \ref{seccyclic}, which will now be denoted by $d_i^*,s_i^*,\tau_n^*$.  Their right adjoints, that we will describe below, will now be denoted by $d_{i*},s_{i*},\tau_{n*}$.

It is immediate that in order to understand the right adjoints, it suffices to describe $\Delta_*, \epsilon_*$, and $\tau_{n*}$.  Let us start with $\Delta_*$ which is the right adjoint of $\Delta^*$ and in this case both functors are literally what the notation suggests them to be, namely, we have the coproduct $\Delta: H\to H\ot^\tau H$, and the functors are obtained from it as in Section \ref{sec:adjunction}.  More precisely, by \eqref{adjunction}, with $A=H$, $B=H\ot^\tau H$, $C=1$, and $M=H\ot^\tau H$ with the left $H$-action via $\Delta$ and the right $H\ot^\tau H$-action via right multiplication in $H\ot^\tau H$, we have that since $\Delta^*(-)=M\cdot_B -$, so \begin{equation}\label{radjoint}
\Delta_*(-)=-\ra_H H\ot^\tau H.
\end{equation} In order to simplify \eqref{radjoint} we need the following lemma.

\begin{lemma}\label{simplifylem}Let $H$ be a Hopf algebra in $\Bc$.  Then the maps  \begin{equation}\label{twomaps}
\includegraphics[height=.7in]{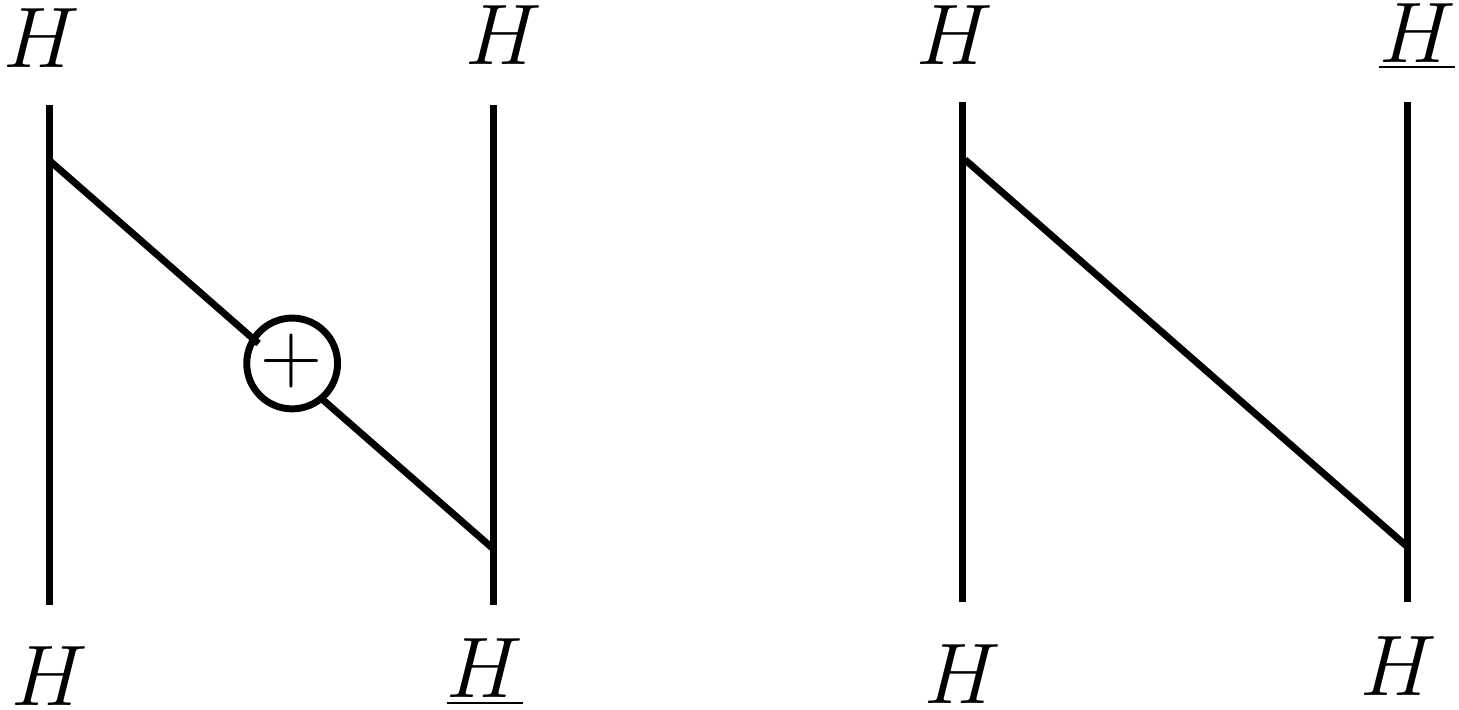}
\end{equation} are inverses of each other and identify the two $(H, H\ot^\tau H)$-bimodules, where 

\begin{enumerate}[a)]
\item the left action of $H$ on $H\ot H$ is via $\Delta: H\to H\ot^\tau H$ and multiplication in $H\ot^\tau H$, while on $H\ot\underline{H}$ it is via multiplication on the first $H$-factor only.

\item the right action of $H\ot^\tau H$ on $H\ot H$ is via multiplication in $H\ot^\tau H$, while on $H\ot\underline{H}$ it is given via the diagram: \begin{equation}\label{act}
\includegraphics[height=.8in]{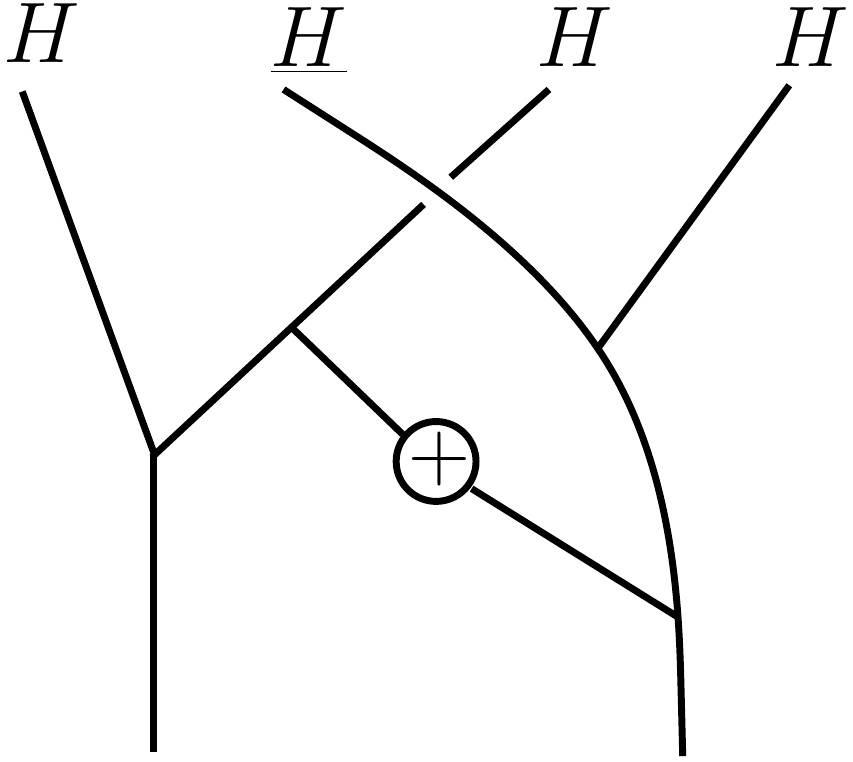}
\end{equation}
\end{enumerate}
Note that the maps in \eqref{twomaps} are still $H$-linear if $V\in H_\Bc\mod$ replaces the second $H$.

\end{lemma}
\begin{proof} For $a)$ we have \begin{equation*}
\includegraphics[height=.8in]{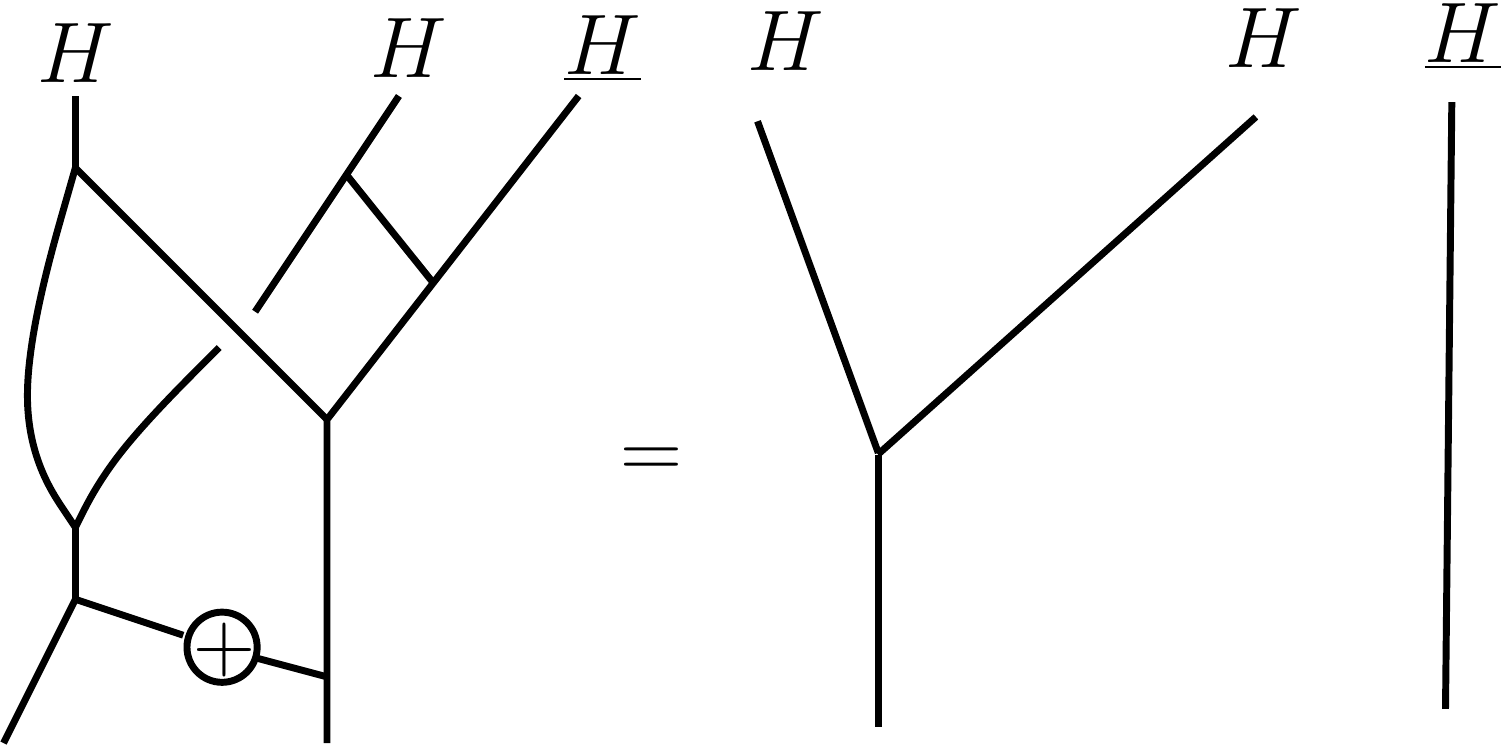}
\end{equation*}
and for $b)$ we have \begin{equation*}
\includegraphics[height=.8in]{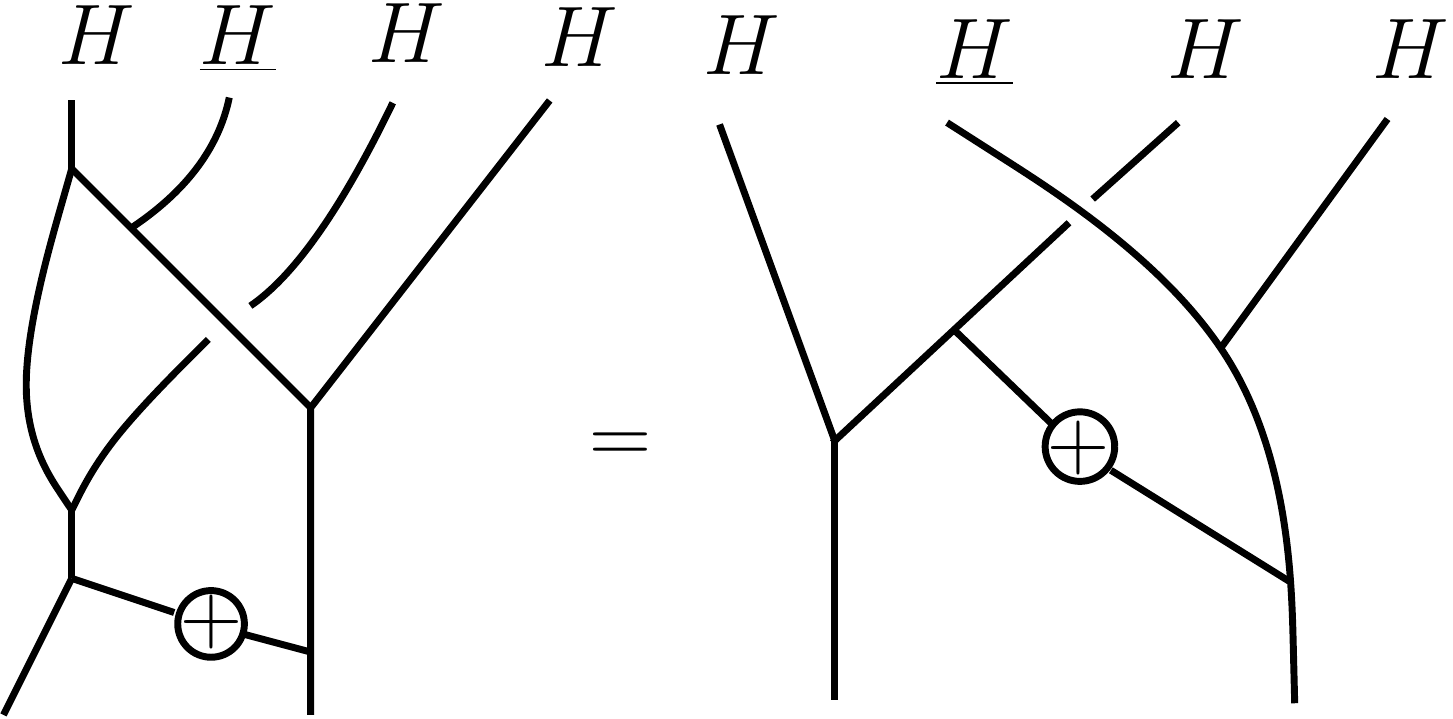}
\end{equation*}
\end{proof}

\begin{remark}
Note that \eqref{act} is a generalization of $x\ot y\cdot a\ot b= xa^1\ot S(a^2)yb$.
\end{remark}

\begin{corollary}\label{deltaright}
Let $M\in H_\Mc\mod$, then $\Delta_* M\simeq{}^*H\cdot M$ with the $H\ot^\tau H$-module structure as follows: \begin{equation*}
\includegraphics[height=.95in]{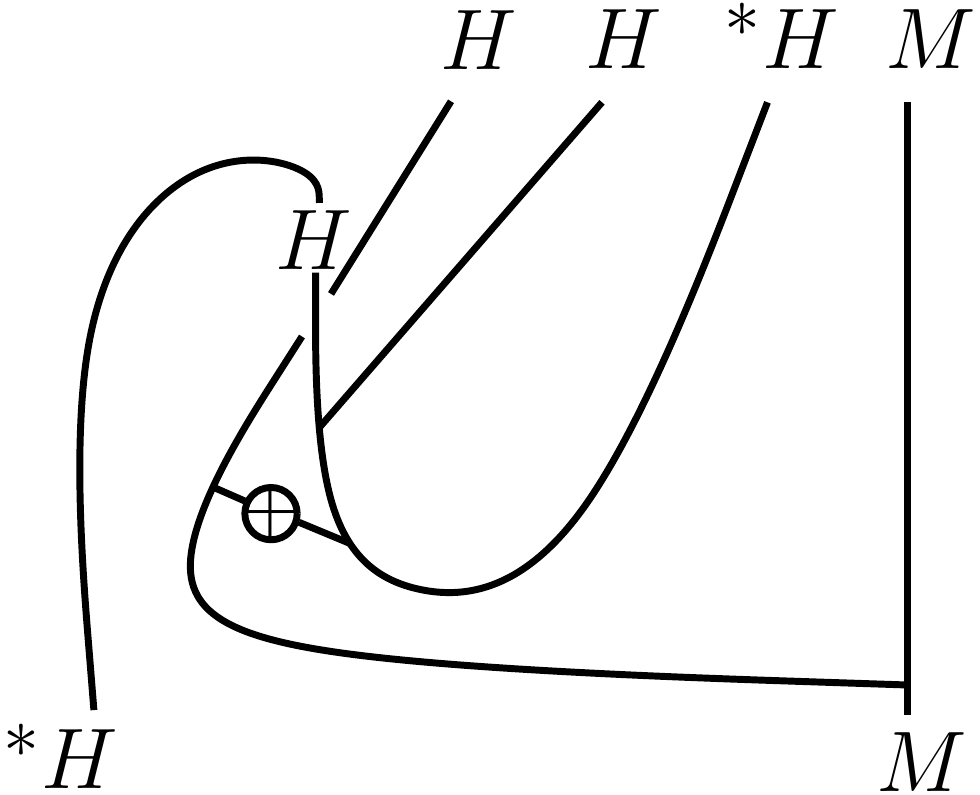} 
\end{equation*} which generalizes $a\ot b\cdot \chi\ot m=\chi(S(a^2)-b)\ot a^1m$.
\end{corollary}
\begin{proof}
	We use Lemma \ref{simplifylem} to identify $H\ot H$ with $H\ot \underline{H}$ as $(H, H\ot^\tau H)$-bimodules.  This is followed by Lemma \ref{hom:lem} with $A=H$, $X=\underline{H}$, and $B=H\ot^\tau H$.  Note that $L,L'$ are trivial here. 
\end{proof}

\begin{proposition}\label{prop:delta}
The functor $d_{i*}:\Cc_{n-1}\to \Cc_n$ for $0\leq i\leq n-1$ is as follows: for $M\in H^i\ot^\tau H\ot^\tau H^{n-1-i}_\Mc\mod$, $$d_{i*}(M)={}^*H\cdot M$$ with the $H^i\ot^\tau H\ot^\tau H\ot^\tau H^{n-1-i}$-module structure given by Lemma \ref{hom:lem} with $L=H^i$, $B=H\ot^\tau H$, and $L'=H^{n-1-i}$.  The action of $H\ot^\tau H$ on ${}^*H\cdot M$ is given by Corollary \ref{deltaright}.
\end{proposition}

\begin{remark}\label{adjunctiondelta}
Explicitly, for $M\in H\ot^\tau H_\Mc\mod$, we have $M\to \Delta_*\Delta^*M$ given by \begin{equation}\label{adjdelta}
\includegraphics[height=.7in]{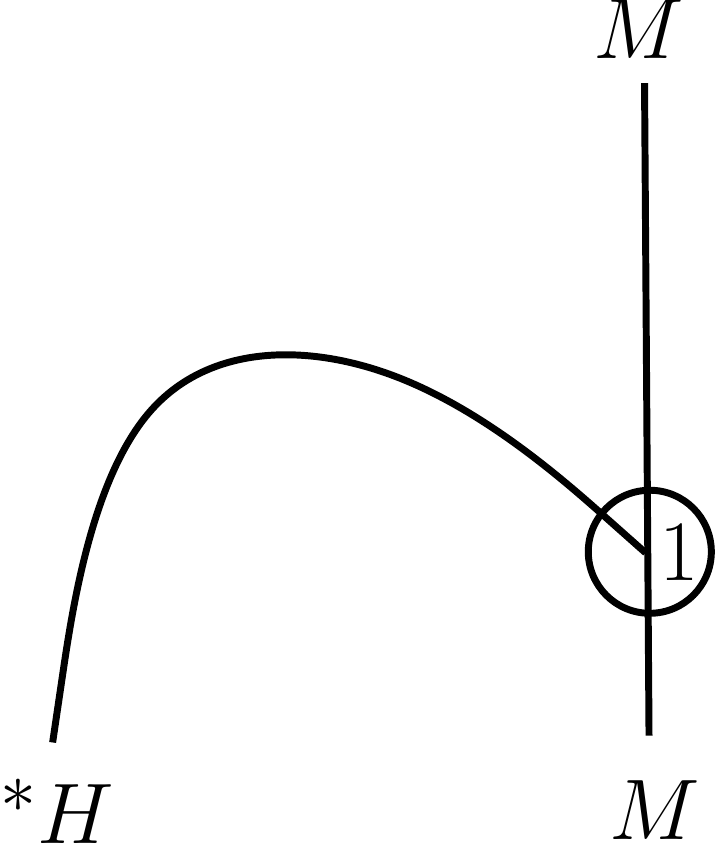}
\end{equation} where $1$ signifies the use of the second $H$-action on $M$.  Note that $0$ is the first action and corresponds to $Id\to d_{1*}d^*_1$.  The map \eqref{adjdelta} generalizes $M\to Hom(H,M)$ with $m\mapsto \phi_m(h)=(1\ot h)\cdot m$.  On the other hand for $M\in H_\Mc\mod$, we have $\Delta^*\Delta_*M\to M$ given by ${}^*1\ot Id_M: {}^*H\cdot M\to M$ as usual.
\end{remark}

On to the $\epsilon_*$; again by \eqref{adjunction} we have, for $M\in H_\Mc\mod$, that $$\epsilon_*(M)=M\ra_H 1=:M^{inv\, H}\in\Mc$$ where the left $H$-action on $1$ is via $\epsilon$.  More precisely, $M^{inv\, H}$ is the equalizer of the two maps $M\to M\ra H={}^*H\cdot M$, namely, ${}^*\epsilon\ot Id_M$ and  the adjoint of the action map $H\cdot M\to M$.  Recall \eqref{tt} that $\tau^*_n=\sigma^\Mc_{H^n,H}$ and so we will use its inverse for $\tau_{n*}$.  

The above demonstrates that the diagram of categories $\Cc_n$ associated to a Hopf algebra $H$ in a rigid braided $\Bc$, and a stable braided module category $\Mc$ involves functors that possess right adjoints.  Considering the adjoints we obtain a dual diagram whose inverse limit computes the direct limit of the original.  The inverse limit of the dual diagram is easy to describe;  note that the simplex category $\Delta$ is a subcategory of $\Lambda$ and we will begin with describing the inverse limit over it. 

Suppose that $M_n\in\Cc_n$ is an object in the inverse limit, then let $M:=M_0\in\Cc_0$, and without loss of generality we may assume that $M_n=d_{n*}d_{n-1 *}\cdots d_{1*}M$.  As part of the structure, see Section \ref{lim:sec}, we have an isomorphism $g_{d_0}:d_{0*}M_0\to M_1$ in $\Cc_1$.  Let $\gamma=g_{d_0}$ so that \begin{equation}\label{actionstr}\gamma:d_{0*}M\to d_{1*}M\end{equation} is the only    structure needed to define an object in the inverse limit.  It is necessary and sufficient that it satisfy two conditions:
the unit condition in $\Cc_0$: \begin{equation}\label{unitalityofact}s_{0*}(\gamma)=Id_M,\end{equation} and the associativity condition in $\Cc_2$: \begin{equation}\label{associativityofstr}d_{2*}(\gamma)d_{0*}(\gamma)=d_{1*}(\gamma).\end{equation} 

To obtain the inverse limit over $\Lambda$ itself, we need one more condition (see \eqref{center}) in $\Cc_0$: \begin{equation}\label{stabilitystr}s_{-1*}(\gamma)=(\varsigma_M)^{-1}.\end{equation}  Note that the last condition can not be formulated unless $\Mc$ is stable; the source of the maps is $M$ and the target is  $(\sigma^\Mc_{1,H})^{-1}(M)=\tau_{0*}M=M^H$.

\subsection{The monad}
Following \cite{chern} we can describe the inverse limit in terms of a monad on $\Cc_0$.  More precisely, the data of \eqref{actionstr} is encoded via adjunction into an action:
 \begin{equation}\label{monactstr}a:\Ac(M):=d^*_1d_{0*}M=\Delta^*\sigma^\Mc_{H,H}\Delta_*M\to M.\end{equation} 
 
 \begin{remark}
 We recall that for a braided $\Bc$-module $\Mc$ with $\Cc=H_\Bc\mod$, and $\Ac$ as above we have $$HH_\Mc(\Cc)=\Ac_{\Cc_0}\mod$$ and if $\Mc$ is stable then $HC_\Mc(\Cc)$ is the full subcategory of $\Ac_{\Cc_0}\mod$ consisting of $M$ with the trivial $\mathfrak{z}\in\Zc(\Ac)$ action.    The unit and $\mathfrak{z}$ in $\Ac$ are described in \eqref{u:eq} and \eqref{sigmaact:eq} respectively.
 \end{remark}

 We begin by describing $\Ac$ explicitly as an endofunctor of $\Cc_0=H_\Mc\mod$.

\begin{lemma}\label{monadact} Let $M\in H_\Mc\mod$, then the action of $H$ on $\Ac(M)\simeq{}^*H\cdot M$ is as follows:
\begin{equation}\label{monact}
\includegraphics[height=1.1in]{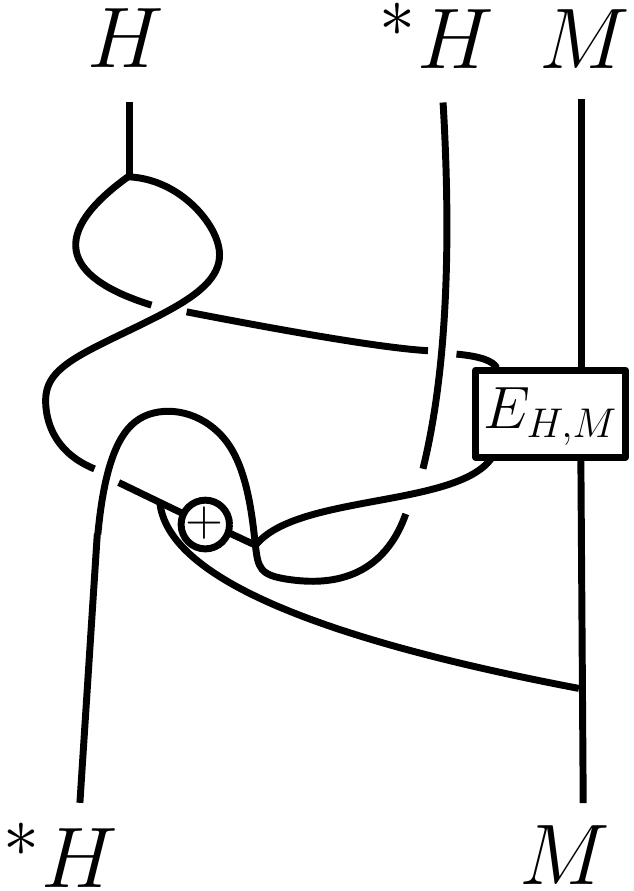}
\end{equation} which generalizes $a\cdot\chi\ot m=\chi(S(a^3)-a^1)\ot a^2m$.
\end{lemma}

\begin{proof}
Starting with Corollary  \ref{deltaright} that describes $\Delta_*M$ as an $H\ot^\tau H$-module, we apply $\sigma^\Mc_{H,H}$.  Using naturality of $E_{-,-}$ in the second component and property $C1$, we obtain, essentially, the result after simplifying.  Following up with $\Delta^*$ finishes the proof.
\end{proof}

On to the monadic versions of \eqref{unitalityofact} and \eqref{stabilitystr}.  Namely, consider the natural maps \begin{equation}\label{u:eq}u:Id=\Delta^*\sigma^\Mc_{H,H}s_0^*s_{0*}\Delta_*\to\Ac\end{equation} and $$\varsigma':\sigma^\Mc_{1,H}=\tau_0^*=\Delta^*\sigma^\Mc_{H,H}s_{-1}^*s_{-1*}\Delta_*\to\Ac;$$ let \begin{equation}\label{sigmaact:eq}\mathfrak{z}=\varsigma'\circ\varsigma: Id\to\Ac,\end{equation}  then  \eqref{unitalityofact} is equivalent to $$(a\circ u)_M=Id_M$$ and \eqref{stabilitystr}  is equivalent to $$(a\circ \mathfrak{z})_M=Id_M.$$  In our particular case we have:

\begin{lemma}\label{unitcentral}
Let $M\in H_\Mc\mod$, then $$u_M={}^*\epsilon\ot Id_M:M\to\Ac(M)\simeq{}^*H\cdot M$$  and  \begin{equation}\label{sigma:equ}
\includegraphics[height=.7in]{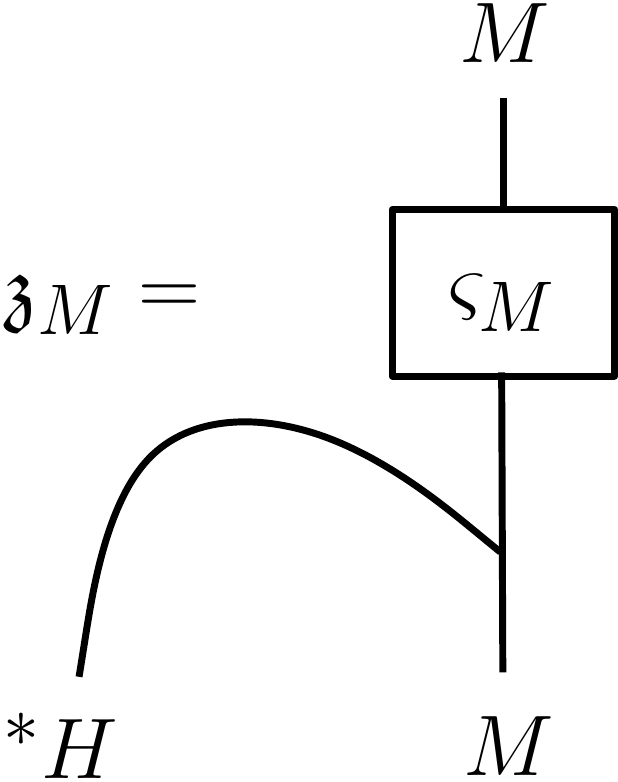}
\end{equation} where the former is exactly as usual, i.e., $M\to Hom(H,M)$ with $m\mapsto \epsilon(-)m$, and the latter generalizes  $m\mapsto (-)m$.
\end{lemma}

\begin{lemma}\label{multiplication}
The multiplication $m:\Ac^2\to\Ac$ is as follows: let $M\in H_\Mc\mod$ then $m:{}^*H \cdot{}^*H\cdot M\to {}^*H\cdot M$ is given by  \begin{equation}\label{hstarmult}
\includegraphics[height=.7in]{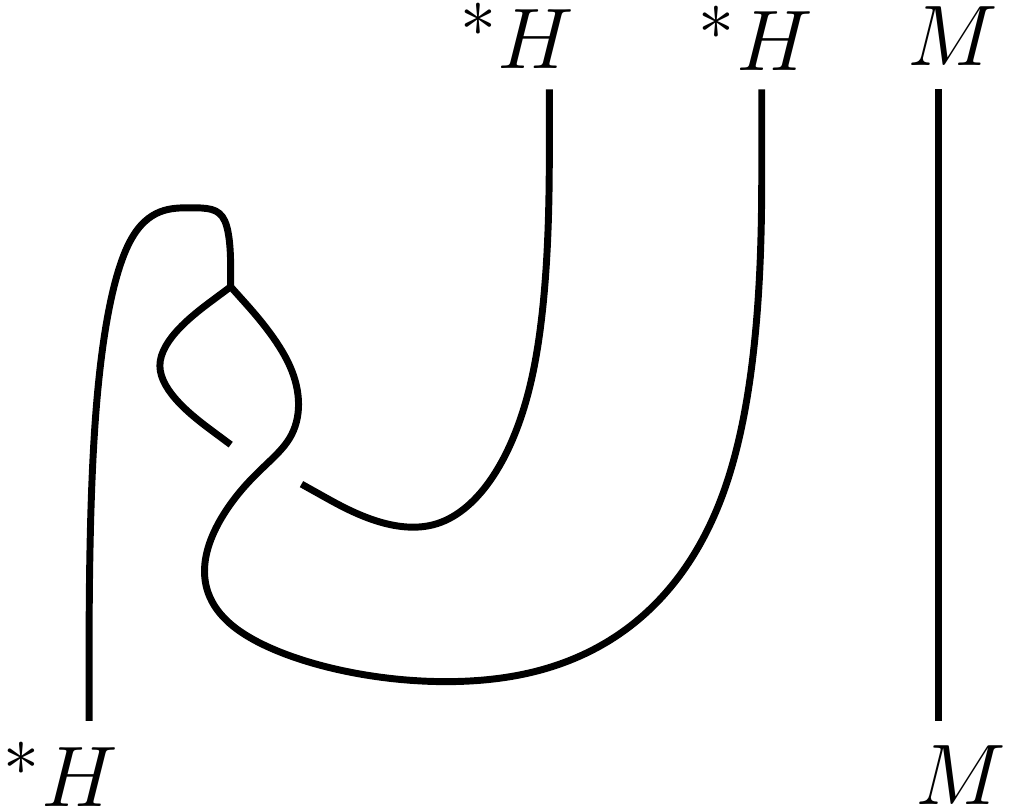}
\end{equation}
\end{lemma}

\begin{proof}
Recall that the product on $\Ac$ is obtained via adjunction from \begin{equation}\label{big}d_{0*}d^*_1d_{0*}\leftarrow d^*_2d_{0*}d_{0*}\simeq d^*_2d_{1*}d_{0*}\to d_{1*}d^*_1d_{0*}\end{equation} if the first arrow is invertible. In our case the map $d^*_2d_{0*}\to d_{0*}d^*_1$ is equality, and thus, so is our first arrow.  The associativity isomorphism becomes non-trivial under our identifications that produced $\Ac(M)={}^*H\cdot M$.  More precisely, we have:  \begin{equation*}\label{assoc}
\includegraphics[height=.8in]{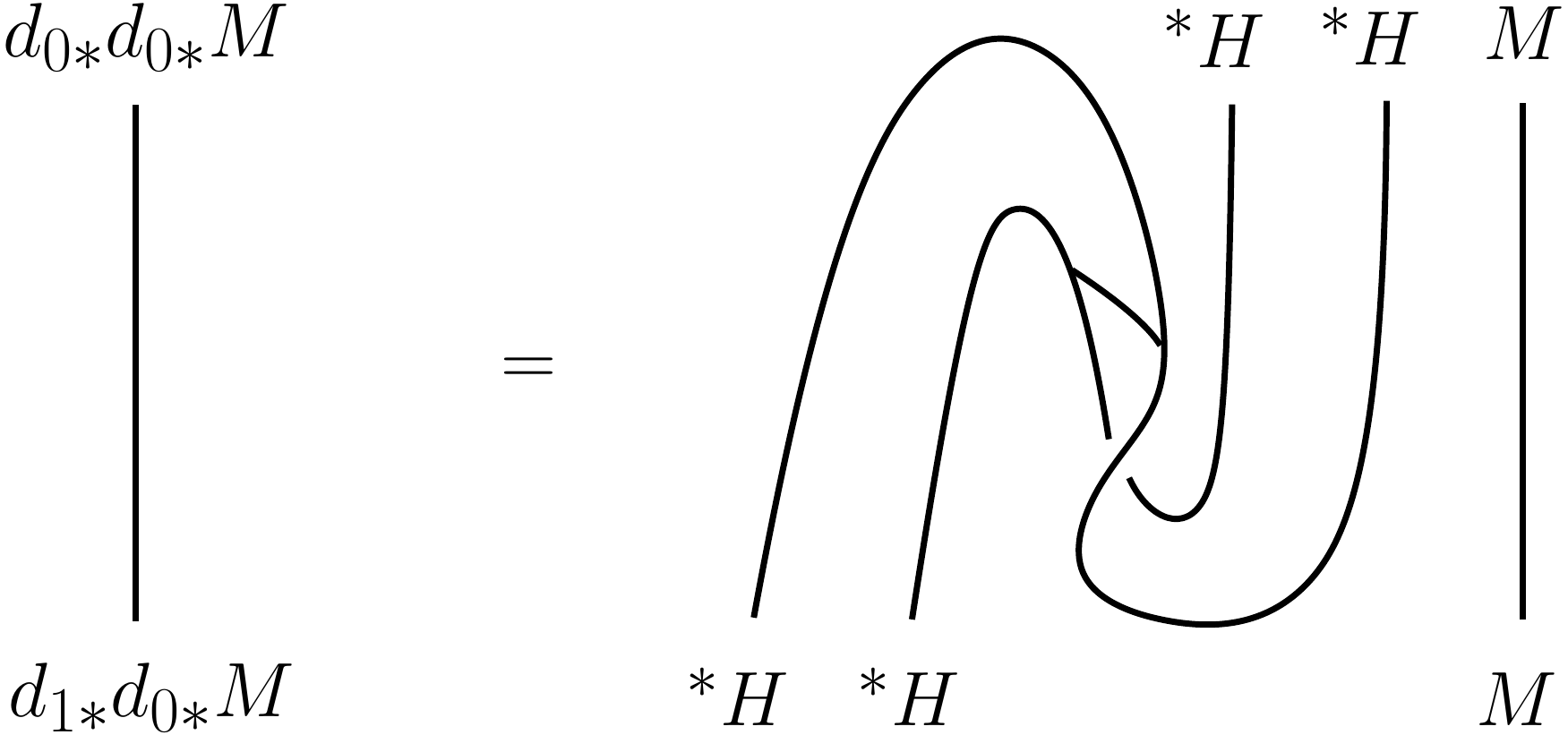}
\end{equation*} The last arrow is complicated: \begin{equation*}
\includegraphics[height=.8in]{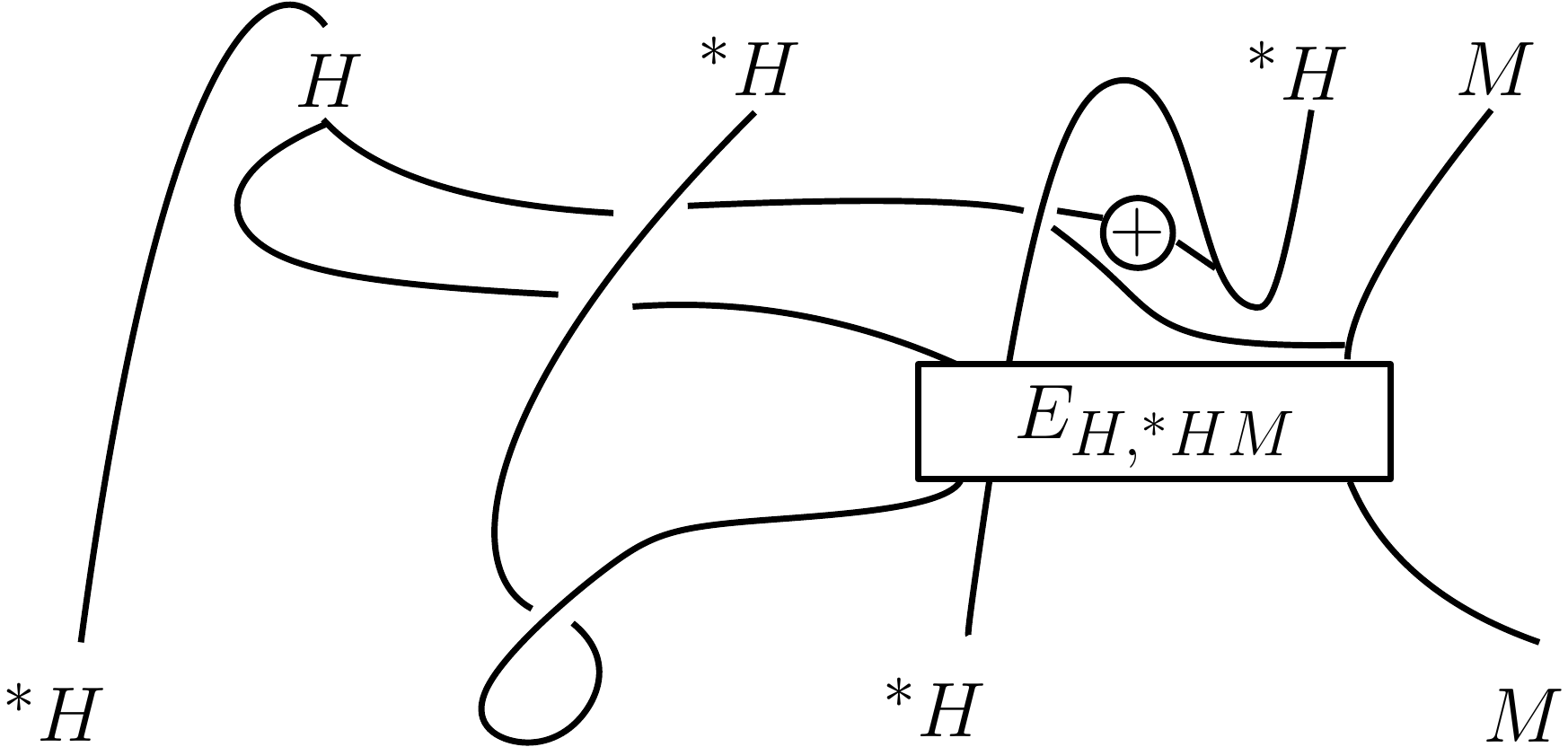}
\end{equation*} but as we do not need \eqref{big} but rather its adjoint that yields $d^*_1d_{0*}d^*_1d_{0*}\to d^*_1d_{0*}$, we need to apply ${}^*\epsilon$ to the left-most ${}^*H$ in the above.  This results in the complicated diagram turning into identity on the ${}^*H\cdot M$, with ${}^*\epsilon$ on the left-most ${}^*H$.  Which turns the associativity diagram into what we claimed.
\end{proof}

\begin{lemma}\label{centralact}
Let $M\in H_\Mc\mod$, then $\Ac(M)$ is a free $\Ac$-module and the action of $\mathfrak{z}\in\Ac$   on $\Ac(M)$ is given  by \begin{equation}
\includegraphics[height=.8in]{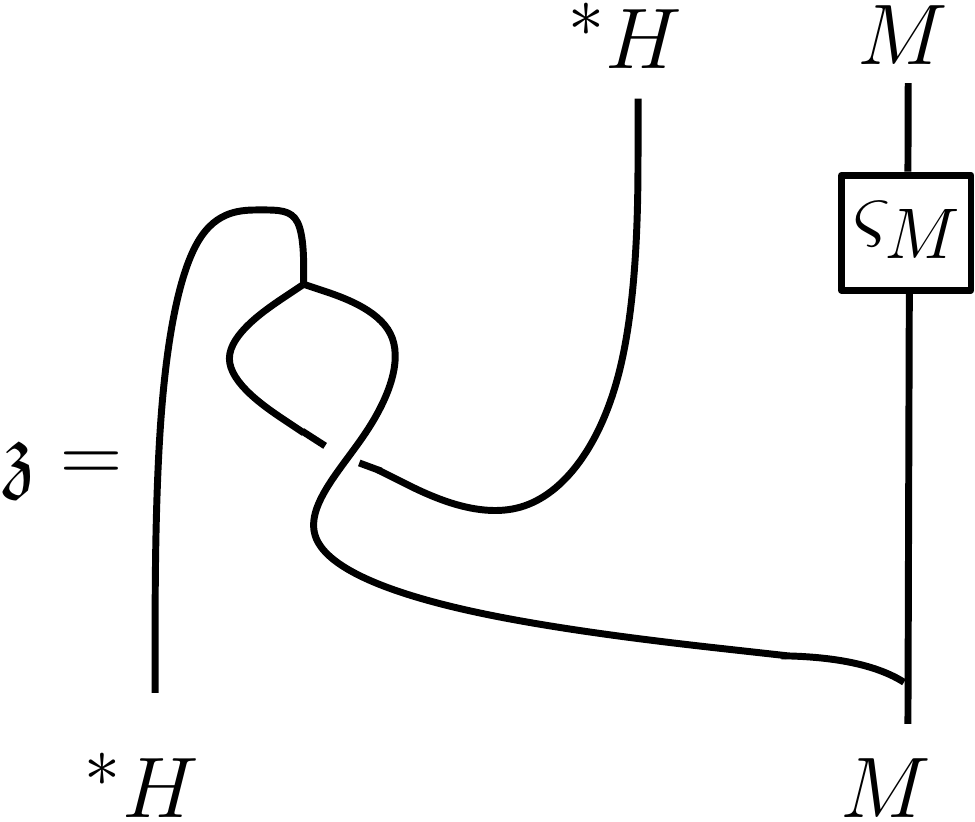}
\end{equation} which generalizes $Hom(H,M)\to Hom(H,M)$ with $\phi(-)\mapsto (-)^2\phi((-)^1)$.
\end{lemma}
\begin{proof}
Combine Lemmas \ref{monadact}, \ref{unitcentral}, \ref{multiplication}, and simplify.
\end{proof}

\subsection{The module/comodule description of $HH_\Mc(\Cc)$ and $HC_\Mc(\Cc)$}
We can now describe $\Ac_{\Cc_0}\mod$  more in line with the usual definition of anti-Yetter-Drinfeld modules.   

\begin{definition}\label{alg:def:hstr}
Let $H$ be a Hopf algebra in $\Bc$ and $\Mc$ a stable braided $\Bc$-module. \begin{itemize}
\item Define the algebra structure on ${}^*H$ by $u={}^*\epsilon: 1\to {}^*H$ and $m:{}^*H\ot{}^*H\to {}^*H$ given by \eqref{hstarmult}.
\item For $M\in H_\Mc\mod$ and $M\in {}^*H_\Mc\mod$ define \begin{equation}\label{varsigma:eq}\varsigma^H_M=\rho_{{}^*\!H,M}\circ\rho_{H,M}\circ coev_{{}^*H,H}\circ\varsigma_M\end{equation}  Note that if the stable structure $\varsigma$ is rescaled, then $\varsigma^H$ is similarly rescaled.
\end{itemize}
\end{definition}

\begin{definition}\label{def:saYD}
Suppose that $\Mc$ is a braided $\Bc$-module. Let  $aYD^H_\Mc$ denote the category of $M\in\Mc$ such that $$M\in H_\Mc\mod\quad\text{and}\quad M\in {}^*H_\Mc\mod$$ with morphisms compatible with both structures.  The two actions are compatible as follows: \begin{equation}\label{main:eq}
\includegraphics[height=1in]{compatibilityandact.pdf}
\end{equation} If $\Mc$ is stable, let $saYD^H_\Mc$ denote the full subcategory of $aYD^H_\Mc$ with $M$ such that  $$\varsigma^H_M=Id_M.$$  Thus, $saYD^H_\Mc=(aYD^H_\Mc)^{\varsigma^H_M}$. 
\end{definition}

The following is immediate from the interpretation  of the limits in Definition \ref{HHHC:def} as modules over a monad $\Ac$ and the subsequent description of $\Ac$ in the above.

\begin{theorem}\label{thm:aYDHH}
Let $H\in\Bc$ be a Hopf algebra and $\mathcal{M}$ be a  braided $\Bc$-module.  Then $$HH_{\mathcal{M}}(H_{\mathcal{B}}\mod)\simeq aYD^H_\Mc$$  and if $\Mc$ is stable then $$HC_{\mathcal{M}}(H_{\mathcal{B}}\mod)\simeq saYD^H_\Mc.$$ 
\end{theorem}

Let us compare our Definition \ref{def:saYD} to the special case found in \cite{balanced}.  More precisely, we recall the definition of a $\theta$-twisted pair in  involution.  Note that we dropped the modular part ($\delta(\sigma)=1$), as it is not general enough.

\begin{definition}
Let $\theta$ be a twist on $\Bc$, $\delta$ a character of $H$, and $\sigma$ a group-like in $H$.  The pair $(\delta,\sigma)$ is a $\theta$-twisted pair in  involution if  \begin{equation*}
\includegraphics[height=1in]{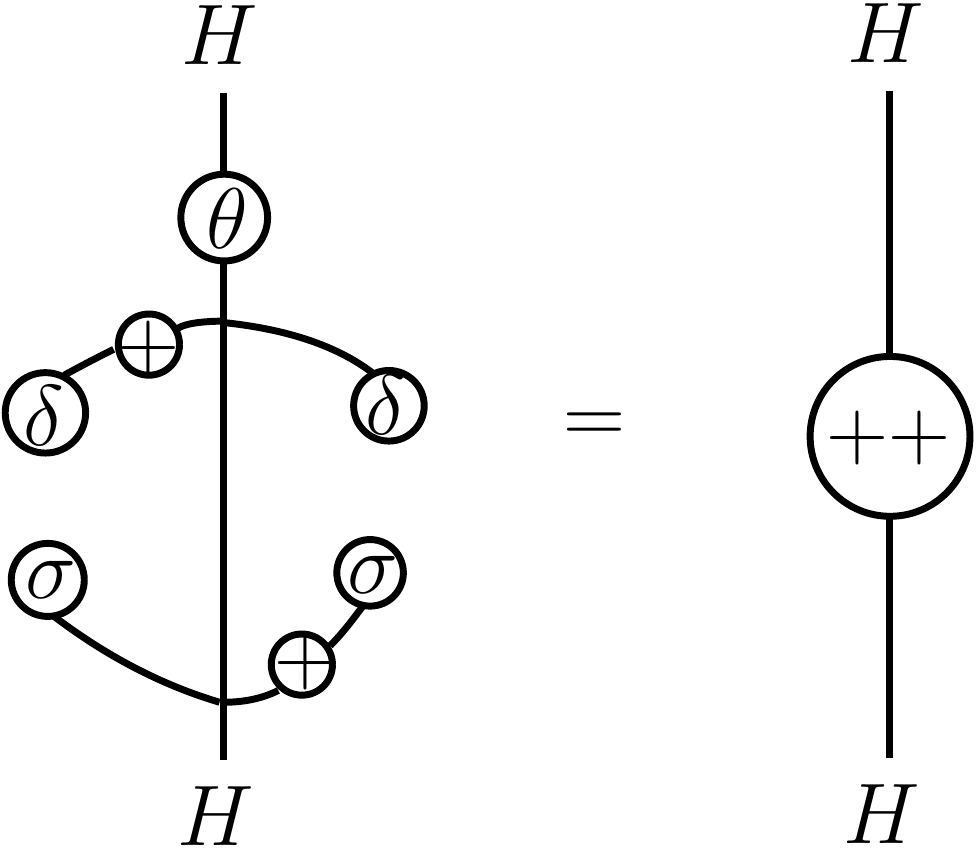}
\end{equation*}
\end{definition}

We will now consider ``1-dimensional" elements in $HH_{\Bc_\varsigma}(H_\Bc\mod)$ and their stable versions.  Namely, let $x\in\Bc^\times$ have the additional structure of an anti-Yetter-Drinfeld module with coefficients in $\Bc_\varsigma$, i.e., $x\in HH_{\Bc_\varsigma}(H_\Bc\mod)$.  Recall $\varsigma^\dagger$ from Remark \ref{pair:rem} and $\phi^x$ from \eqref{phidef}.

\begin{lemma}\label{lem:modpair}
A ``1-dimensional" $x\in HH_{\Bc_\varsigma}(H_\Bc\mod)$ corresponds to a triple $(x,\delta,\sigma)$ such that $x\in\Bc^\times$, $\delta$ is a character of $H$, $\sigma$ is a group-like in $H$ and $(\delta,\sigma)$ is a $\theta$-twisted pair in  involution, where $$\theta=(\varsigma^\dagger \phi^{x})^{-1}.$$  This element is stable if $$\delta(\sigma)=\varsigma_x^{-1}.$$ 
\end{lemma}
\begin{proof}
We obtain $\delta$ from the $H$-action and $\sigma$ from ${}^*H$-action.  Stability translates to $\delta(\sigma)\varsigma_x=1$.  Using \eqref{main:eq} we get, after simplifying, \begin{equation*}
\includegraphics[height=1in]{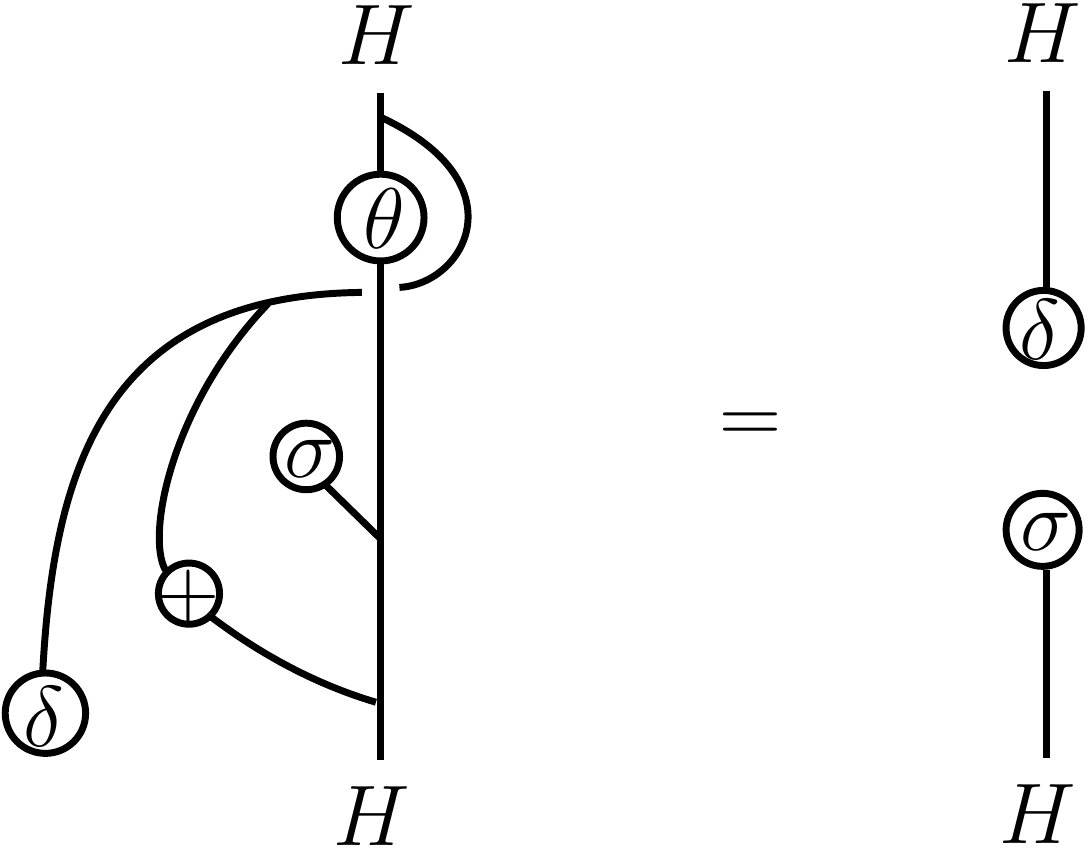}
\end{equation*}  Applying \begin{equation*}
\includegraphics[height=1in]{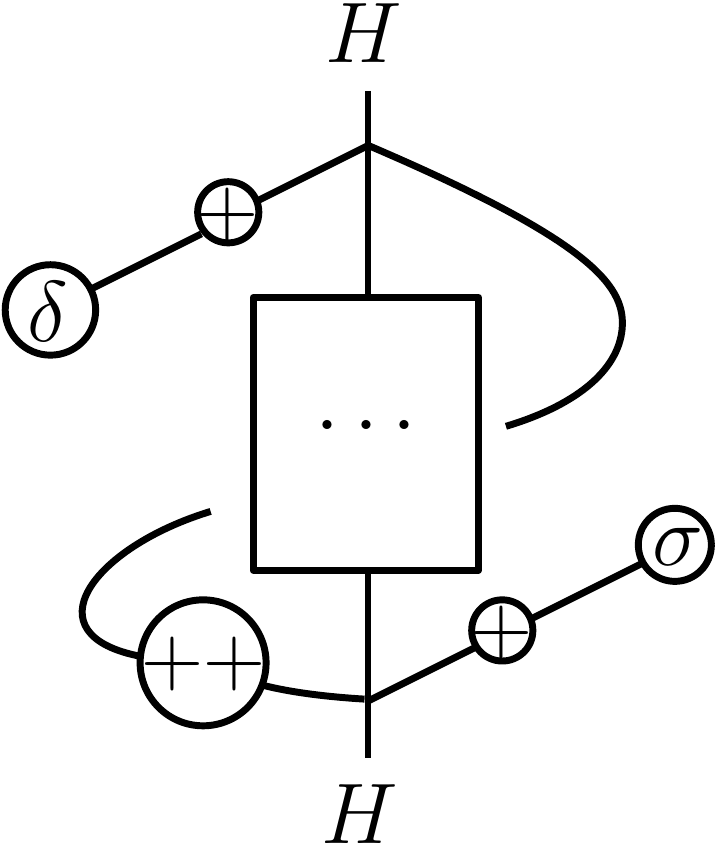}
\end{equation*} to both sides yields the result. 
\end{proof}
\begin{remark}
The original definition of a $\theta$-twisted modular pair in  involution thus covers the case of $x=1$.  If we drop the stability condition it does cover an arbitrary case as well since $x\in\Bc_\varsigma$ corresponds to $1\in\Bc_{\varsigma\phi^{x^{-1}}}$.  Unless $\varsigma_x=1$, the original definition misses the stable ``1-dimensional" elements based on $x\in\Bc^\times$. 
\end{remark}

\subsection{Cohomology of algebras from $M\in HC_\Mc(H_\Bc\mod)$ and $K\in \Mc^\varsigma$}\label{sec:frommtocoh}

The main use of coefficients is that they yield cohomology theories for algebras.  Let $A$ be an algebra in $H_\Bc\mod$.  In general,  $M\in HC_\Mc(H_\Bc\mod)$ will not yield a cohomology of $A$;  it does if, for example, $\Mc=HH(\Bc)$.   To obtain cohomology in all cases we also require a $K\in \Mc^\varsigma$.   While we do not focus on cohomology theories for algebras in this paper, the considerations presented here give us Proposition \ref{half:prop}, which is one half of the main result.

\begin{lemma}\label{isolemma}
Let $A,B\in Alg(\Bc)$, $\Mc$ a braided $\Bc$-module. Let  $V\in A_\Bc\mod$, $W\in B_\Bc\mod$, and $K\in \Mc$ then we have a canonical isomorphism in $B\ot^\tau A_\Mc\mod$: $$\sigma^K_{W,V}:=E_{W,K}\tau^{-1}_{W,V}:(W\ot^\tau V)\cdot K\to \sigma^\Mc_{A,B}((V\ot^\tau W)\cdot K).$$  Furthermore, if $\Mc$ is stable  and $K\in\Mc^\varsigma$ then $\sigma^K_{V,1}=\varsigma_{V\cdot K}$. 
\end{lemma}

\begin{proof}
It is clear that both sides are $B\ot^\tau A$-modules in $\Mc$ so it suffices to compare the $A$ and $B$ actions.  For $A$-action we have:\begin{equation*}
\includegraphics[height=.7in]{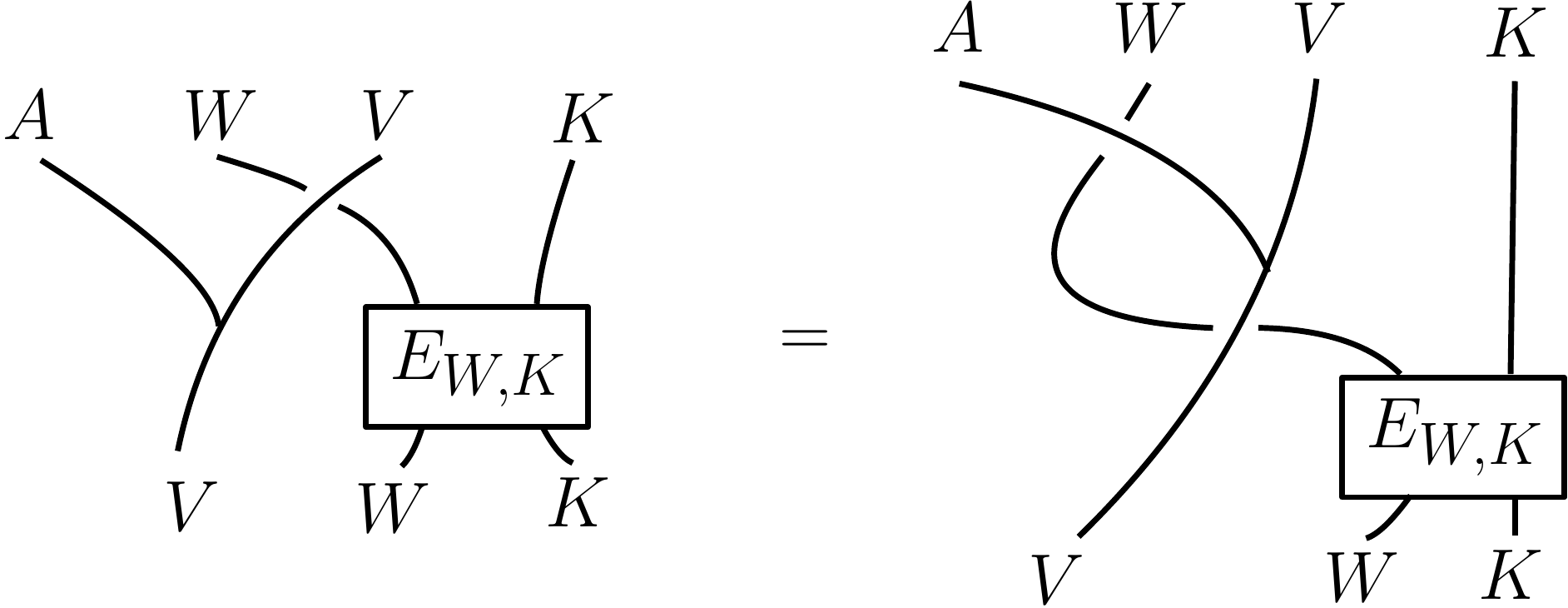}
\end{equation*}

For $B$-action we need to check that:\begin{equation*}
\includegraphics[height=.8in]{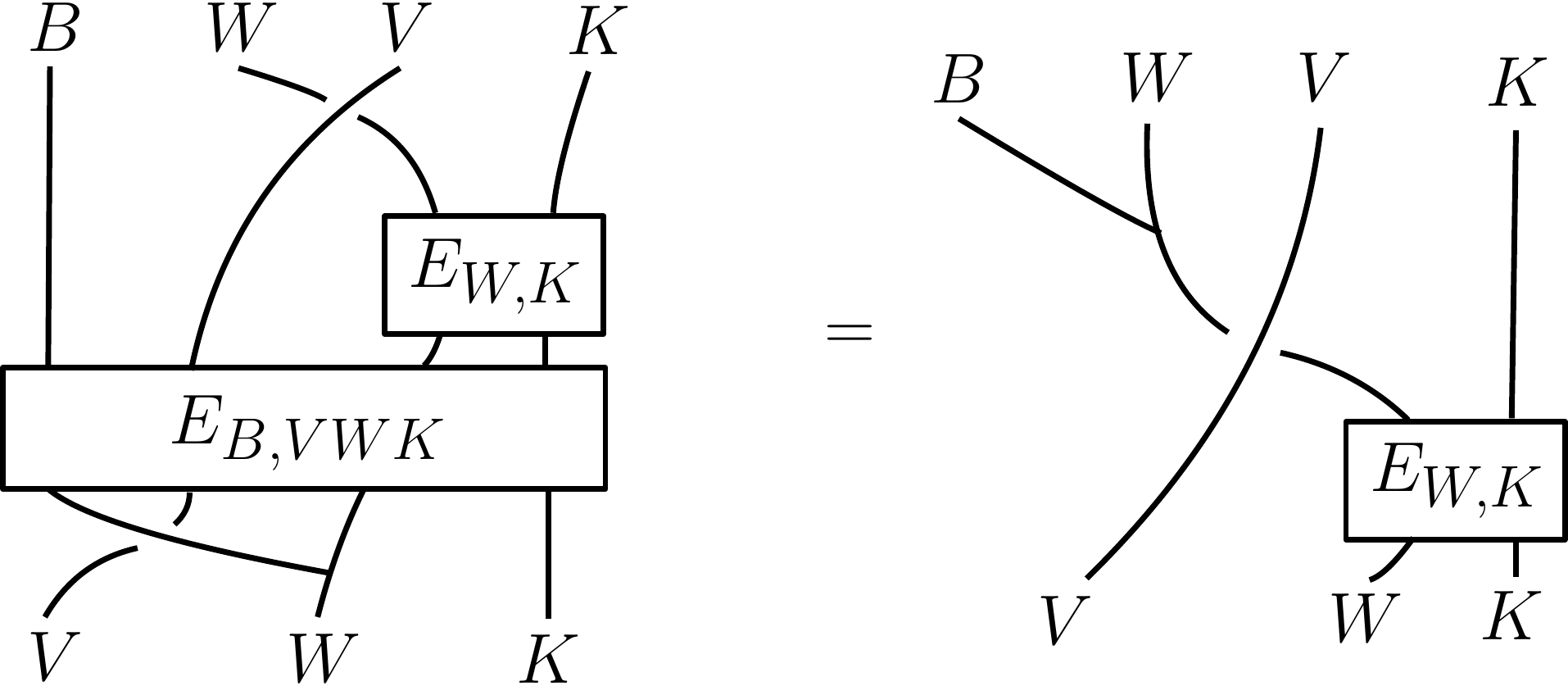}
\end{equation*} but the LHS is \begin{align*}
\rho_{B,W}&\tau_{B,V}E_{B,VWK}E_{W,K}\tau^{-1}_{W,V}=\rho_{B,W}\tau_{B,V}\tau^{-1}_{V,B}E_{B,WK}\tau^{-1}_{B,V}E_{W,K}\tau^{-1}_{W,V}\\
&=\rho_{B,W}E_{B,WK}E_{W,K}\tau^{-1}_{B,V}\tau^{-1}_{W,V}=\rho_{B,W}E_{BW,K}\tau^{-1}_{BW,V}=E_{W,K}\tau^{-1}_{W,V}\rho_{B,W}
\end{align*} which is the RHS.  For the last statement: $\sigma^K_{V,1}=E_{V,K}=\varsigma_{VK}\varsigma^{-1}_K=\varsigma_{VK}$.
\end{proof}

Let $M\in\Cc_1$, i.e., $M\in (H_0\ot^\tau {H_1})_\Mc\mod$, we have a natural isomorphism \begin{equation}\label{twist0}\sigma_M^\Ac:\Ac(\Delta^*\tau_1^*(M))\to \Ac(\Delta^* (M))\end{equation} in $\Ac_{\Cc_0}\mod$ that arises automatically by adjunctions, i.e., we follow the $Id$ through the sequence of isomorphisms below:\begin{align*}H&om_{\Ac}(\Ac(d^*_0 M),\Ac(d^*_0 M))\simeq Hom_{\Cc_0}(d^*_0 M,\Ac(d^*_0 M))\simeq Hom_{\Cc_1}(M,d_{0*}\Ac(d^*_0 M))\\
&\to Hom_{\Cc_1}(M,d_{1*}\Ac(d^*_0 M))\simeq Hom_{\Cc_0}(d^*_1 M,\Ac(d^*_0 M))\simeq Hom_{\Ac}(\Ac(d^*_1 M),\Ac(d^*_0 M))
\end{align*}

Explicitly:
\begin{equation}\label{tw1}
\includegraphics[height=.8in]{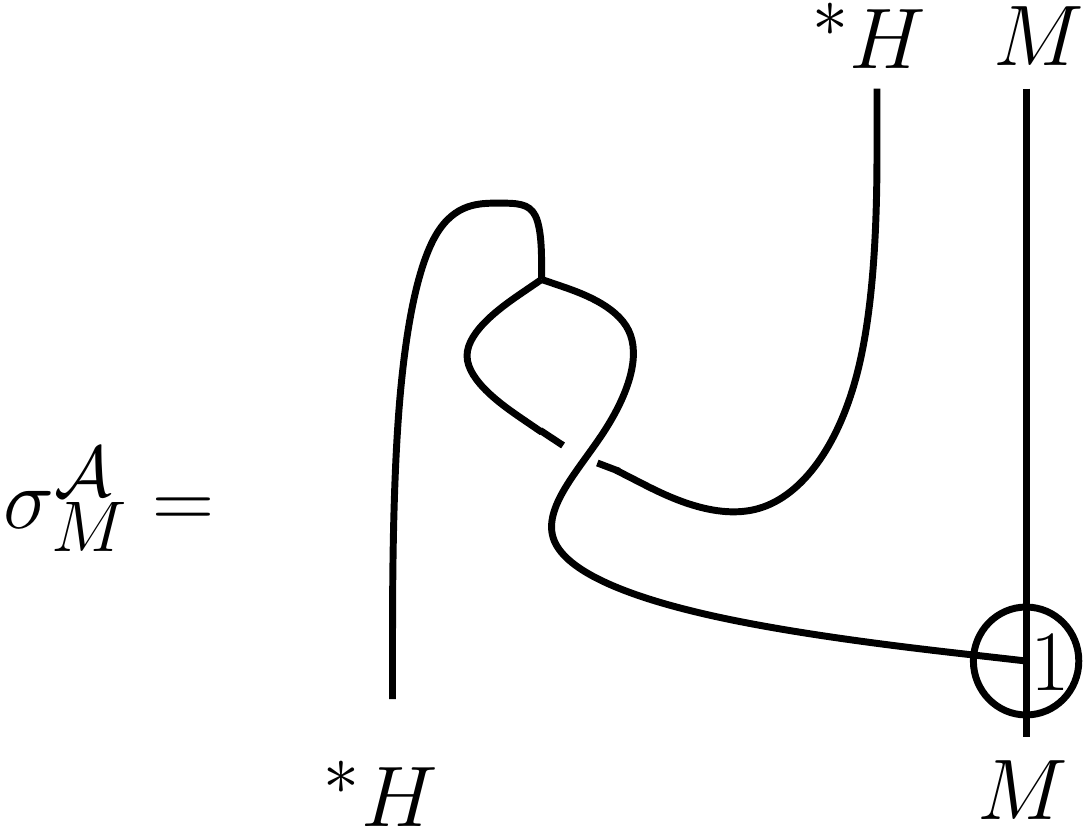} 
\end{equation} where $1$ indicates the use of the second $H$-action.

\begin{definition}\label{tracestructure}
Define for $V,W\in H_\Bc\mod$ and $K\in \Mc$ an isomorphism in $\Ac_{\Cc_0}\mod$: \begin{equation}\label{theflip}
\iota^K_{V,W}:=\sigma^\Ac_{W\ot^\tau V\cdot K}\circ \Ac\Delta^*(\sigma^K_{V,W}):\Ac(\Delta^*(V\ot^\tau W)\cdot K)\to \Ac(\Delta^*(W\ot^\tau V)\cdot K)
\end{equation} \end{definition}

\begin{lemma}\label{iota:lem}
If $K\in\Mc$ is stable, then the isomorphism  $\iota^K$ equips the functor $$\Ac^K:=\Ac(-\cdot K): H_\Bc\mod\to \Ac_{\Cc_0}\mod$$ with a  $\mathfrak{z}$-trace structure.  
\end{lemma}
\begin{proof}
Namely, for $V,W,T\in H_\Bc\mod$ we need (recall \eqref{productinh}): $$\iota^K_{V\pr W, T}=\iota^K_{W, T\pr V}\iota^K_{V,W\pr T}$$ as isomorphisms from $\Ac^K(V\pr W\pr T)$ to $\Ac^K(T\pr V\pr W)$.  Note that the RHS is \begin{equation*}
\includegraphics[height=1in]{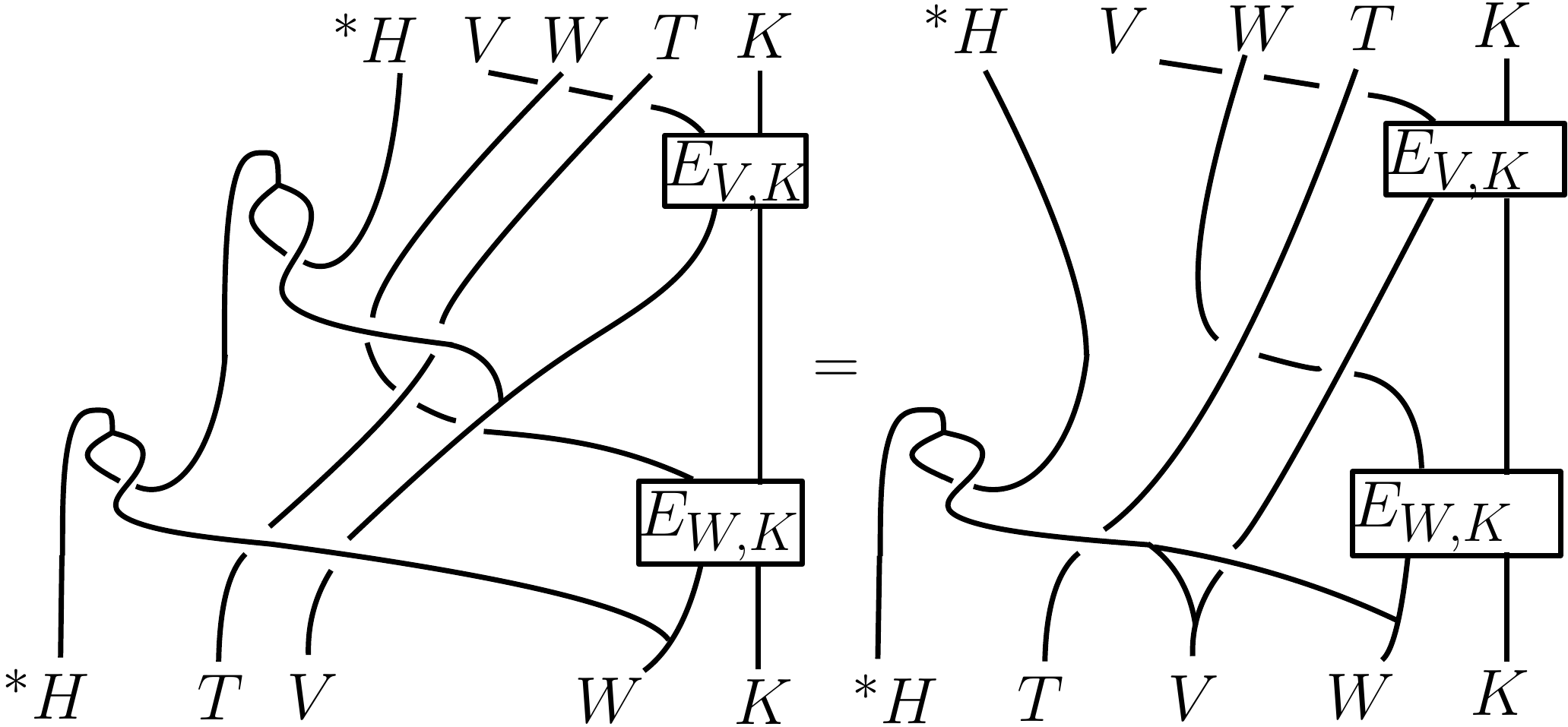} 
\end{equation*} while the LHS is \begin{equation*}
\includegraphics[height=1in]{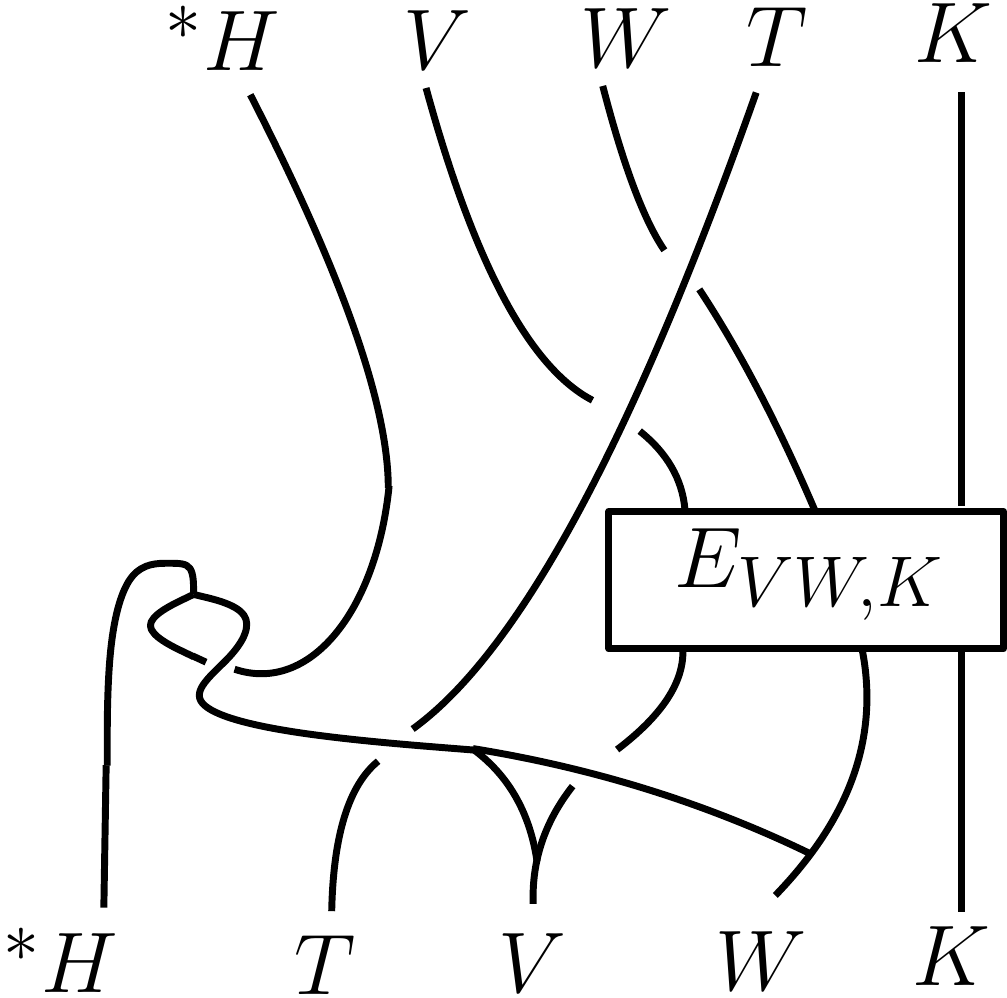} 
\end{equation*} so that it suffices to show that $E_{W,K}\tau^{-1}_{W,V}E_{V,K}\tau^{-1}_{V,W}=E_{W,K}E_{V,WK}=E_{VW,K}$.

It remains to show that  $$\iota^K_{V,1}:\Ac^K(V)\to \Ac^K(V)=\text{the action of } \mathfrak{z}\in\Ac$$ which follows from Lemmas \ref{centralact} and \ref{isolemma}.  Note that it is only the last statement that requires the stability of $K$. 
\end{proof}

\begin{proposition}\label{prop:contra}Let $K\in\Mc^\varsigma$ and $M\in HC_\Mc(\Cc)$.  Let $$F^M_K(-)=Hom_{\Cc_0}(-\cdot K,M):H_\Bc\mod\to\vect.$$   Then $F^M_K$ is a symmetric contratrace with $\iota_{V,W}$ as follows:
\begin{equation}\label{contrafk}
\includegraphics[height=1.2in]{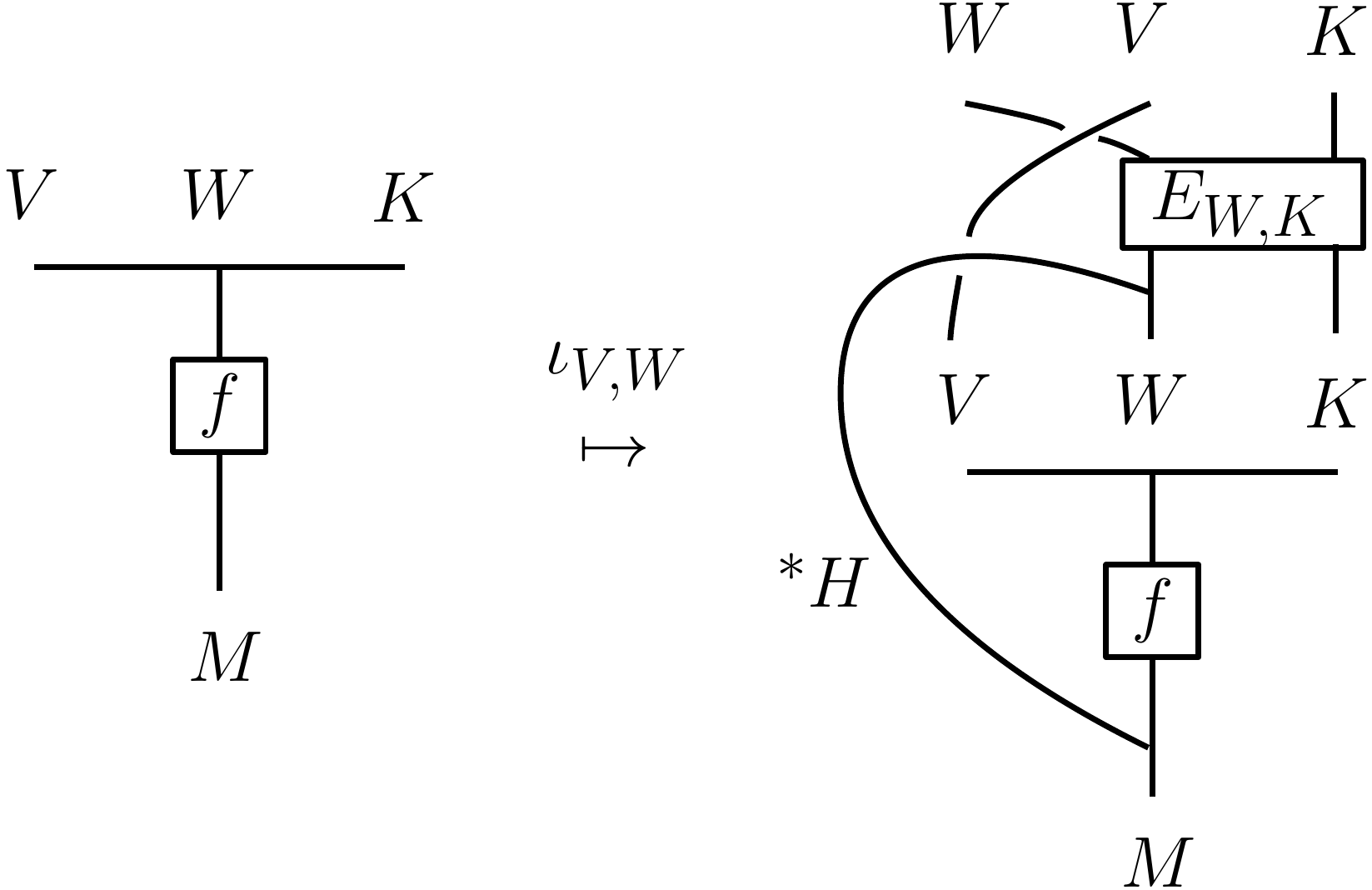}
\end{equation}Thus, $F^M_K(A^{\ot(\bullet+1)})$ is a cocyclic object in $\vect$.  If $M\in HH_\Mc(\Cc)$, then $F^M_K$ is a contratrace, and $F^M_K(A^{\ot(\bullet+1)})$ is a paracocyclic  object, with $\tau_n^{n+1}$ induced by $\varsigma^H_M$.
\end{proposition} 

\begin{proof}
We obtain the structure $\iota_{V,W}$ via the chain of isomorphisms: \begin{align*}
Hom_{\Cc_0}(W\pr V\cdot K,M)=&Hom_{\Ac_{\Cc_0}\mod}(\Ac(W\pr V\cdot K), M)\\ \stackrel{-\circ\iota^K_{V,W}}{\simeq} &Hom_{\Ac_{\Cc_0}\mod}(\Ac(V\pr W\cdot K), M)\\
=&Hom_{\Cc_0}(V\pr W\cdot K,M).
\end{align*}  The (symmetric) contratrace property of $\iota$ follows immediately from Lemma \ref{iota:lem}.
\end{proof}

\begin{corollary}\label{cor:contra}
Let $K\in\Mc^\varsigma$ and suppose that the functor $-\cdot K:\Bc\to\Mc$ above has a right adjoint $K\la -$.  Then \begin{equation}\label{themaphh}K\la -:HH_\Mc(\Cc)\to HH(\Cc)\end{equation} and the functor is compatible with $\varsigma^H$ on LHS and $\varsigma$ on RHS, so that, in particular, $K\la -:HC_\Mc(\Cc)\to HC(\Cc)$.
\end{corollary}

\begin{proof}
The existence of the right adjoint automatically implies that we have an adjoint pair: $-\cdot K:H_\Bc\mod\leftrightarrows H_\Mc\mod:K\la -$ which ensures that $F^M_K$ of Proposition \ref{prop:contra} is representable.  The corollary now follows by \cite{ks}.  See the discussion below for a summary.
\end{proof}

The correspondence between representable contratraces and objects in $HH(\Cc)$ can be summarized as follows.  Given $$\iota_{V,W}:Hom_H(V\ot W, M)\to Hom_H(W\ot V, M),$$ let $f=Id_M\ot ev_{W^*,W}: M\ot W^*\ot W\to M$, then  we obtain $\iota_{MW^*,W}(f): W\ot M\ot W^*\to M$ from which the anti-center structure on $M$ is obtained by adjunction: \begin{equation}\label{convert}\tau^\bullet_{M,W^*}: M\ot W^*\to {}^*W\ot M.\end{equation}  On the other hand, given $\tau^\bullet$, for $f\in Hom_H(V\ot W, M)$ define $$\iota_{V,W}(f)=ev_{W,{}^*W}\circ\tau^\bullet_{M,W^*}\circ f\circ coev_{W,W^*}.$$ 

Consider a particular case of Corollary \ref{cor:contra}, namely $\Mc=HH(\Bc)$ and $K=Tr(1)$; note that $K\in\Mc^\varsigma$ by \cite{chern}.  Recall that we have an adjoint pair of functors $$Tr:\Bc\to HH(\Bc): U$$ where $U$ forget the anti-center structure and $Tr$ is its left adjoint.  We would like an explicit description of \eqref{themaphh}.  This serves as one half of the result in the next section.

\begin{proposition}\label{half:prop}
Let $H\in\Bc$ be a Hopf algebra and $\Cc=H_\Bc\mod$, we have a  functor $$\Phi: HH_{HH(\Bc)}(\Cc)\to HH(\Cc)$$ that sends an $M\in aYD^H_{HH(\Bc)}$ to $M\in HH(\Cc)$ with the anti-center structure given by \begin{equation}\label{onehalf}
\includegraphics[height=1in]{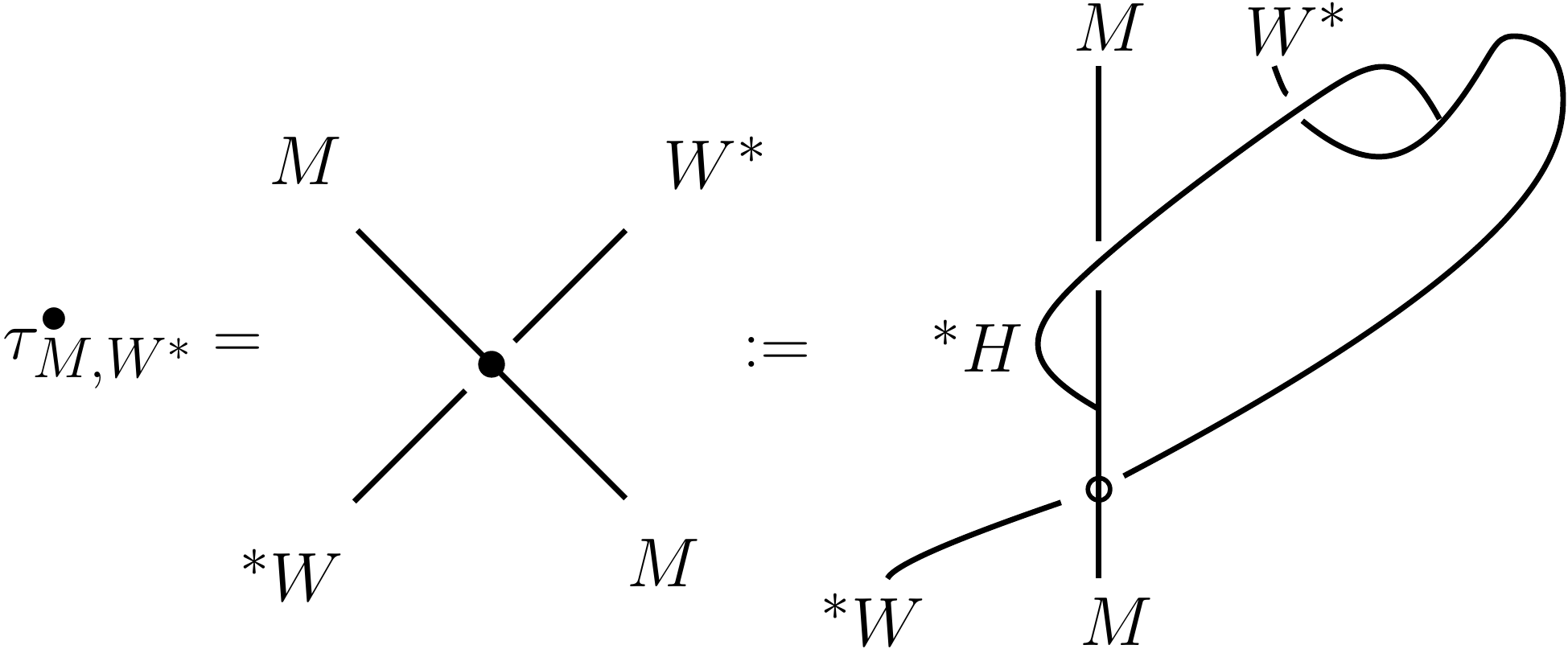}
\end{equation} This functor identifies $\varsigma^H$ on LHS with  $\varsigma$ on RHS.
\end{proposition}

\begin{proof}
Observe that  for $M\in HH(\Bc)$ the natural map $Tr U(M)=M\ot Tr(1)\to M$ can be interpreted as a right action of $Tr(1)$.  For $X\in\Bc$ we have $X\to U Tr(X)=X\ot Tr(1)$ given by the unit map $1\to Tr(1)$.  With this \eqref{contrafk} becomes: \begin{equation*}
\includegraphics[height=1.1in]{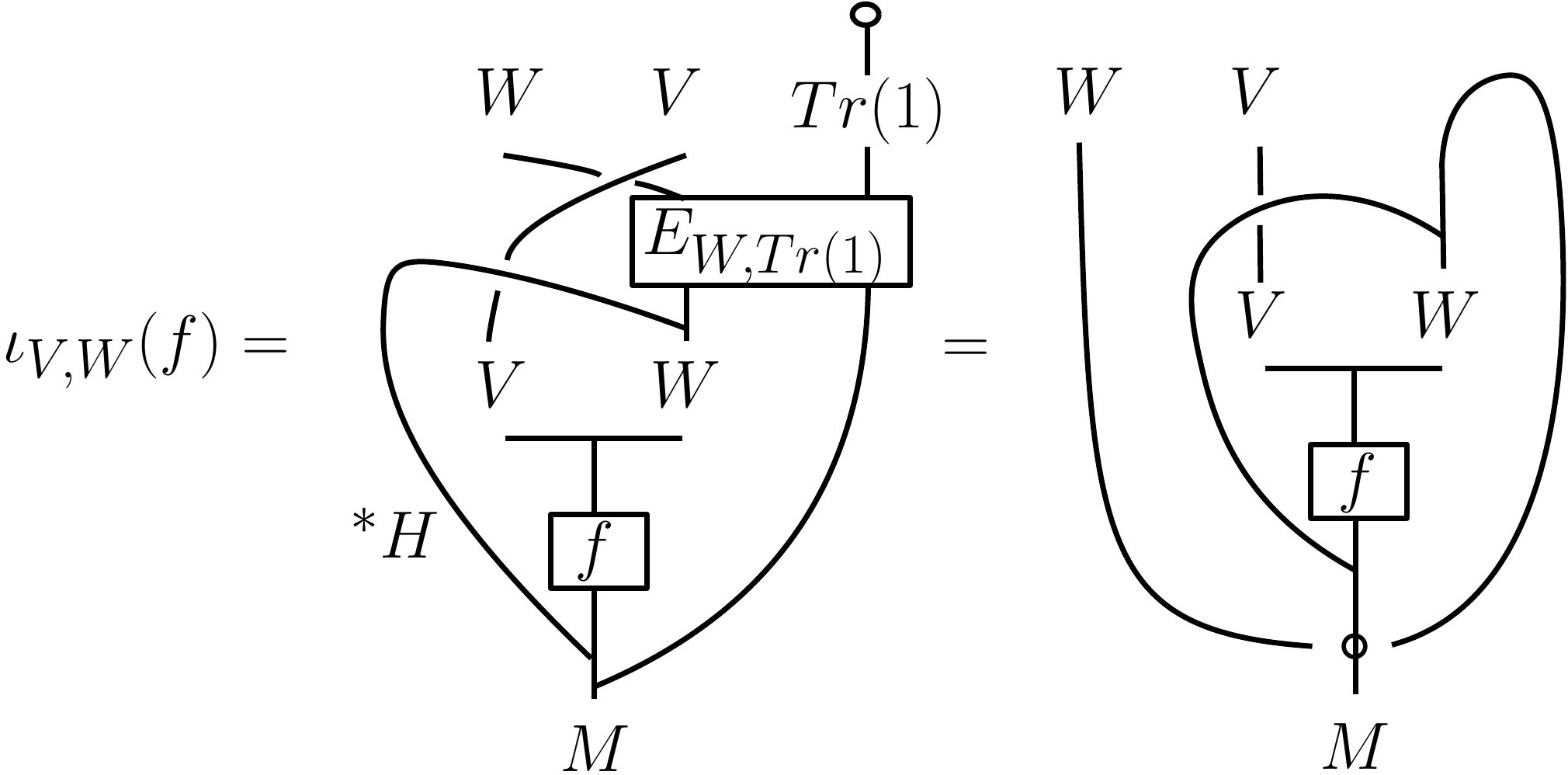}
\end{equation*} which, according to \eqref{convert} becomes \eqref{onehalf}.
\end{proof}

\subsection{The localization theorem}

As usual, assume that $H\in\Bc$ is a Hopf algebra and $\Cc=H_\Bc\mod$. The goal of this section is to prove that the functor $\Phi$ in Proposition \ref{half:prop} is an isomorphism of categories by constructing its inverse.  The following definition extracts the $aYD^H_{HH(\Bc)}$-structure from $HH(\Cc)$-structure.

\begin{definition}\label{def:mainstr}
Let $M\in HH(\Cc)$, denote the structure by $\tau^\bullet$, then \begin{itemize}
\item $M$ is an $H$-module in $\Bc$, with action denoted by $\rho_{H,M}$.
\item Let $\rho_{{}^*H,M}$ be given by \begin{equation*}
\includegraphics[height=.7in]{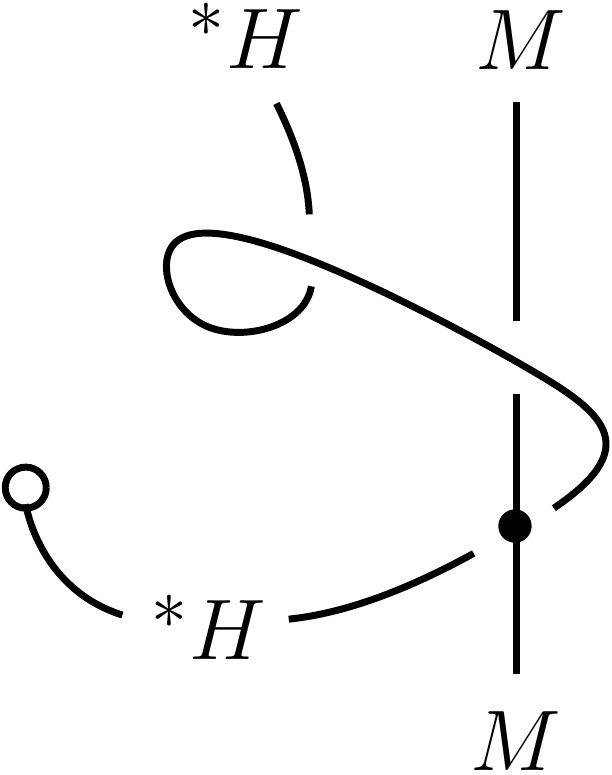}
\end{equation*} Note that the unit $1\to H$ is not $H$-linear, so the above is non-trivial.
\item $M\in HH(\Bc)$ since $\Bc$ fully embeds into $\Cc$, i.e., $X\in\Bc$ can be given $H$-structure by $\rho_{H,X}=\epsilon\ot Id_X$.  The resulting $HH(\Bc)$ structure is denoted by $\tau^\circ$, thus $$\tau^\circ=\tau^\bullet|_\Bc.$$  For $V\in\Cc$, we again denote by $\underline{V}$ an object in $\Bc$ obtained by forgetting the $H$-structure.  Note that $\tau^\bullet_{M, V^*}\neq\tau^\circ_{M,\underline{V}^*}$ in general, see Lemma \ref{bullet}; we will suppress the underline when using $\tau^\circ$.
\end{itemize}
\end{definition}

\begin{lemma}
With the algebra structure on ${}^*H$ given by Definition \ref{alg:def:hstr}, the  $\rho_{{}^*H,M}$ above defines an action.
\end{lemma}

\begin{proof}
Unitality is obvious since $\epsilon: H\to 1$ is $H$-linear.  Similarly, though less obviously, associativity follows from $H$-linearity of $\Delta: H\to H\ot H$.
\end{proof}

\begin{remark}\label{ahhb:rem}
Let $A\in\Bc$ be an algebra, suppose that $M\in HH(\Bc)$ and $M\in A_\Bc\mod$ then $M\in A_{HH(\Bc)}\mod$ if and only if we have: \begin{equation}\label{ahhb}
\includegraphics[height=.7in]{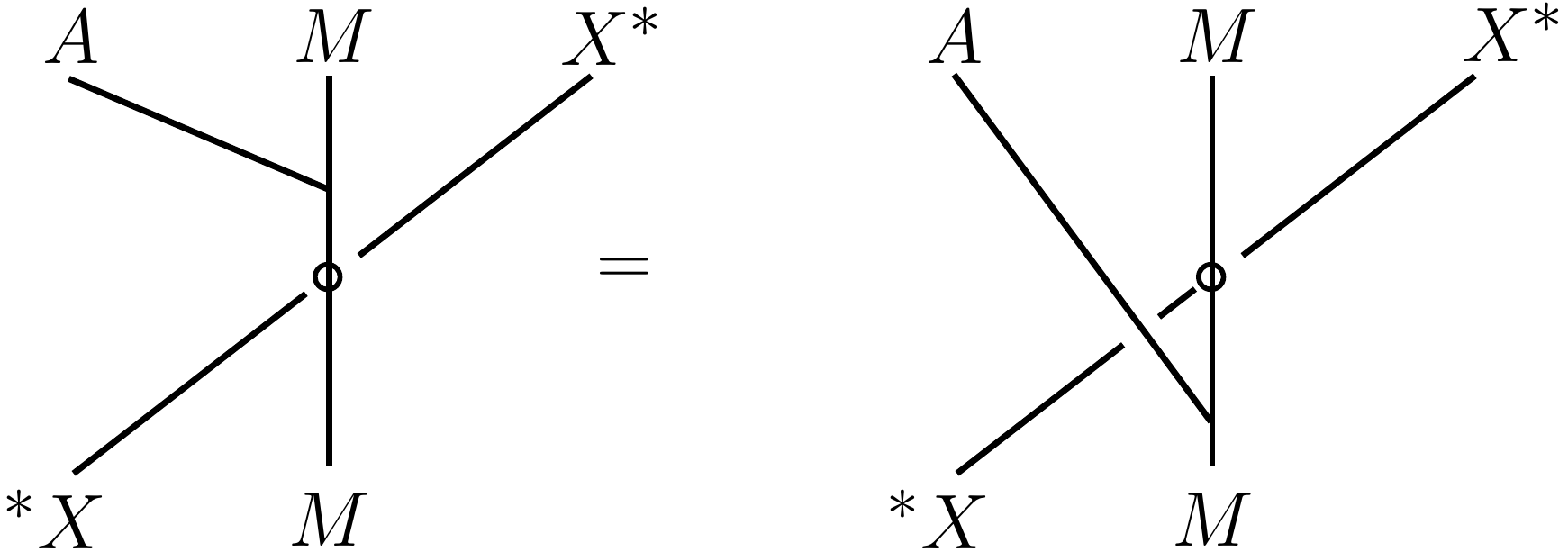}
\end{equation}
\end{remark}

\begin{lemma}\label{lem:hstr}
We have $M\in {}^*H_{HH(\Bc)}\mod$.
\end{lemma}
\begin{proof}
By Remark \ref{tau:rem} we have, for $V\in\Cc$ and $X\in \Bc$: \begin{equation*}
\includegraphics[height=.7in]{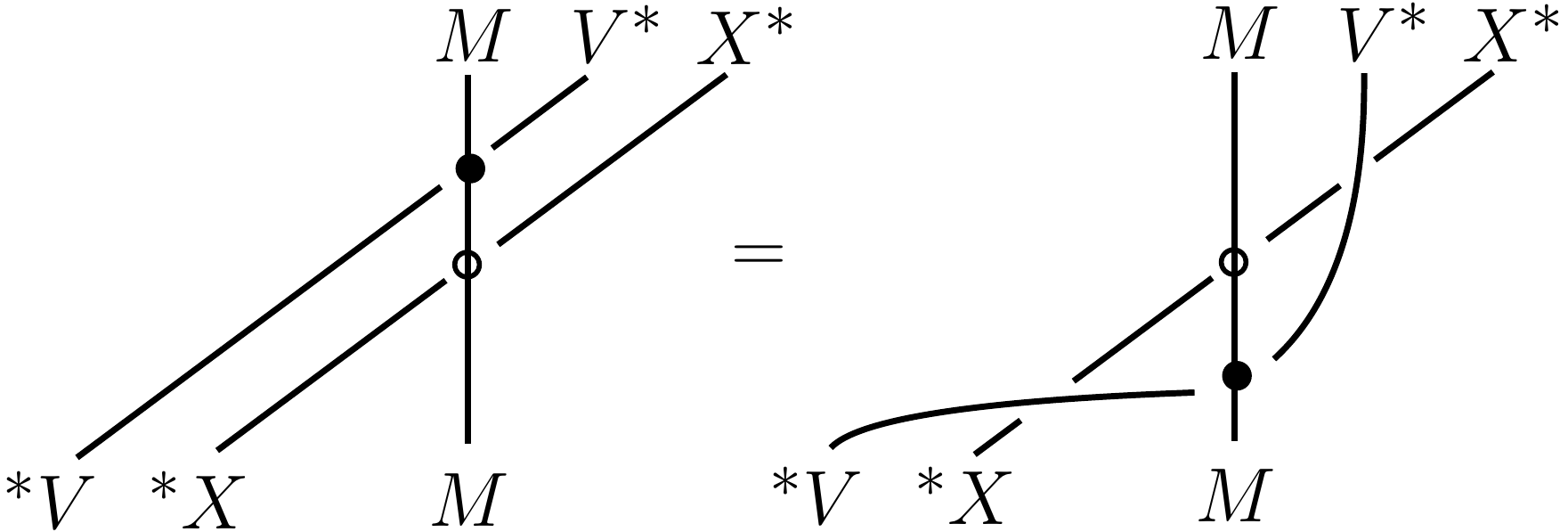}
\end{equation*} Thus, \begin{equation*}
\includegraphics[height=.7in]{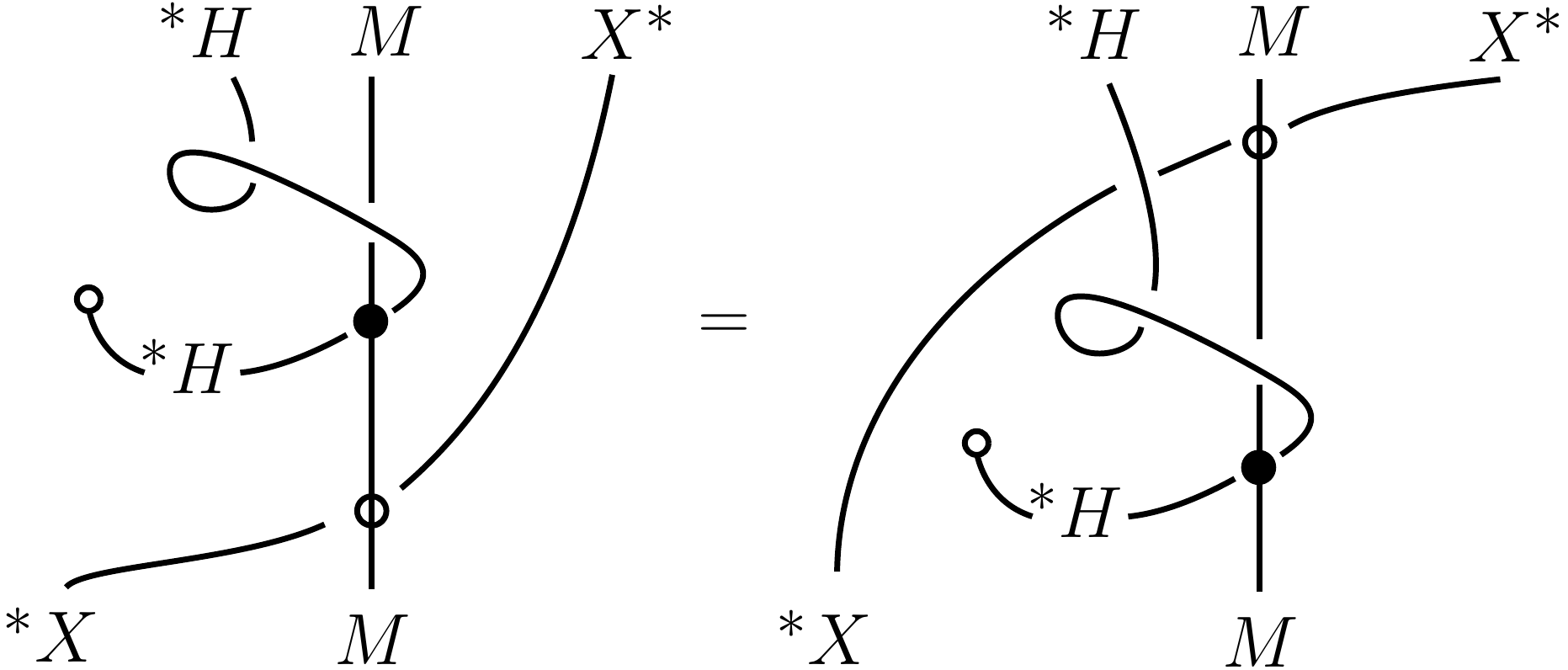}
\end{equation*} and we are done by Remark \ref{ahhb:rem}.
\end{proof}

\begin{lemma}\label{lem:h}
We have $M\in  H_{HH(\Bc)}\mod$.
\end{lemma}

\begin{proof}
Let $X\in\Bc$ and consider it in $\Cc$ with the trivial $H$-action.  Then since $\tau^\circ_{M,X^*}=\tau^\bullet_{M,X^*}$ is $H$-linear, we obtain \eqref{ahhb} with $A=H$.
\end{proof}

\begin{lemma}\label{bullet}
With $M$ as above and $V\in\Cc$ we have \begin{equation*}
\includegraphics[height=.8in]{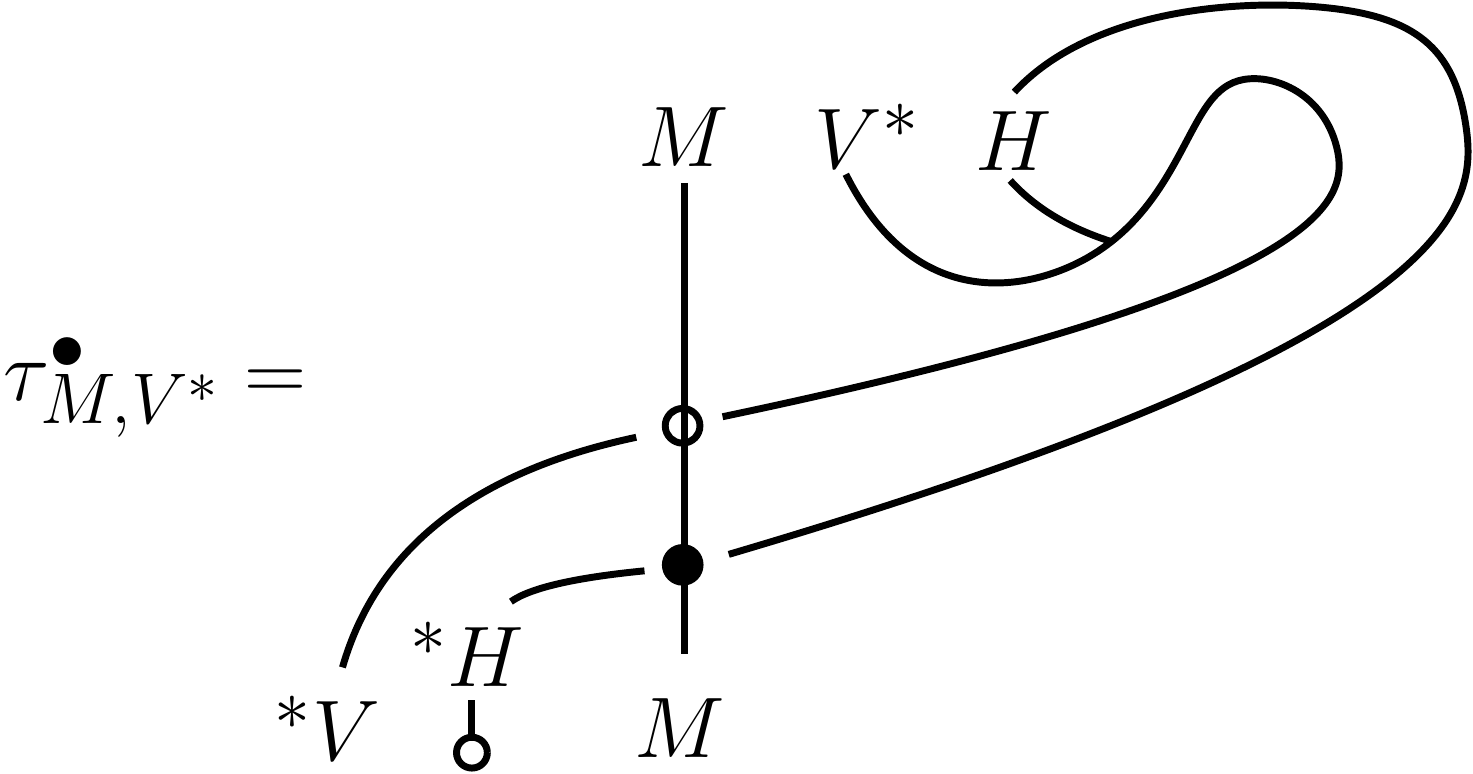}
\end{equation*}
\end{lemma}

\begin{proof}
Since $\epsilon$ is $H$-linear we have:\begin{equation*}
\includegraphics[height=.6in]{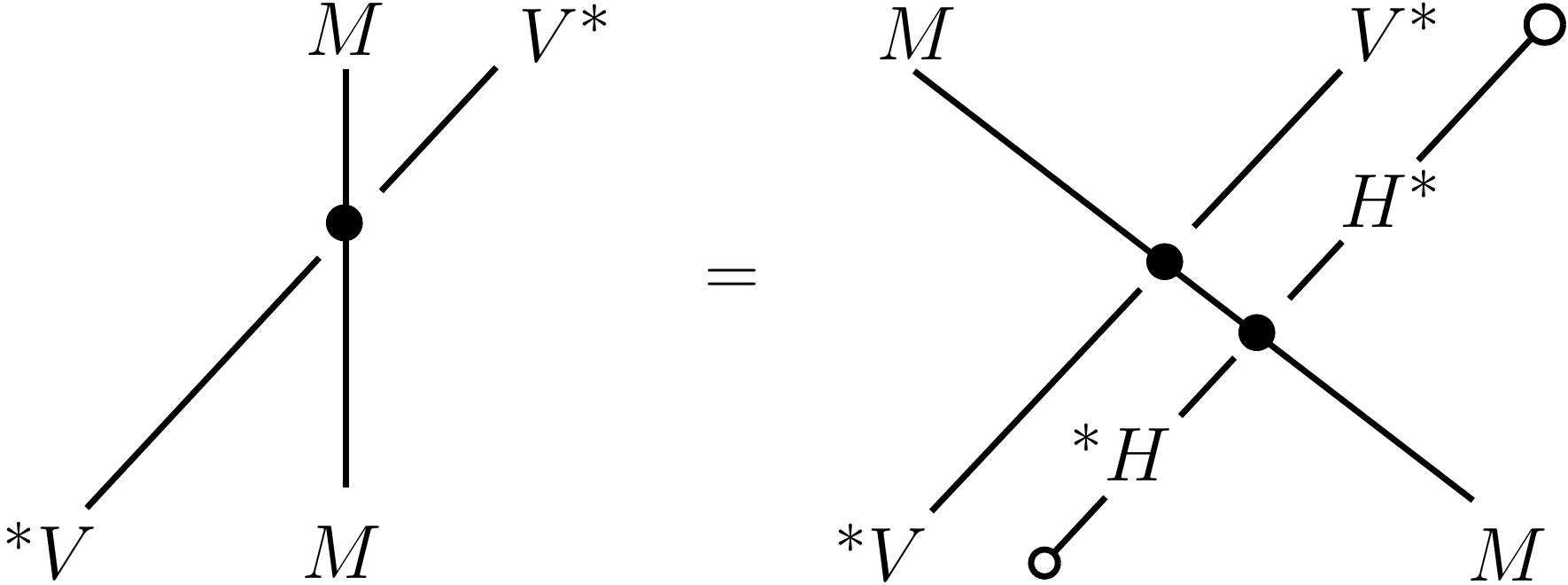}
\end{equation*} On the other hand, by Lemma \ref{simplifylem}: \begin{equation*}
\includegraphics[height=.8in]{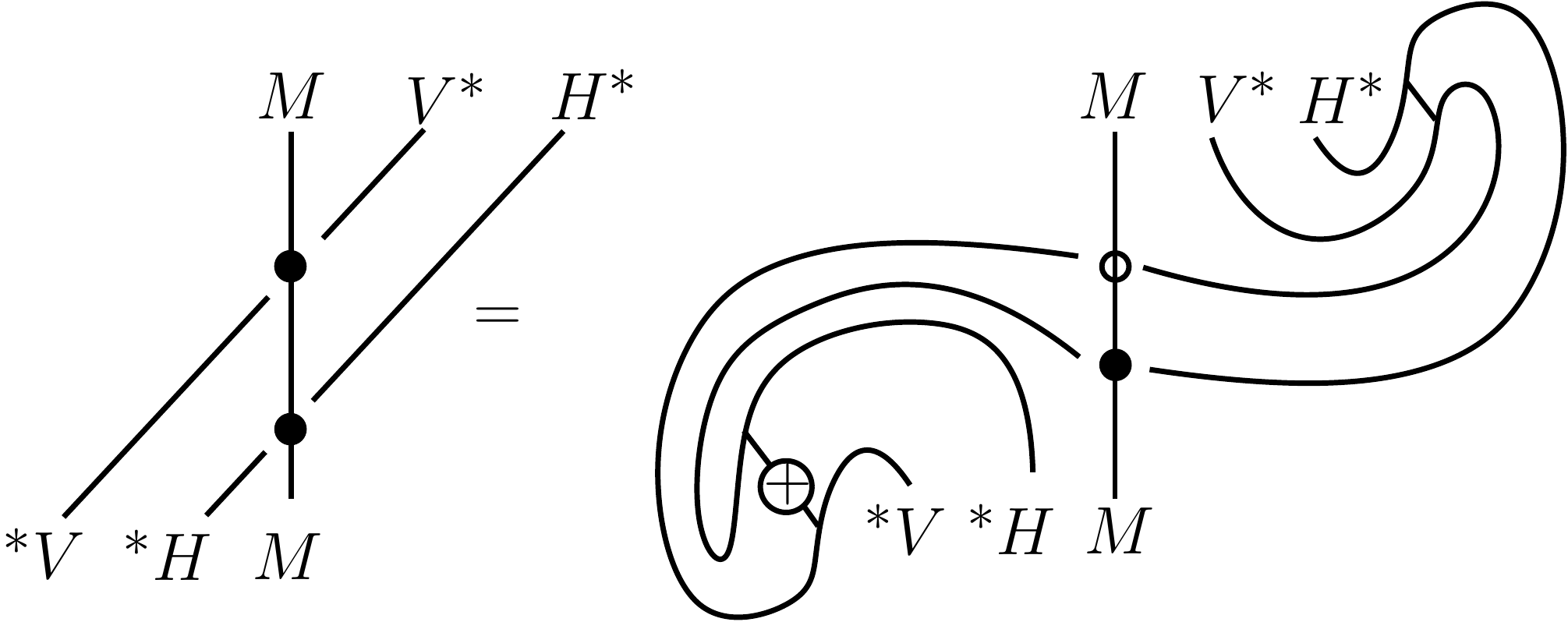}
\end{equation*} Now apply $-\circ \epsilon^*$  and ${}^*u\circ -$ to obtain the result.
\end{proof}

\begin{proposition}\label{prop:recover}
Let $M\in HH(\Cc)$ and let the $\tau^\circ$ and ${}^*H$-action be as in Definition \ref{def:mainstr}.  We can recover the $HH(\Cc)$ structure from $\tau^\circ$ and ${}^*H$-action.\begin{equation}\label{eq:recover}
\includegraphics[height=.9in]{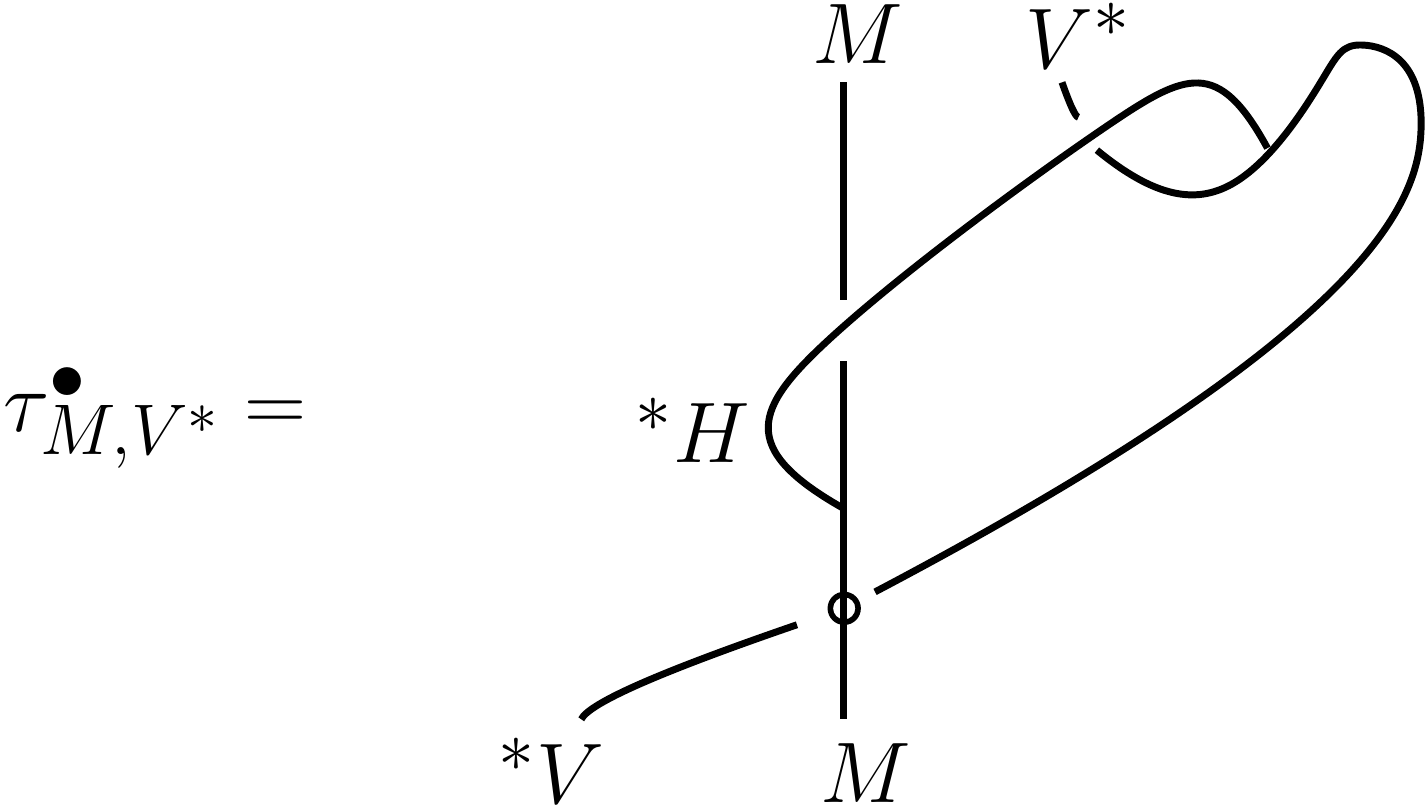}
\end{equation}
\end{proposition}

\begin{proof}
This follows immediately from Lemma \ref{lem:hstr} and Lemma \ref{bullet}.
\end{proof}

\begin{theorem}\label{thm:compat}
Let $M\in HH(\Cc)$ and let the $\tau^\circ$, $H$ and ${}^*H$-actions be as in Definition \ref{def:mainstr}.  We have the following compatibility between $H$ and ${}^*H$-actions: \begin{equation*}
\includegraphics[height=1in]{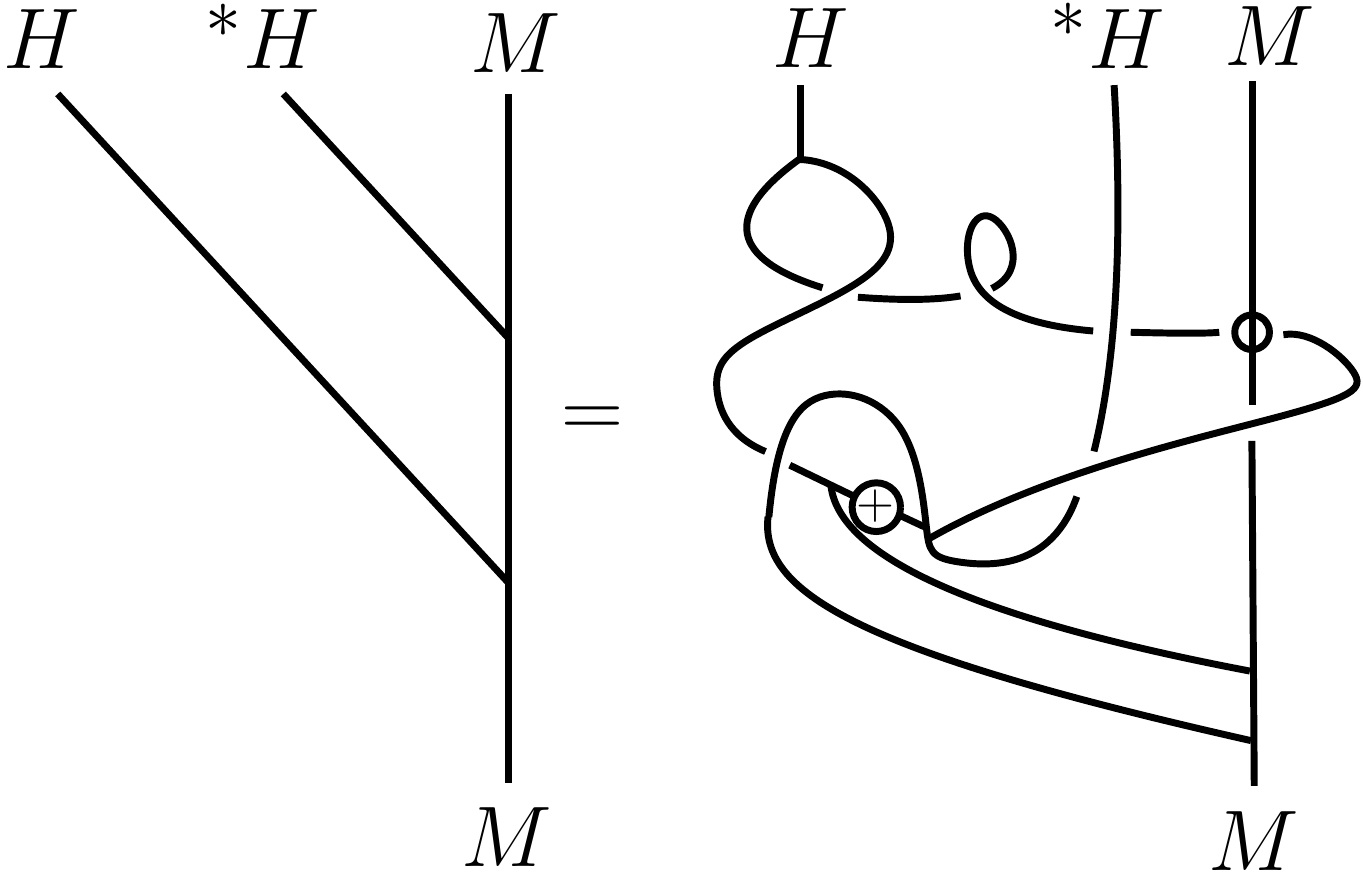}
\end{equation*}
\end{theorem}
\begin{proof}
By assumption $\tau^\bullet_{M,V^*}$ is $H$-linear, so that $ev_{V,{}^*V}\tau^\bullet_{M,V^*}:V\ot M\ot V^*\to M$ is $H$-linear, so $\rho_{H,M}ev_{V,{}^*V}\tau^\bullet_{M,V^*}=ev_{V,{}^*V}\tau^\bullet_{M,V^*}\rho_{H,VMV^*}:H\ot V\ot M \ot V^*\to M$.   Replace  $V$ by $H$ and pre-compose with $u:1\to H$ to obtain  $$\rho_{H,M}ev_{H,{}^*H}\tau^\bullet_{M,H^*}Id_H\ot u\ot Id_{MH^*}=ev_{H,{}^*H}\tau^\bullet_{M,H^*}\rho_{H,HMH^*}Id_H\ot u\ot Id_{MH^*}$$ as maps from $H \ot M \ot H^*$ to  $M$.  Finally, pre-compose with $$\tau_{{}^*H,M}\alpha:=\tau_{{}^*H,M}ev_{H,{}^*H}\tau_{H^*,{}^*H}coev_{H,H^*}:{}^*H\ot M\to M\ot H^*$$ to obtain that $$\rho_{H,M}ev_{H,{}^*H}\tau^\bullet_{M,H^*}u\tau_{{}^*H,M}\alpha=ev_{H,{}^*H}\tau^\bullet_{M,H^*}\rho_{H,HMH^*}u\tau_{{}^*H,M}\alpha$$ as maps from  $H\ot {}^*H\ot M$ to $M$. By Proposition \ref{prop:recover} this results in:\begin{equation*}
\includegraphics[height=.9in]{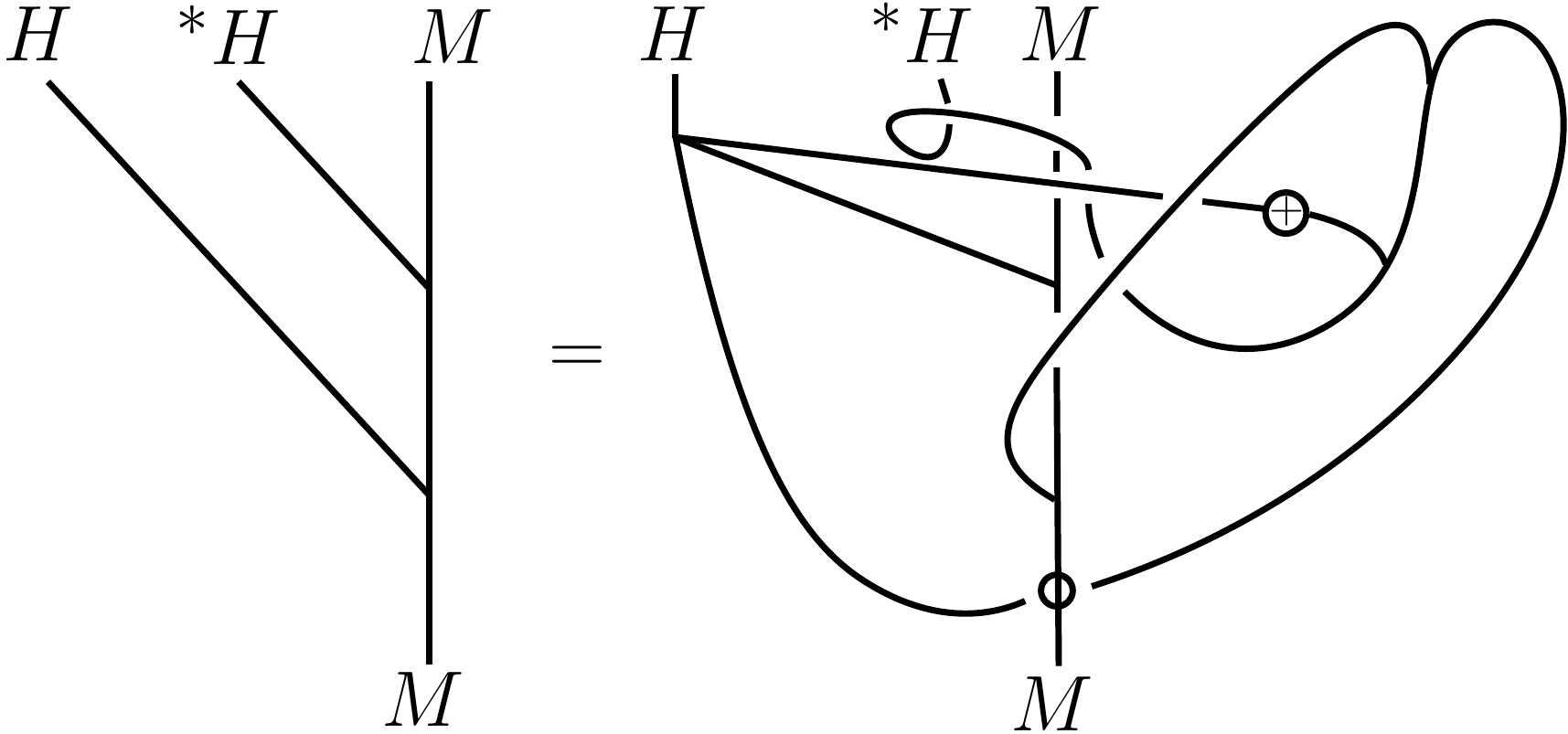}
\end{equation*} and the latter eventually simplifies to the result.
\end{proof}

\begin{theorem}\label{thm:localization}
Let $H\in\Bc$ be a Hopf algebra, let $\Cc=H_\Bc\mod$, then $$HH_{HH(\Bc)}(\Cc)\simeq HH(\Cc)$$ and the isomorphism identifies $\varsigma^H$ on LHS with  $\varsigma$ on RHS.
\end{theorem}
\begin{proof}
The $\varsigma$-compatible functor $\Phi:HH_{HH(\Bc)}(\Cc)\to HH(\Cc)$ was given in Proposition \ref{half:prop}.  On the other hand, we can define $\Theta: HH(\Cc)\to HH_{HH(\Bc)}(\Cc)$ by  using Definition \ref{def:mainstr}.  Namely, we have Lemma \ref{lem:hstr}, Lemma \ref{lem:h}, and Theorem \ref{thm:compat} that prove that for $M\in HH(\Cc)$, we have $\Theta(M)\in aYD^H_{HH(\Bc)}=HH_{HH(\Bc)}(\Cc)$.  By Proposition \ref{prop:recover} we have $\Phi\Theta=Id$, and for $\Theta\Phi=Id$ note that if $M\in aYD^H_{HH(\Bc)}$, then \eqref{eq:recover} evaluated on $X\in\Bc$ recovers the original $HH(\Bc)$ structure.  Furthermore, the $H$-structure is the same, and for the ${}^*H$-structure observe that: \begin{equation*}
\includegraphics[height=.9in]{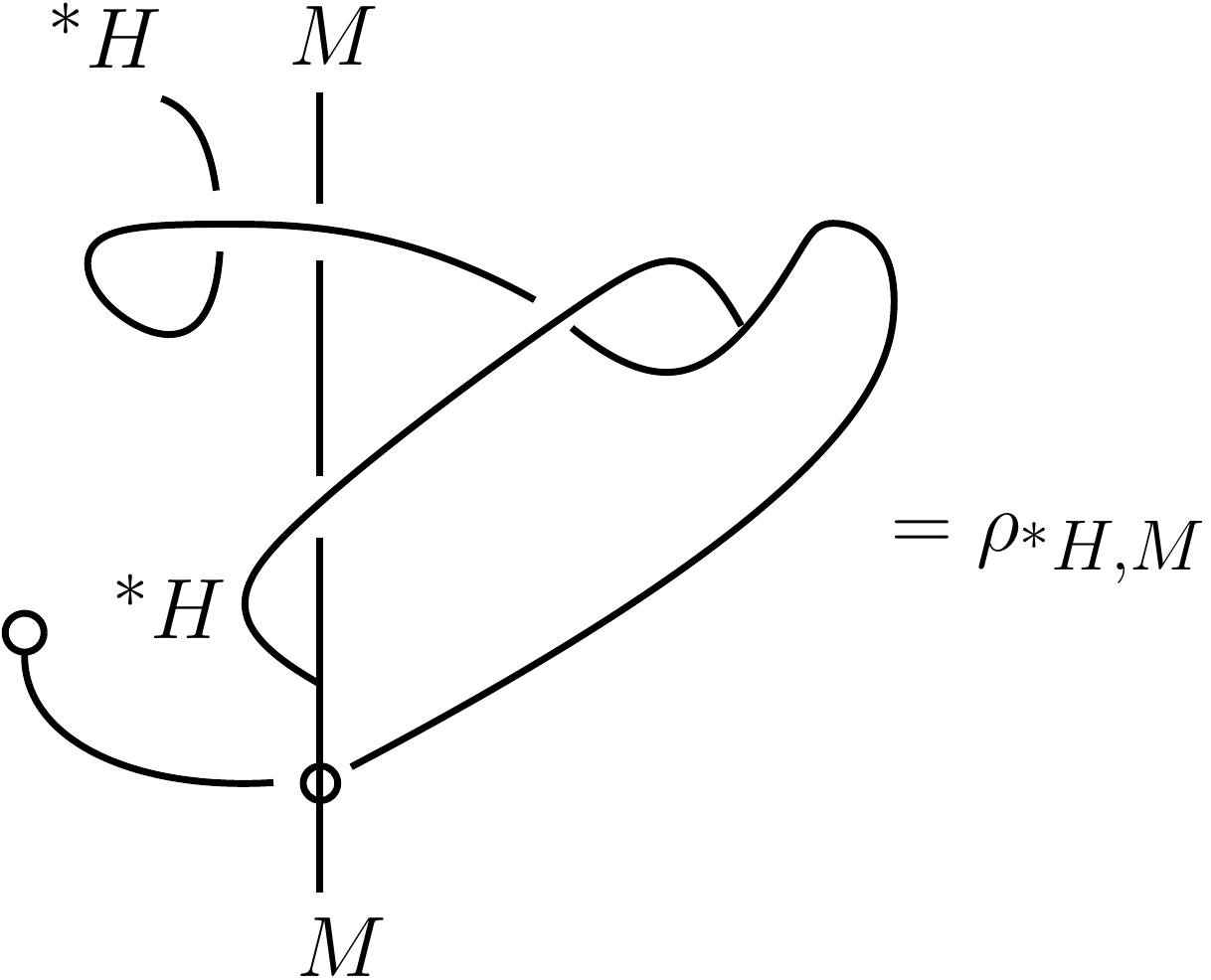}
\end{equation*}
\end{proof}

\begin{definition}
Let $\rho\in NatIso^\ot(Id_\Bc,(-)^\#)$ be a pivot. Define $HH_\rho(\Cc)$ to be the full subcategory of $M\in HH(\Cc)$ such that for $X\in\Bc\subset H_\Bc\mod$ we have $\tau^\bullet_{X,M}=\rho_X\tau^{-1}_{X,M}$.  
\end{definition}

The following is immediate from Theorem \ref{thm:localization}.
\begin{corollary}\label{cor:aydbs}
Let $\varsigma$ and $\rho$ be related as in \eqref{hhstr1}.  Then $$HH_{\Bc_\varsigma}(\Cc)\simeq HH_\rho(\Cc)$$ compatibly with $\varsigma^H$ on the LHS and the restriction of $\varsigma$ on $HH(\Cc)$ to $HH_\rho(\Cc)$ on the RHS.
\end{corollary}

We now apply the theorem to the case considered in Section \ref{sec:decompo}.  Namely, $\Bc=(\vect_G,\chi)$ with $H\in\Bc$  a Hopf algebra and $\Cc=H_\Bc\text{-mod}$.  We assume that $\omega:G\to\hat{G}$ is an isomorphism. Recall that  $\theta(x)=\chi(x,x)=\varsigma^{-1}(x)$; let $$I=\theta(G)\subset k^\times\quad\text{and}\quad n_i=\#\{x|\theta(x)=i\}.$$ 

\begin{corollary}\label{cor:hcdecomp}
Let $X=HH_{\Bc_\varsigma}(\Cc)$, then \begin{equation*}\label{localrem}HC(\Cc)=\bigoplus_{i\in I}X^{i\varsigma^H}\boxtimes\vect^{n_i}.\end{equation*}
\end{corollary}

\begin{proof}
By Corollary \ref{cor:decompmore} all the summands of $HH(\Bc)$ are identical as braided modules, with only the stable structure differing by a specified scalar multiple.  Thus, by Theorem \ref{thm:localization}, all the summands of $HH(\Cc)$ are identical, with only the $\varsigma^H$'s differing  by a scalar multiple.  
\end{proof}

\section{An application: Taft Hopf algebras}\label{sec:blah}
Let $\xi$ be a primitive $p$th root of unity in $k$, where $p$ is a prime.  The Taft Hopf algebra \cite{taft} $T_p(\xi)$ is generated as a $k$-algebra by $g$ and $x$ with the relations \begin{equation}g^p=1,\quad x^p=0, \end{equation} \begin{equation} gx=\xi xg.\end{equation}  Thus it is $p^2$ dimensional over $k$.  Furthermore, the coalgebra structure is \begin{equation}
\Delta(g)=g\ot g,\quad \Delta(x)=x\ot 1+g\ot x
\end{equation} with $\epsilon(g)=1$, $\epsilon(x)=0$, and thus $S(g)=g^{-1}$, while $S(x)=-g^{-1}x$.  Note that $$S^2(x)=\xi^{-1}x\neq x,$$ making $T_2(-1)$ the smallest Hopf algebra with $S^2\neq Id$.  The Taft algebra $T_2(-1)$ is somewhat different from the other $T_p(\xi)$ and has its own name: Sweedler’s Hopf algebra.

\subsection{Taft algebra as a Hopf algebra in a braided monoidal category.}\label{hopfbraidtaft}
It is interesting to observe, especially in the case of $T_2(-1)$, that Taft Hopf algebras, or rather their monoidal categories of modules, can be better understood via related, simpler, Hopf algebras in braided monoidal categories \cite{majidbraidhopf}.  This process is called transmutation, the reverse process is called bosonization. 

Consider an abelian group $G=\zp$ of integers modulo $p$, with a bicharacter  $$\chi(i,j)=\xi^{ij}.$$  Let $\Bc=(\vect_{\zp},\xi)$ denote the resulting, by Section \ref{sec:decompo}, braided category.  Note that this braiding is not symmetric unless $p=2$, in which case the braided monoidal category we get is just $s\vect$, the category of super vector spaces with the usual sign conventions braiding.  What is true for higher $p$ is that the braided category $(\vect_{\zp},\xi)$ is a ribbon category with the ribbon element $$\theta(v)=\xi^{i^2}v,\quad\text{for}\quad v\in V_i.$$  

\begin{definition}
For $p$ prime, and $\mu\in\zp$ consider the anti-twist on $\Bc$ given by $$\varsigma_\mu(v)=\xi^{-i^2-\mu i}v,\quad\text{for}\quad v\in V_i.$$  Denote by  $\Bc_\mu$ the resulting stable braided $\Bc$-module.  Recall from Section \ref{sec:decompo} that $$HH(\Bc)=\bigoplus_{\mu\in\zp}\Bc_\mu$$ as stable braided $\Bc$-modules.
\end{definition}

The following lemma is immediate, the Hopf algebra is called the one-dimensional anyonic enveloping algebra or anyonic line:

\begin{lemma}\label{taftlemma1}
Let $H=k[x]/x^p$ with $x$ of degree $1$ be a Hopf algebra in $(\vect_{\zp},\xi)$, where $\Delta(x)=x\ot 1+1\ot x$, $S(x)=-x$, and $\epsilon(x)=0$.  Then as monoidal categories: $$H_{\vect_{\zp}}\text{-mod}=T_p(\xi)_\vect\text{-mod},$$ i.e., modules over $H$ in $\vect_{\zp}$ are the same as modules over $T_p(\xi)$.

\end{lemma}
\begin{proof}
Both are $\zp$-graded vector spaces, with the grading on the right given by the eigenspaces, $E_{\xi^i}$, of $g$.  The action of $x$ gives a degree $1$ operator with $x^p=0$.  The monoidal structure agrees as well and explains our slightly unusual convention for $\Delta(x)=x\ot 1+ g\ot x$ and not $\Delta(x)=1\ot x+ x\ot g$ as is found in the majority of the literature.
\end{proof}

If $p=2$ then $T_2(-1)$-modules are the same, according to the above observation, as sheaves on $(k^{0|1},+)$ (with convolution), the odd analogue of the additive group $(k,+)$.  As the group is abelian, the category is actually itself braided.  This observation does not generalize to higher primes.  In fact according to \cite{nobraid}, $T_p(\xi)$ are not quasi-triangular for $p>2$.

\begin{remark}
It is often mentioned that $T_2(-1)$ is the smallest non-trivial Hopf algebra, i.e., not obtained from a finite group $G$.  We see above that if one allows supergroups then it is trivial still.
\end{remark}

\subsection{Some calculations}
In this section we let $H=k[x]/x^p$ as in Lemma \ref{taftlemma1}.  It is a Hopf algebra in the braided category $\Bc=(\vect_{\zp},\xi)$.  Recall that, for $\xi$ a root of unity and $n\in\mathbb{Z}_{\geq 0}$, one writes $$(n)_\xi=1+\cdots +\xi^{n-1}\quad\text{and}\quad (n)_\xi !=(n)_\xi\cdots(1)_\xi.$$  Let ${}^i e={}^*(x^i)\in {}^*H$, so that the degree of ${}^i e$ is $-i$ and they form a basis of ${}^*H$.  Note that under the algebra structure on ${}^*H$ obtained from Lemma \ref{multiplication}, we have $${}^i e\cdot {}^j e=\xi^{-ij}\dfrac{(i+j)_\xi !}{(i)_\xi ! (j)_\xi !}\,{}^{i+j}e$$ using the observation that $\Delta(x^n)=(1\ot x+ x\ot 1)^n=\displaystyle\sum_{i+j=n}\dfrac{(i+j)_\xi !}{(i)_\xi ! (j)_\xi !}\, x^i\ot x^j$.

\begin{lemma}\label{isolemma2}
Let $z$ be of degree $-1$, then $${}^*H\simeq k[z]/z^p.$$
\end{lemma}
\begin{proof}
It is immediate that $z^i\mapsto \xi^{-\frac{(i-1)i}{2}}(i)_\xi !\,{}^i e$ is an algebra isomorphism.
\end{proof}

We can now describe $aYD^H_{\Bc_\mu}$ and $\varsigma^H$:

\begin{proposition}\label{prop:taft}
Let $\Bc=(\vect_{\zp},\xi)$ and $H=k[x]/x^p$, then $HH_{\Bc_\mu}(H_\Bc\mod)$ consists of $M\in\Bc$ equipped with two operators: $x$ of degree $1$, and $z$ of degree $-1$ such that $x^p=0$, $z^p=0$, and for $m\in M_i$ we have \begin{equation}\label{reltaft} (xz-\xi zx)m=(\xi^{-2i+1-\mu}-1)m. 
\end{equation} Furthermore, the action of $\varsigma^H$ is given by \begin{equation}\label{acttaft01}
m\mapsto \xi^{-i^2-\mu i}\left(\displaystyle\sum_{j=0}^{p-1}\dfrac{\xi^{\frac{(j-1)j}{2}}}{(j)_\xi !}z^jx^j\right)m.
\end{equation}
\end{proposition}
\begin{proof}
This is obtained by unpacking \eqref{main:eq}, \eqref{varsigma:eq}, and using Lemma \ref{isolemma2}.  Note that $$z={}^*x\in {}^*H.$$
\end{proof}
\begin{definition}\label{reladd}
Proposition \ref{prop:taft} can be used to define an algebra $D^a_\mu(T_p(\xi))$ generated by $x,z,g$ that satisfy $$x^p=z^p=g^p-1=0$$ and $$gxg^{-1}=\xi x,\quad gzg^{-1}=\xi^{-1}z,\quad xz-\xi zx=\xi^{1-\mu} g^{-2}-1.$$  Thus, $$HH_{\Bc_\mu}(H_\Bc\mod)=D^a_\mu(T_p(\xi))\mod.$$
\end{definition}

\begin{remark}\label{rem:p2}
Recall that if $p=2$ then $\Bc=s\vect$ has two anti-twists: the $\varsigma_0$ above, and $\varsigma_1$ which is trivial; the category is  symmetric.  We note that \eqref{reltaft} and \eqref{acttaft01} yield ($y=-z$): $$xy+yx=2,\quad \varsigma^H m=(-1)^i(1-yx)m$$ for the $\varsigma_0$  structure, while we would get $$xz+zx=0,\quad \varsigma^H m=(1+zx)m$$ for the trivial one. Thus, from  \cite{cyctaft} we see that $$HC_{\Bc_0}(H_\Bc\mod)=\vect,$$  while $$HC_{\Bc_1}(H_\Bc\mod)=\left(k[x,z]/(zx)\right)_{s\vect}\mod$$ where $k[x,z]$ is a free super-commutative algebra on two odd generators  (this was the interesting part of $HH(H_\Bc\mod)^{S^1}$ in \cite{cyctaft}).  We obtain the decomposition:  $$saYD^{T_2(-1)}=HC(T_2(-1)\mod)=\vect\oplus \left(k[x,z]/(zx)\right)_{s\vect}\mod.$$  
\end{remark}

\subsection{The case of  $p>2$}  In this case there is a $p$th root of unity $q$ such that $q^2=\xi^{-1}.$ If we let $p=2m+1$ then $$q=\xi^m.$$ Let us write, for $n\in\mathbb{Z}_{\geq 0}$, $$[n]_q=\dfrac{q^n-q^{-n}}{q-q^{-1}}\quad \text{and} \quad [n]_q!=[n]_q[n-1]_q\cdots [1]_q.$$  Note that $$(n)_\xi !=q^{-n(n-1)/2}[n]_q !.$$  

\begin{proposition}\label{reltaftpnot2}
Let $p>2$.  We have an identification of categories: \begin{equation}\label{taftid}HH_{\Bc_\mu}(H_\Bc\mod)=u_q(sl_2)\mod.
\end{equation}  In particular, the choice of $\mu$ is not yet relevant, as expected since $\omega=\xi^{2ij}$ is non-degenerate.
\end{proposition}
\begin{proof}
For $M\in\Bc$, let $g$ act on $M_i$ by $\xi^i$.  It is clear that with this convention we have an identification of categories $\Bc=k[g]/(g^p-1)\mod$.  Let $$E=q^{1-\mu}x,\quad F=zg,\quad K=q^{\mu-1}g^{-1}$$ then $$E^p=F^p=K^p-1=0$$ and $$KEK^{-1}=q^2 E,\quad KFK^{-1}=q^{-2}F, \quad [E,F]=K-K^{-1}.$$
\end{proof}

\begin{remark}
Note our use (to be consistent with \cite{ker1}) of the $[E,F]=K-K^{-1}$ convention, instead of the $[E,F]=\dfrac{K-K^{-1}}{q-q^{-1}}$ convention.
\end{remark}

Recall (from \cite{ker1} for example) that $u_q(sl_2)$ is a ribbon Hopf algebra with the ribbon element $$v_0=Ku_K u_0$$ where $$u_K=\dfrac{\sum_{i=0}^{p-1}q^{mi^2}K^i}{\sum_{i=0}^{p-1}q^{mi^2}},\quad u_0=\displaystyle\sum_{j=0}^{p-1}\dfrac{q^{\frac{(j+3)j}{2}}}{[j]_q !}K^jF^jE^j.$$

\begin{proposition}\label{reltaftpnot2sigma}
Under the identification of \eqref{taftid} we have:  $$\varsigma^H_\mu= q^{m(\mu^2-1)}v_0$$ where $v_0$ is the ribbon element of $u_q(sl_2)$. In particular, $\varsigma^H_\mu$ depends only on the square of $\mu$, as expected by Remark \ref{pair:rem}.
\end{proposition}
\begin{proof}
Consider the action of $\varsigma^H$ on $M_i$, then starting with  \eqref{acttaft01} we have: \begin{align*}\varsigma^H_\mu &=\xi^{-i^2-\mu i}\left(\displaystyle\sum_{j=0}^{p-1}\dfrac{\xi^{\frac{(j-1)j}{2}}}{(j)_\xi !}z^jx^j\right)=\xi^{-i^2-\mu i}\left(\displaystyle\sum_{j=0}^{p-1}\dfrac{q^{-\frac{(j-1)j}{2}}}{[j]_q !}z^jx^j\right)\\&=\xi^{-i^2-\mu i}\left(\displaystyle\sum_{j=0}^{p-1}\dfrac{q^{\frac{(j+3)j}{2}}}{[j]_q !}K^jF^jE^j\right)=\xi^{-i^2-\mu i}u_0.
\end{align*} We now focus on $\xi^{-i^2-\mu i}$, expanding it as a linear combination of characters and completing the square twice we obtain: \begin{align*}\xi^{-i^2-\mu i}&=\frac{K^\mu q^{(1-\mu)\mu}}{p}\displaystyle\sum_{k,s=0}^{p-1}\xi^{-s^2+sk}q^{(1-\mu)k} K^k=\frac{q^{(1-\mu^2)/2}K}{p}\left(\displaystyle\sum_{s=0}^{p-1}q^{2s^2}\right)\left(\displaystyle\sum_{k=0}^{p-1}q^{-k^2/2}K^k\right)\\
\intertext{using the properties of Gauss sums, see \cite{lang} we get:}
&=q^{(1-\mu^2)/2} K\left(\displaystyle\sum_{k=0}^{p-1}q^{mk^2}K^k\right)/\left(\displaystyle\sum_{s=0}^{p-1}q^{ms^2}\right)=q^{m(\mu^2-1)}Ku_K.
\end{align*}
\end{proof}

\begin{theorem}\label{reltaftthm}
Let $p$ be a prime, $H$ as in Lemma \ref{taftlemma1}  then: $$HC_{\Bc_0}(H_\Bc\mod) \simeq\vect.$$
\end{theorem}

\begin{proof}
The case of $p=2$ is addressed in Remark \ref{rem:p2}.  For $p$ odd, we recall from \cite{ker1} that the center $Z$ of $u_q(sl_2)$ decomposes: \begin{equation}\label{centdec}Z\simeq k\times\prod_{j=0}^{m-1}k[x_j,y_j]/(x_j^2,y_j^2,x_jy_j)\end{equation} and using the notation of the original: $P_m=1\in k$ and $P_j=1\in k[x_j,y_j]/(x_j^2,y_j^2,x_jy_j)$ while $N^+_j=x_j$ and $N^-_j=y_j$ we have an expression for the ribbon element: $$v_0=q^mP_m+\sum_{j=0}^{m-1}q^{2j(j+1)}(P_j+\alpha_jN^+_j+\beta_jN^-_j)$$ where $\alpha_j$ and $\beta_j$ are some known constants that we do not need here.  
Let $$u_q(sl_2)=U\times\prod_{j=0}^{m-1}U_j$$ be the decomposition induced by \eqref{centdec}, and set $\overline{U}_j=U_j/(x_j,y_j)$.  Note that $$1-\varsigma^H_0=1-q^{-m}v_0$$ acts by $0$ on $U$ and by $1-q^{-m+2j(j+1)}=1-q^{2(j-m)^2}$ on $\overline{U}_j$.  Since $1-q^{2(j-m)^2}\neq 0$ so $1-\varsigma^H_0$ acts invertibly on $U_j$; thus, as in \cite{cyctaft} we have that $HC_{\Bc_0}(H_\Bc\mod)=\vect.$ 

Using  \cite{mgs}, we see that the only part of the representation theory of $u_q(sl_2)$ that is supported at the $Spec (k)$ part of the center consists of copies of the unique projective irreducible representation of $u_q(sl_2)$, which is  the unique irreducible of dimension $p$.
\end{proof}

Theorem \ref{reltaftthm} dealt with the case of $\varsigma_0$, there remain cases $\varsigma_\mu$ for $\mu=1,\cdots, p-1$.  Recall that  $HC_{\Bc_\mu}(H_\Bc\mod)$,  depends only on the square of $\mu$ and so we may assume that $\mu$ is odd (as $-\mu$ is then even).  Furthermore, let \begin{equation}\label{muj}
2j+1=\mu,
\end{equation} such a $j$ is unique among $0,\cdots, m-1$.

\begin{proposition}
Let $p$ be an odd prime, $H$ as in Lemma \ref{taftlemma1}, and  $\mu$ as in \eqref{muj}, then: $$HH_{\Bc_\mu}(H_\Bc\mod)=U_j/(-\alpha_jN^+_j-\beta_jN^-_j)\text{-mod}.$$
\end{proposition}

\begin{proof}
The proof is analogous to that of Theorem \ref{reltaftthm}, i.e., we use Proposition \ref{reltaftpnot2sigma} to compute that $1-\varsigma^H_\mu$ acts invertibly on all but $U_j$ where the action is not $0$ as in Theorem \ref{reltaftthm}, but is $-\alpha_jN^+_j-\beta_jN^-_j$ which is square zero. 
\end{proof}

\section{Appendix}

The reader is invited to peruse  this section if needed.  It is  referred to in the body of the paper on occasion.

\subsection{Monoidal functors and algebras}\label{appx}
If $F$ is a (strongly) monoidal functor between two monoidal categories $\Cc$ and $\Dc$ then the extra data consisting of (iso)morphisms: \begin{equation}\label{monstr}f_{A,B}:F(A)\ot F(B)\to F(A\ot B),\end{equation} for all $A,B\in\Cc$ is the monoidal part of the structure.  For our purposes all monoidal functors are assumed to be strongly monoidal.   

\begin{lemma}\label{alg}
Let $F:\Cc\to\Dc$ be a monoidal functor and $A$ an algebra and $C$ a coalgebra in $\Cc$.  Then $F(A)$ and $F(C)$ are also an algebra and a coalgebra, respectively, in $\Dc$.  If $\alpha:A\to B$ is an algebra map, then so is $F(\alpha)$ (similarly for coalgebras).  If $M$ is an $A$-module, then $F(M)$ is an $F(A)$-module (similarly for coalgebras).
\end{lemma}
\begin{proof}
The proof is straightforward, we mention that the multiplication $m_{F(A)}$ on $F(A)$ is $F(m_A)\circ f_{A,A}$.
\end{proof}

Now suppose that $\Cc$ and $\Dc$ are both braided and let $\tau$ denote the braiding in both.  Furthermore, we use $\tau$ to denote the monoidal functor \begin{equation}\label{tauprod}
\tau:\Cc^{\boxtimes 2}\to \Cc
\end{equation} such that  $\tau(A\boxtimes B)=A\ot B$ and the monoidal structure isomorphism $$\tau(A\boxtimes B)\ot \tau(C\boxtimes D)\to \tau((A\boxtimes B)\ot(C\boxtimes D))$$ is $$Id_A\ot \tau_{B,C}\ot Id_D:A\ot B\ot C\ot D\to A\ot C\ot B\ot D.$$

\begin{definition}\label{algprod}
Let $A,B$ be algebras in $\Cc$ braided.  Then $$A\ot^\tau B=\tau(A\boxtimes B)$$  is an algebra in $\Cc$ by Lemma \ref{alg}. Explicitly, $$m_{A\ot^\tau B}=(m_A\ot m_B)(Id_A\ot \tau_{B,A}\ot Id_B),$$ or using strings:\begin{equation*}
\includegraphics[height=.7in]{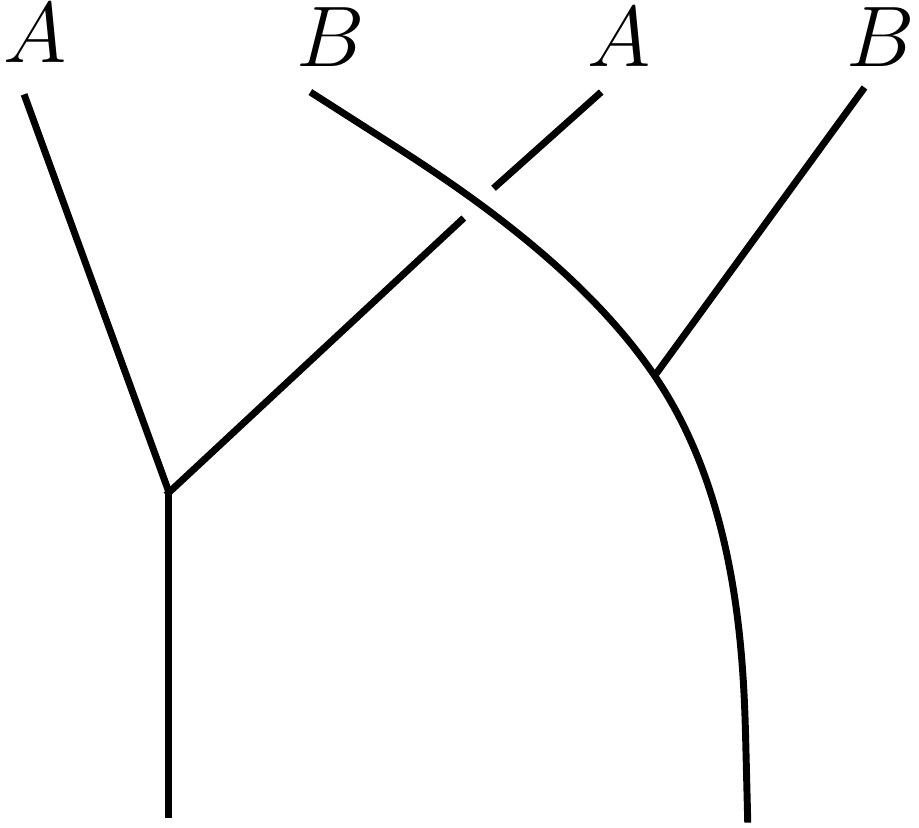}
\end{equation*} Furthermore, if $f:A\to A'$ and $g:B\to B'$ are algebra maps then so is $$f\ot g:A\ot^\tau B\to A'\ot^\tau B'.$$

Note that if  $M, N$ are $A$ and $B$-modules respectively then we have an $A\ot^\tau B$-module $M\ot^\tau N$ with $$\rho_{M\ot^\tau N}=(\rho_M\ot \rho_N)(Id_A\ot \tau_{B,M}\ot Id_N).$$  Similarly, if $M, N$ are right $A$ and $B$-modules respectively then we have a right $A\ot^\tau B$-module $M\ot^\tau N$ with $$\rho_{M\ot^\tau N}=(\rho_M\ot \rho_N)(Id_A\ot \tau_{N,A}\ot Id_N).$$
\end{definition}

\begin{remark}
Considering the reverse braiding $\tau^{-1}$,  let $A\ot^{\tau^{-1}} B=\tau^{-1}(A\boxtimes B)$, which is also an algebra in $\Cc$.
\end{remark}

Let $\Mc$ be a $\Cc$-module category. The following characterization of modules over $A\ot^{\tau} B$, i.e.,  $M\in A\ot^{\tau} B_\Mc\mod$ will be useful.  Observe that $M$ is an $A\ot^{\tau} B$-module if and only if it is an $A$-module ($\rho_A:A\cdot M\to M$) and a $B$-module ($\rho_B:B\cdot M\to M$) such that the two actions satisfy the compatibility condition: \begin{equation}\label{tensor}
\rho_A\circ(Id_A\ot\rho_B)=\rho_B\circ(Id_B\ot\rho_A)\circ(\tau^{-1}_{A,B}\ot Id_M)
\end{equation} best understood as an equality of string diagrams:

\begin{equation}\label{f1}
\includegraphics[height=.7in]{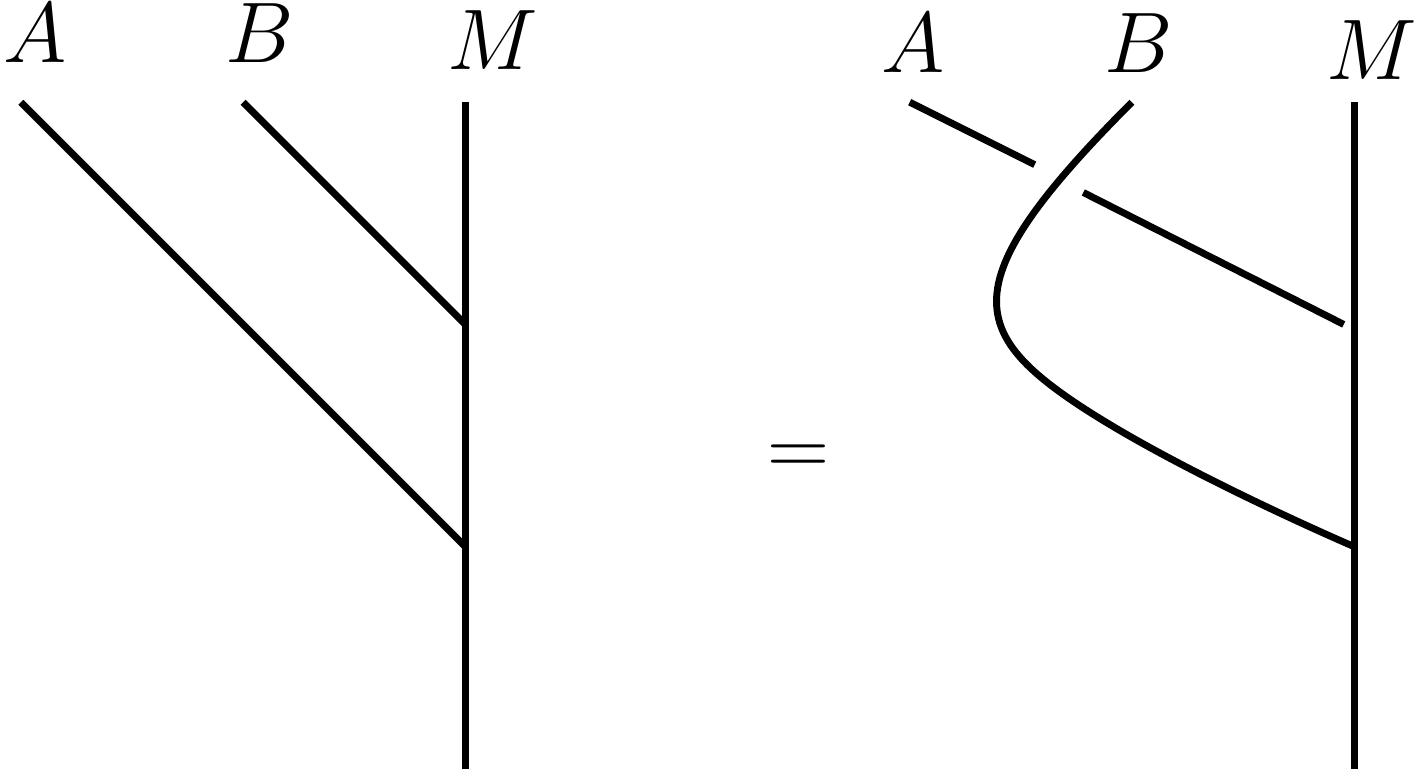}
\end{equation} so that the $A$-string crosses under the $B$-string.

\subsection{Some adjunctions}\label{sec:adjunction}

Let $\Cc$ be a closed monoidal category, namely we have right adjoints to the functors $-\ot M$ and $M\ot -$ (for all $M\in\Cc$) which we denote by $M\la -$ and $-\ra M$ respectively.  If $A$ is an algebra in $\Cc$ then $\bimod_\Cc A$, the category of $A$-bimodules in $\Cc$, is also a closed category.  The product is $-\ot_A-$, i.e., for $S,T\in \bimod_\Cc A$, we have that $S\ot_A T$ is the coequalizer of the two maps from $S\ot A\ot T$ to $S\ot T$.  We denote the right adjoints of $-\ot_A S$ and $S\ot_A -$ by $S\la_A -$ and $-\ra_A S$  respectively.  They can be constructed as equalizers, namely, $S\la_A T$ is the equalizer of $S\la T\to (S\ot A)\la T$ and $S\la T\to S\la(A\la T)$ where the targets are identified.  Note that this construction uses up the right actions on $S$ and $T$.  

More generally, let $A,B,C\in Alg(\Cc)$, let $\bimod_\Cc(A,B)$ denote the category of left $A$ and right $B$-modules in $\Cc$, and write $Hom_A(-,-)_B$ to denote the morphisms in $\bimod_\Cc(A,B)$.  Suppose that $M\in \bimod_\Cc(A,B)$, $S\in\bimod_\Cc(B,C)$, and $T\in\bimod_\Cc(A,C)$.  Then we have an adjunction \begin{equation}\label{adjunction}
Hom_A(M\ot_B S,T)_C\simeq Hom_B(S, T\ra_A M)_C.
\end{equation}  We also have $Hom_A(M\ot_B S,T)_C\simeq Hom_A(M, S\la_C T)_B$, but it is not what we need.

More generally, let $\Mc$ be a $\Cc$-module category.  But now suppose that $\Cc$ is rigid; then the right adjoint of $X\cdot -$ is ${}^*X\cdot -$.  Again let $A,B\in Alg(\Cc)$ and $M\in \bimod_\Cc(A,B)$.  For $S\in B_\Mc\mod$ and $T\in A_\Mc\mod$ we have \begin{equation}\label{adjunctionstar}
Hom_A(M\cdot_B S,T)\simeq Hom_B(S, {}^*M\cdot^A T),
\end{equation} where the coequalizer and equalizer are as above.

The following is immediate.

\begin{lemma}\label{hom:lem}
Let $A, B\in Alg(\Bc)$, with $\Bc$ a rigid braided category and $\Mc$ be a $\Bc$-module category.  With  $X\in\Bc$, let $A\ot X$ be an $(A,B)$-bimodule with $\rho_{A,A\ot X}=m_A\ot Id_X$.  Let $L, L'\in Alg(\Cc)$ so that $L\ot^\tau(A\ot X)\ot^\tau L'$ is a $(L\ot^\tau A\ot^\tau L',L\ot^\tau B\ot^\tau L')$-bimodule.  Then for $N\in L\ot^\tau A\ot^\tau L'_\Mc\mod$  we have $$N\ra_{L\ot^\tau A\ot^\tau L'}L\ot^\tau(A\ot X)\ot^\tau L'\simeq N\ra_A A\ot X\simeq {}^*X\cdot N$$ and the $L\ot^\tau B\ot^\tau L'$-module structure is as follows:
\begin{equation}\label{act5}
	\includegraphics[height=.9in]{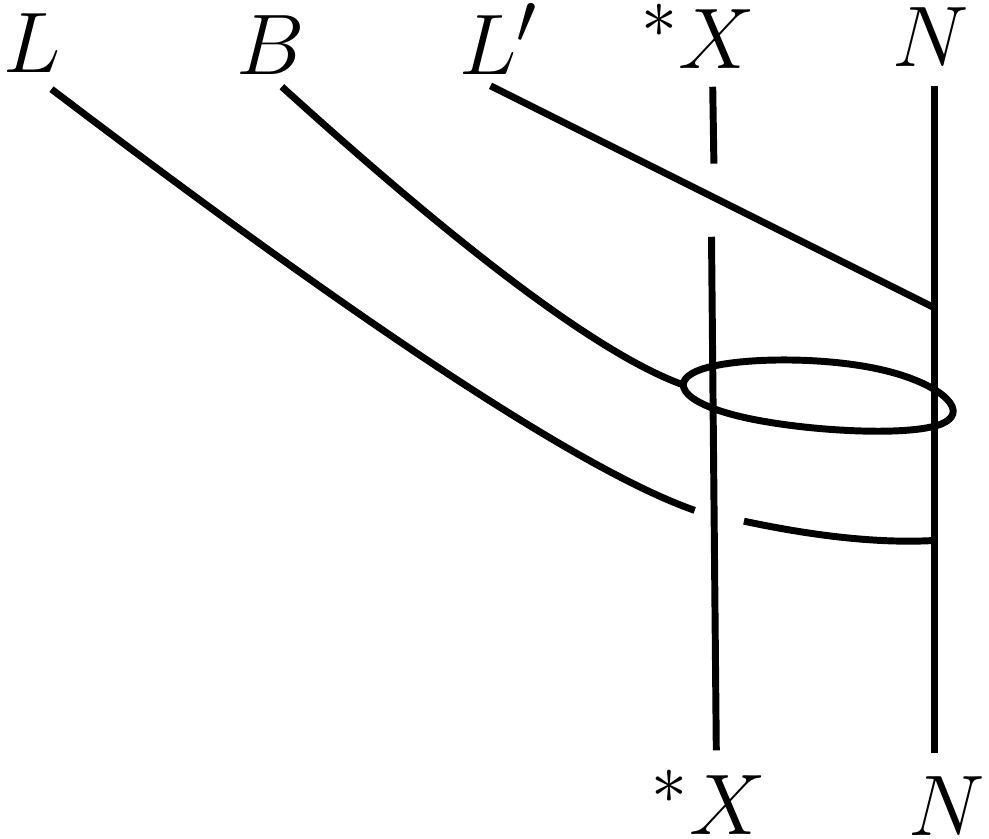} 
\end{equation}
where  \begin{equation}\label{act5a}
	\includegraphics[height=.9in]{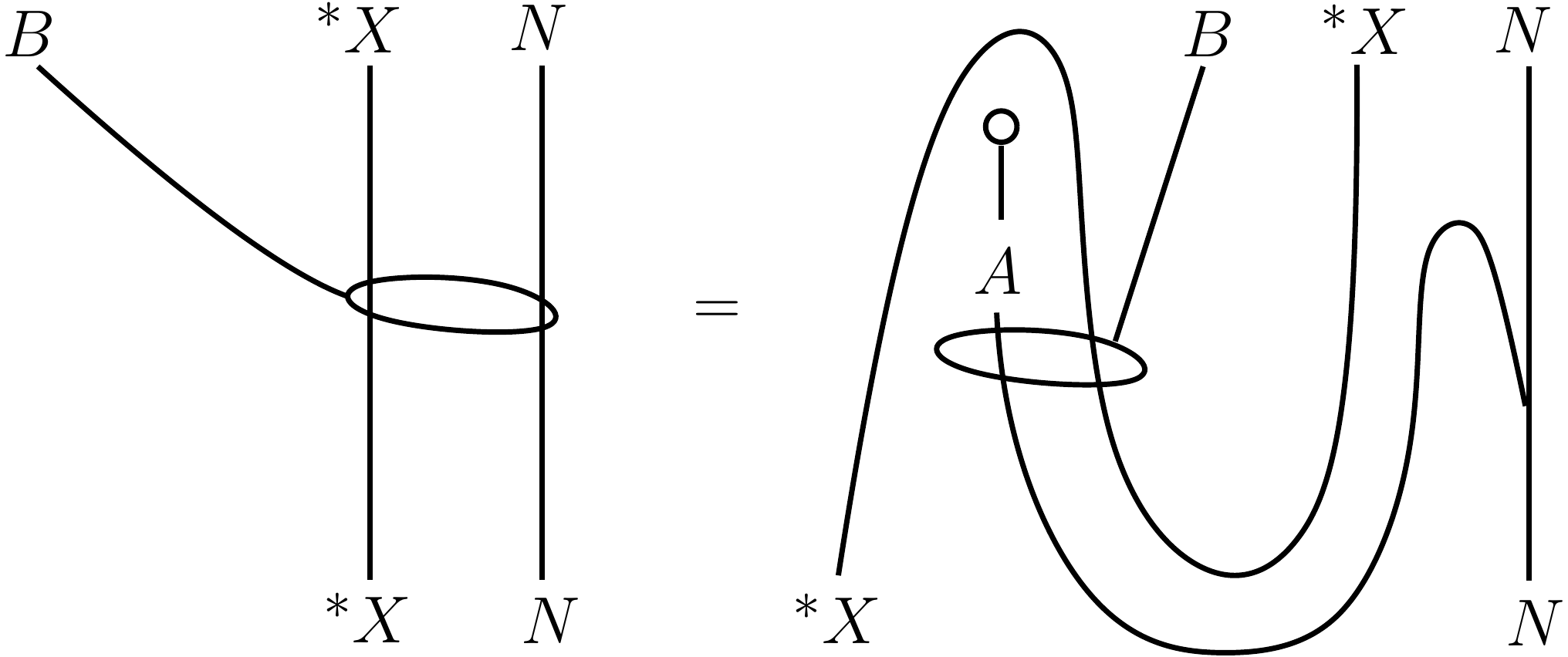} 
\end{equation}
\end{lemma}

\subsection{Inverse limits and Hochschild and cyclic homology categories}\label{lim:sec}
Let $D$ be a small category and suppose we have a diagram of categories over $D$.  Namely, for every object $x\in D$ we get a category  $\Cc_x$ and for every arrow $\alpha:x\to y$ in $D$ we get a functor $\alpha_*:\Cc_x\to\Cc_y$.  Furthermore, this data is complemented with associative isomorphisms $f_{\beta,\alpha}:\beta_*\alpha_*\to(\beta\alpha)_*$.  One may then consider the limit (inverse limit) category of this diagram.  Explicitly, the objects in the limit consist of the following data: objects $M_x\in\Cc_x$ for every $x\in D$ and isomorphisms $g_\alpha:\alpha_*(M_x)\to M_y$ for every $\alpha:x\to y$ in $D$.  These must satisfy the compatibility: $$g_{\beta\alpha}f_{\beta,\alpha}=g_\beta\beta_*(g_\alpha).$$

For us the important $D$'s are the simplex category $\Delta$ yielding the Hochschild homology and the Connes' cyclic category $\Lambda$ yielding the cyclic homology (see \cite{loday}).  More details are available in \cite{survey}.

\bibliography{taft}{}
\bibliographystyle{plain}
\bigskip

\noindent Department of Mathematics and Statistics,
University of Windsor, 401 Sunset Avenue, Windsor, Ontario N9B 3P4, Canada

\noindent\emph{E-mail address}:
\textbf{ishapiro@uwindsor.ca}

\end{document}